\documentclass[ejs,preprint]{imsart}

\arxiv{http://arxiv.org/abs/1506.08826}

\RequirePackage[round]{natbib}
\RequirePackage[colorlinks,citecolor=blue,urlcolor=blue, linkcolor=magenta]{hyperref}

\usepackage{bbm}
\usepackage{comment}
\usepackage{amsmath,  amssymb}
\usepackage{algorithm, algorithmicx}
\usepackage{graphicx}
\usepackage{subfigure}
\usepackage{lineno}
\usepackage[usenames,dvipsnames]{color}
\usepackage{multirow}
\usepackage{booktabs}

\let\hat\widehat
\newtheorem{thm}{Theorem}
\newtheorem{lem}[thm]{Lemma}
\newtheorem{cor}[thm]{Corollary}
\newtheorem{thm2}{Theorem}
\newtheorem{remark}[thm2]{Remark}

\newcommand\K{\mathbb{K}}
\newcommand\R{\mathbb{R}}
\newcommand\G{\mathbb{G}}
\newcommand\E{\mathbb{E}}
\renewcommand\P{\mathbb{P}}

\newcommand\norm[1]{\|#1\|}

\newcommand\Haus{{\sf Haus}}
\newcommand\Vol{{\sf Vol}}

\newcommand\rand{{\sf rand}}

\newcommand\cL{{\cal L}}
\newcommand\cX{{\cal X}}
\newcommand\cC{{\cal C}}
\newcommand\cD{{\cal D}}
\newcommand\cE{{\cal E}}

\newcommand\MS{{\sf MS}}
\newcommand\MSR{{\sf MSR}}
\newcommand\Leb{{\sf Leb}}

\newcommand{\argmin}{\mathop{\mathrm{argmin}}}

\DeclareMathOperator{\dest}{dest}
\catcode`@=11
\newskip\beforeproofvskip
\newskip\afterproofvskip
\beforeproofvskip=\medskipamount
\afterproofvskip=\bigskipamount

\def\prooftag{Proof}
\def\proofskip{\enspace}

\def\proof{\@ifnextchar[{\@@proof}{\@proof}}  %] for emacs matching
\def\@startproof{\par\vskip\beforeproofvskip\leavevmode}
\def\@proof{\@startproof{\scshape\prooftag.}\proofskip}
\def\@@proof[#1]{\@startproof {\scshape\prooftag #1.}\proofskip}

\catcode`@=12

\let\hat\widehat
\let\tilde\widetilde

\begin{document}

\begin{frontmatter}

\title{Statistical Inference Using the Morse-Smale Complex}%\protect\thanksref{T1}}
\runtitle{Inference using the Morse-Smale}
%\thankstext{T1}{Footnote to the title with the `thankstext' command.}

\begin{aug}
  \author{Yen-Chi Chen\ead[label=e1]{yenchic@uw.edu}},
  \and
  \author{Christopher R. Genovese\ead[label=e2]{genovese@stat.cmu.edu}},
  \and 
  \author{Larry Wasserman\ead[label=e3]{larry@stat.cmu.edu}}
  
    \address{University of Washington,\\
    Department of Statistics\\
  Box 354322,\\ 
  Seattle, WA 98195 \\
           \printead{e1}}

  \address{Carnegie Mellon University,\\
    Department of Statistics\\
      5000 Forbes Avenue,\\ 
  Pittsburgh, PA 15213 \\
           \printead{e2,e3}}

%  \thankstext{t2}{Footnote to the first author with the `thankstext' command.}

  \runauthor{Chen et al.}

\end{aug}

\begin{abstract}
The Morse-Smale complex of a function $f$
decomposes the sample space into cells
where $f$ is increasing or decreasing.
When applied to nonparametric density estimation and regression,
it provides a way to represent, visualize, and compare multivariate functions.
In this paper, 
we present some statistical results on estimating Morse-Smale complexes.
This allows us to derive new results for two existing methods:
mode clustering and Morse-Smale regression.
We also develop two new methods based on the Morse-Smale complex:
a visualization technique for multivariate functions
and a two-sample, multivariate hypothesis test.
\end{abstract}



%\tableofcontents

\end{frontmatter}

\section{Introduction}
\label{sec::intro}

%{\bf YEN-CHI: thanks the referees somewhere}

Let $f$ be a smooth, real-valued function defined on a compact set $\K\in \R^d$.
In this paper, $f$ will be a regression function or a density function.
The Morse-Smale complex of $f$ is a partition of $\K$ based on the gradient flow induced by $f$.
Roughly speaking, the complex consists of sets, called \emph{crystals} or \emph{cells},
comprised of regions where $f$ is increasing or decreasing.
Figure \ref{Fig::ex_MS} shows the Morse-Smale complex for a two-dimensional function.
The cells are the intersections of the
basins of attractions (under the gradient flow) of the function's maxima and minima.
The function $f$ is piecewise monotonic over cells with respect to some directions.
In a sense, the Morse-Smale complex provides a generalization of isotonic regression.

Because the Morse-Smale complex represents a multivariate function in terms
of regions on which the function has simple behavior,
the Morse-Smale complex has useful applications in statistics, 
including in clustering, regression, testing, and visualization.
For instance,
when $f$ is a density function, the basins of attraction of $f$'s modes
are the (population) clusters for density-mode clustering
(also known as mean shift clustering \citep{fukunaga1975estimation, chacon2015population}),
each of which is a union of cells from the Morse-Smale complex.
Similarly,
when $f$ is a regression function, the cells of the Morse-Smale complex
give regions on which $f$ has simple behavior.
Fitting $f$ over the Morse-Smale cells provides a generalization of
nonparametric, isotone regression;
\cite{gerber2013morse} proposes such a method.
The Morse-Smale representation of a multivariate function $f$
is a useful tool for visualizing $f$'s structure,
as shown by \cite{gerber2010visual}.
In addition, suppose we want
to compare two multi-dimensional datasets
$X=(X_1,\ldots, X_n)$
and $Y=(Y_1,\ldots, Y_m)$.
We start by forming the Morse-Smale complex
of $\hat p-\hat q$
where
$\hat p$ is density estimate from $X$ and
$\hat q$ is density estimate from $Y$.
Figure \ref{Fig::ex0} shows a visualization built from this complex.
The circles represent cells of the Morse-Smale complex.
Attached to each cell is a pie-chart showing what fraction of the cell
has $\hat p$ significantly larger than $\hat q$.
This visualization is a multi-dimensional extension of the method proposed 
for two or three dimensions in \cite{duong2013local}.

\begin{figure}
\center
\subfigure[Descending manifold]{
	\includegraphics[trim=0in 0in 0in 2.5in, clip,width=2.2in]{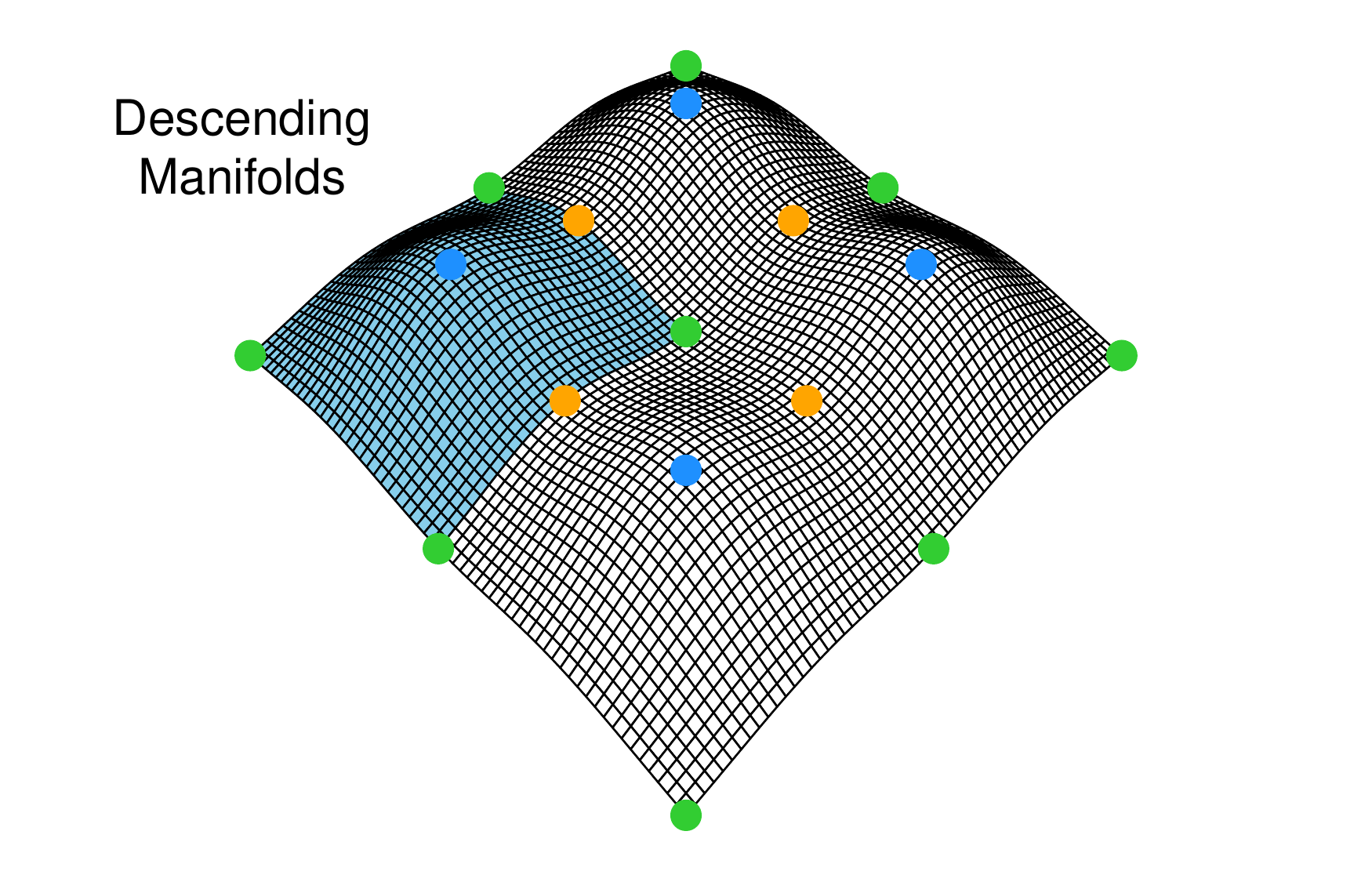}
}
\subfigure[Ascending manifold]{
	\includegraphics[trim=0in 0in 0in 2.5in,clip, width=2.2in]{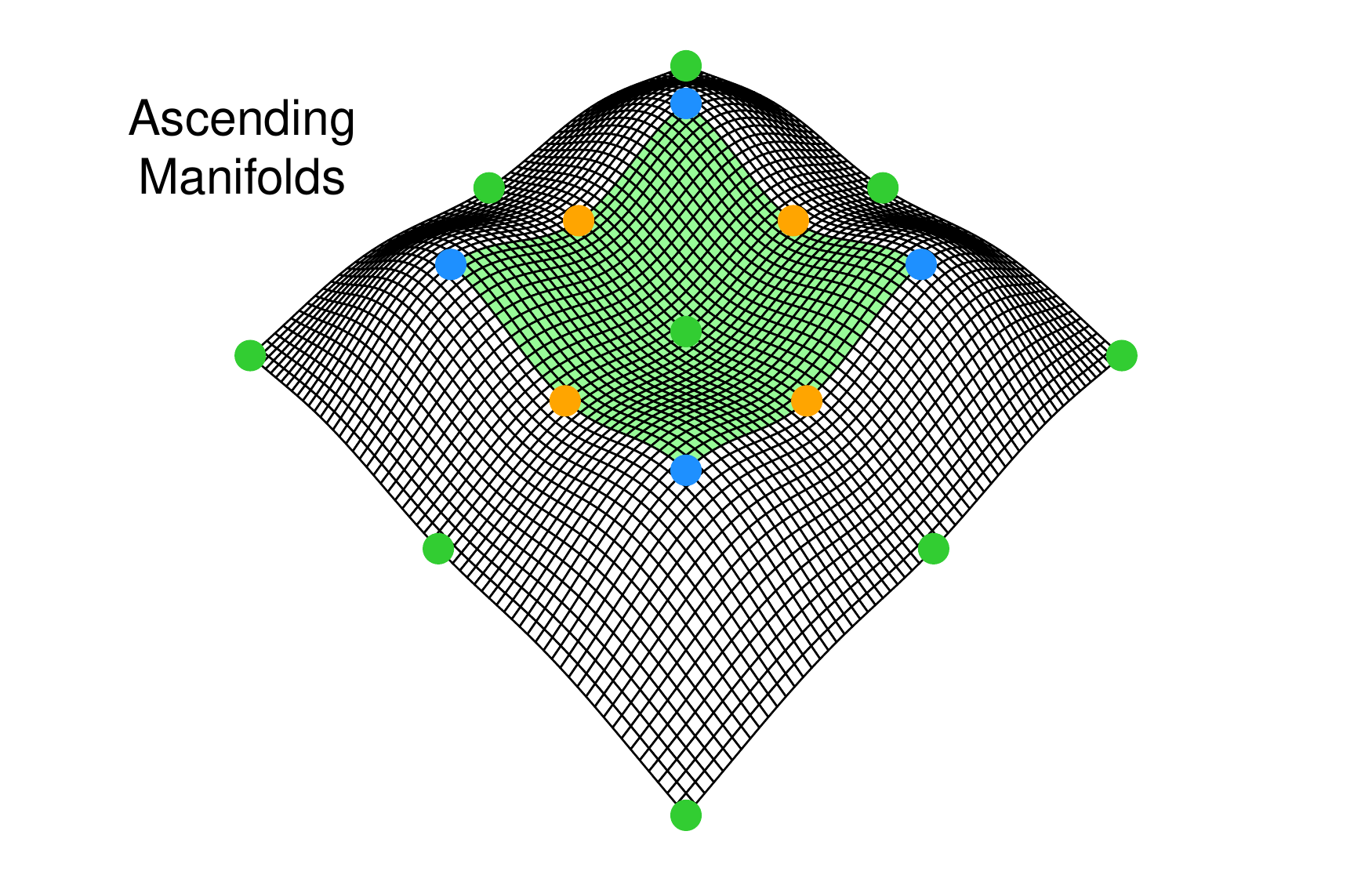}
}
\subfigure[$d$-cell]{
	\includegraphics[trim=0in 0in 0in 2.5in,clip, width=2.2in]{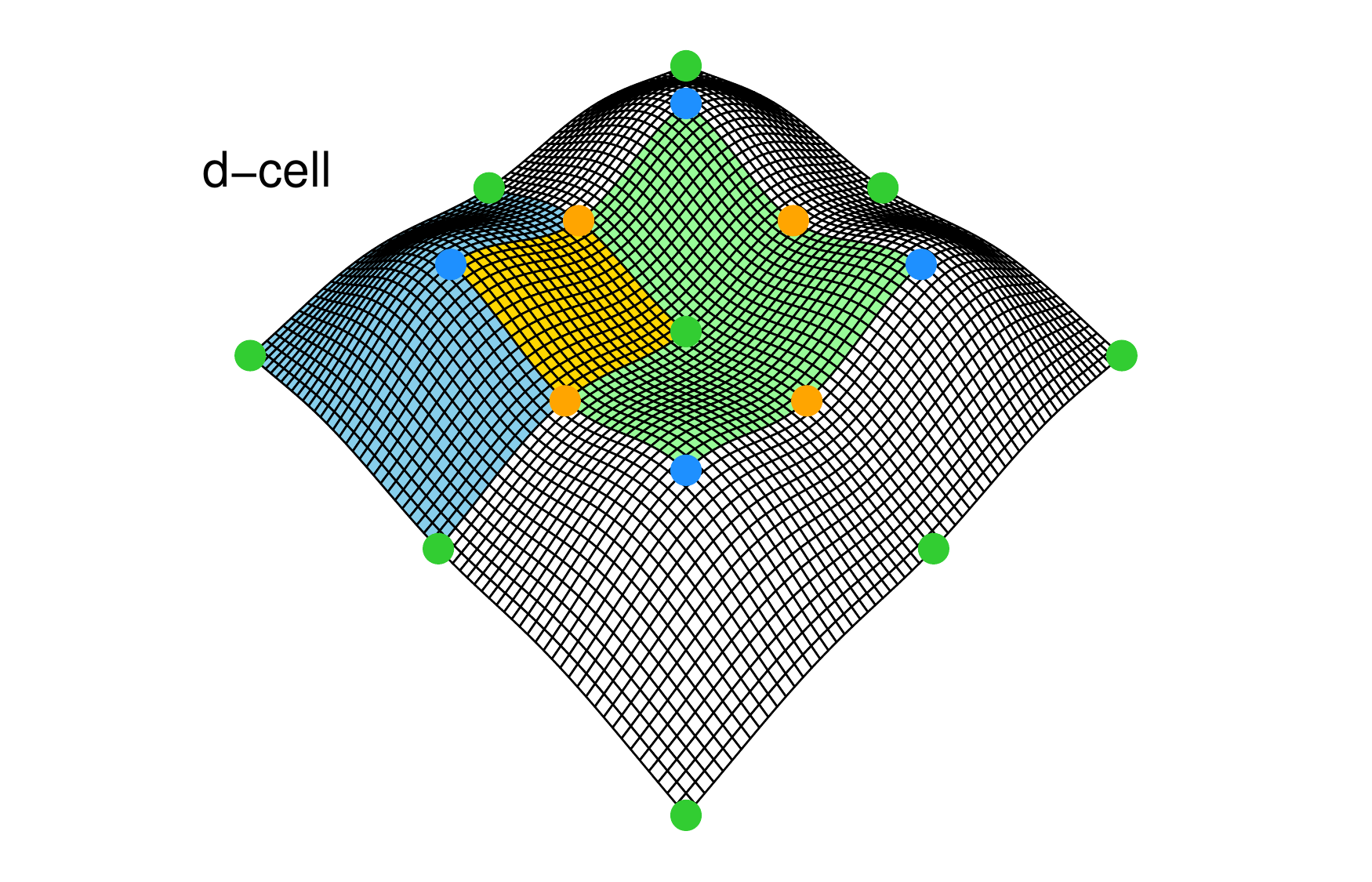}
}
\subfigure[Morse-Smale complex]{
	\includegraphics[trim=0in 0in 0in 2.5in,clip, width=2.2in]{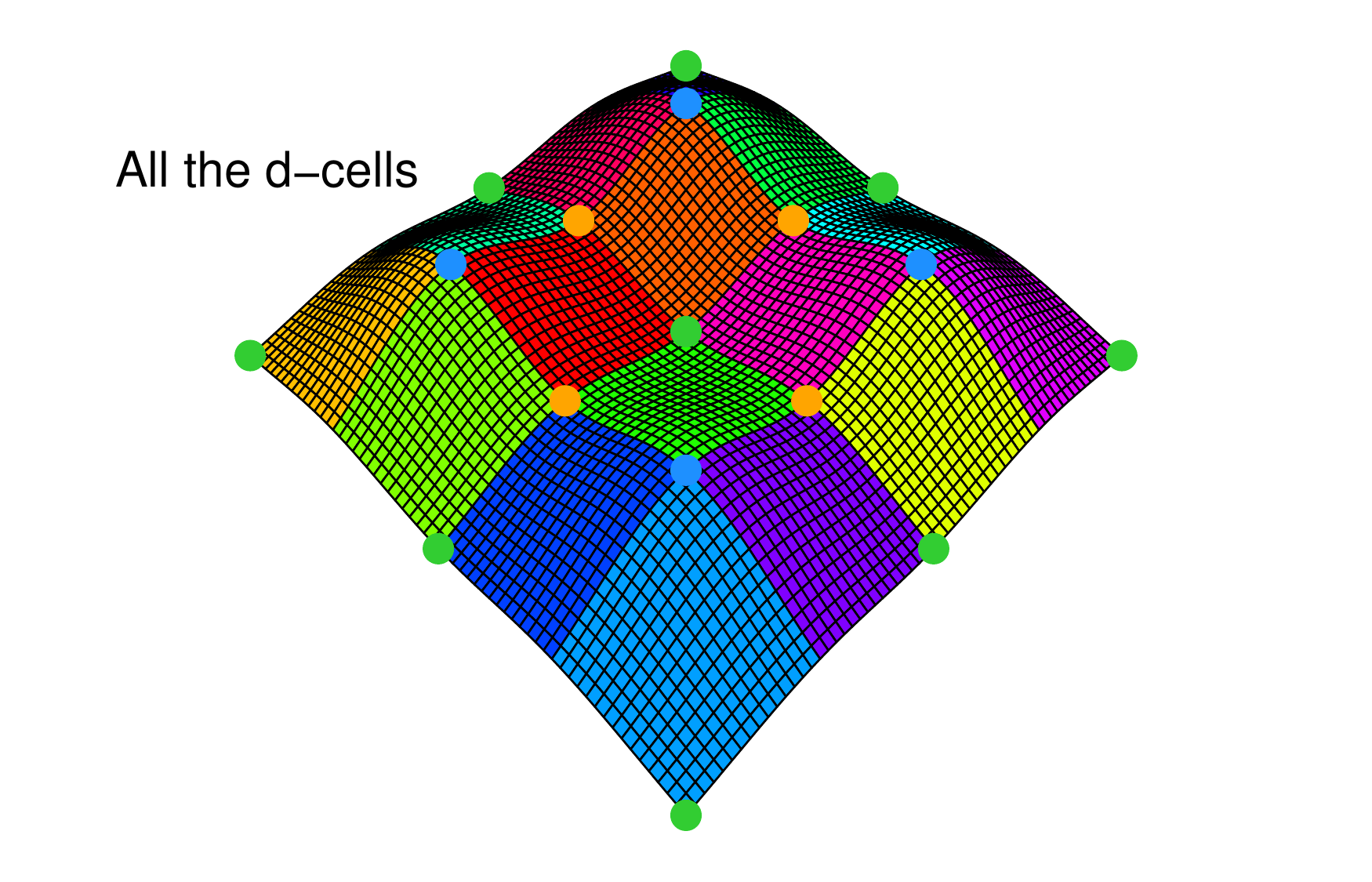}
}
\caption{An example of a Morse-Smale complex.
The green dots are local minima; the blue dots are local modes; the 
violet dots are saddle points.
Panels (a) and (b) give examples of descending $d$-manifolds (blue region)
and an ascending $0$-manifold (green region).
Panel (c) shows the corresponding $d$-cell (yellow region).
Panel (d) is shows all $d$-cells.
}
\label{Fig::ex_MS}
\end{figure}

\begin{figure}
\begin{center}
\includegraphics[scale=0.3]{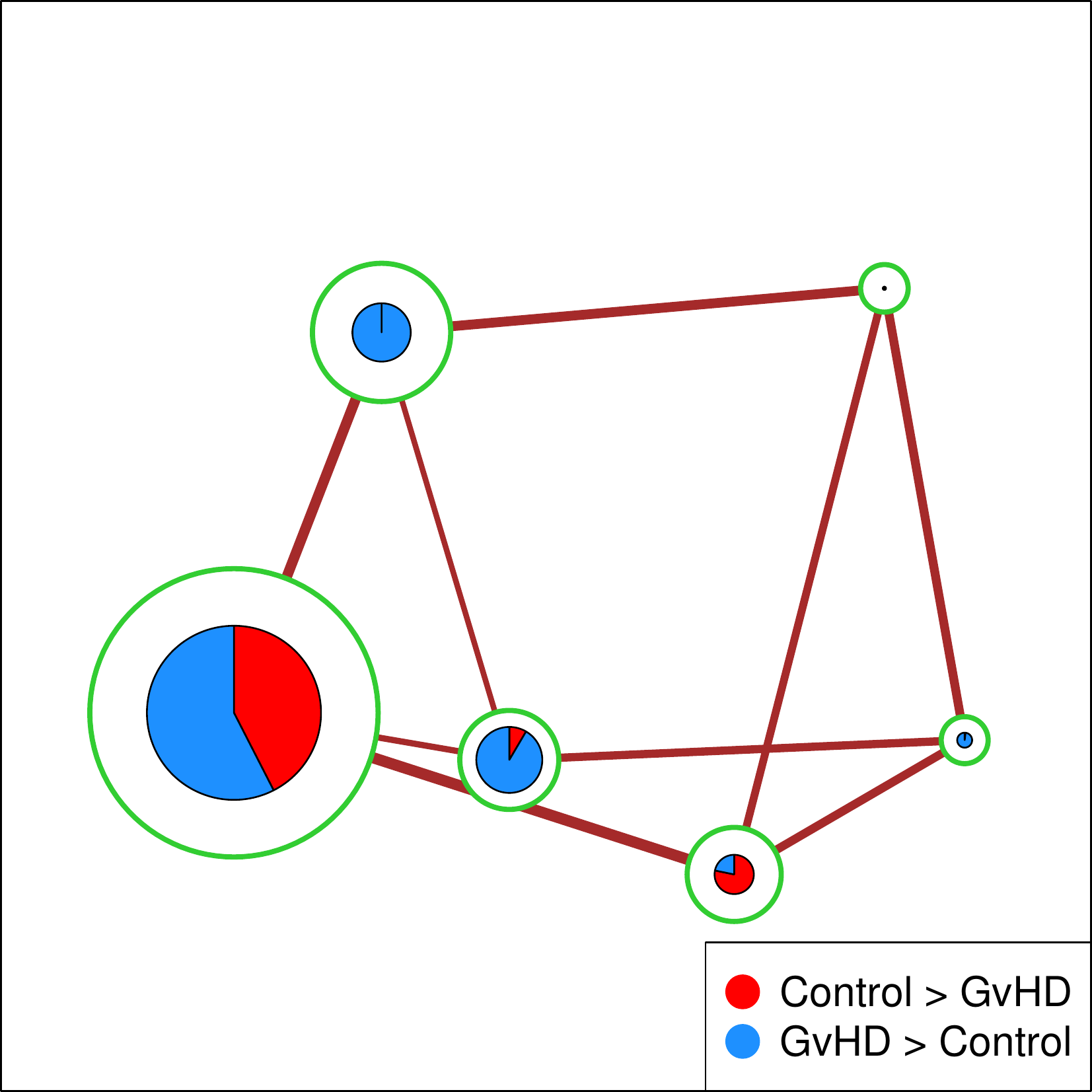}
\end{center}
\caption{Graft-versus-Host Disease (GvHD) dataset \citep{brinkman2007high}.
This is a $d=4$ dimensional dataset.
We estimate the density difference based on the kernel density estimator 
and find regions where the two densities are significantly different.
Then we visualize the density difference using the Morse-Smale complex.
Each green circle denotes a $d$-cell, which is a partition for the support $\K$.
The size of circle is proportional to the size of cell.
If two cells are neighborhors,
we add a line connecting them;
the thickness of the line denotes the amount of boundary they share.
The pie charts show the ratio of the regions within each cell 
where the two densities are significantly 
different from each other.
See Section \ref{sec::two} for more details.}
\label{Fig::ex0}
\end{figure}

For all these applications, the Morse-Smale complex needs to be estimated.
To the best of our knowledge, no theory has been developed for this estimation problem,
prior to this paper.
We have three goals in this paper:
to show that many existing problems can be cast in terms of the Morse-Smale complex,
to develop some new statistical methods based on the Morse-Smale complex, 
and to develop the statistical theory for estimating the complex.

\vspace{1cm}

\noindent
\emph{Main results.}
The main results of this paper are:
\begin{enumerate}
\item \emph{Consistency of the Morse-Smale Complex}.
We prove the stability of the Morse-Smale complex (Theorem~\ref{thm::Haus}) in the following sense:
if $B$ and $\tilde{B}$ are boundaries of the descending $d$-manifolds (or ascending $0$-manifolds) of $p$ and $\tilde{p}$
(defined in Section \ref{sec::morse}), then
$$
\Haus(B,\tilde{B}) = O\left(\|\nabla p -\nabla \tilde{p}\|_\infty\right).
$$

\item \emph{Risk Bound for Mode clustering (mean-shift clustering; section~\ref{sec::mode})}:
  We bound the risk of mode clustering in Theorem \ref{thm::mode}.

\item \emph{Morse-Smale regression (section~\ref{sec::MSR})}:
  In Theorems \ref{thm::MSR} and \ref{thm::MSR2},
  we bound the risk of Morse-Smale regression,
  a multivariate regression method proposed in \cite{gerber2010visual, gerber2011data, gerber2013morse}
  that synthesizes nonparametric regression and linear regression.

\item \emph{Morse-Smale signatures (section~\ref{sec::MSS})}: 
  We introduce a new visualization method for densities and regression functions.

\item \emph{Morse-Smale two-sample testing (section~\ref{sec::two})}: 
  We develop a new method for multivariate two-sample testing
  that can have good power.

\end{enumerate}

\emph{Related work.}
The mathematical foundations for the Morse-Smale complex
are from Morse theory
\citep{morse1925relations, morse1930foundations, milnor1963morse}.
Morse theory has many applications
including computer vision \citep{paris2007topological}, 
computational geometry \citep{cohen2007stability} and
topological data analysis \citep{chazal2014robust}.

Previous work on the stability of the Morse-Smale complex
can be found in \cite{chen2016comprehensive} and \cite{chazal2014robust}
but they only consider critical points rather than the whole Morse-Smale complex.
\cite{arias2016estimation} prove pointwise convergence
for the gradient ascent curves but this is not sufficient for
proving the stability of the complex because the convergence of complexes
requires convergence of multiple curves
and the constants in the convergence rate derived from \cite{arias2016estimation} vary from points to points
and some constants diverge when we are getting closer to the boundaries of complexes.
Thus, we cannot obtain a uniform convergence of gradient ascent curves directly based on their results.
%{\bf YEN-CHI: this last sentence is not clear. Can you say more precisely why their results do not imply
%convergence of the complexes?}
Morse-Smale regression and visualization were
proposed in
\cite{gerber2010visual, gerber2011data, gerber2013morse}.

The R code (Algorithm \ref{Alg::MSS}, \ref{Alg::vis}, and \ref{Alg::two})
used in this paper can be found at
\url{https://github.com/yenchic/Morse_Smale}.

\section{Morse Theory}	\label{sec::morse}

\begin{figure}
\center
	\includegraphics[width=1.5 in]{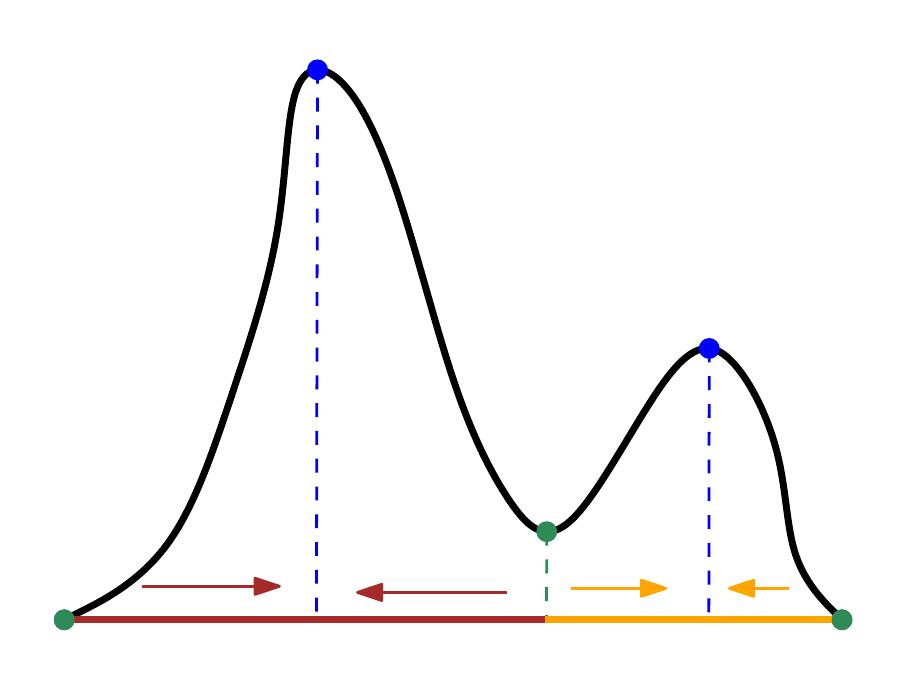}
	\includegraphics[width=1.5 in]{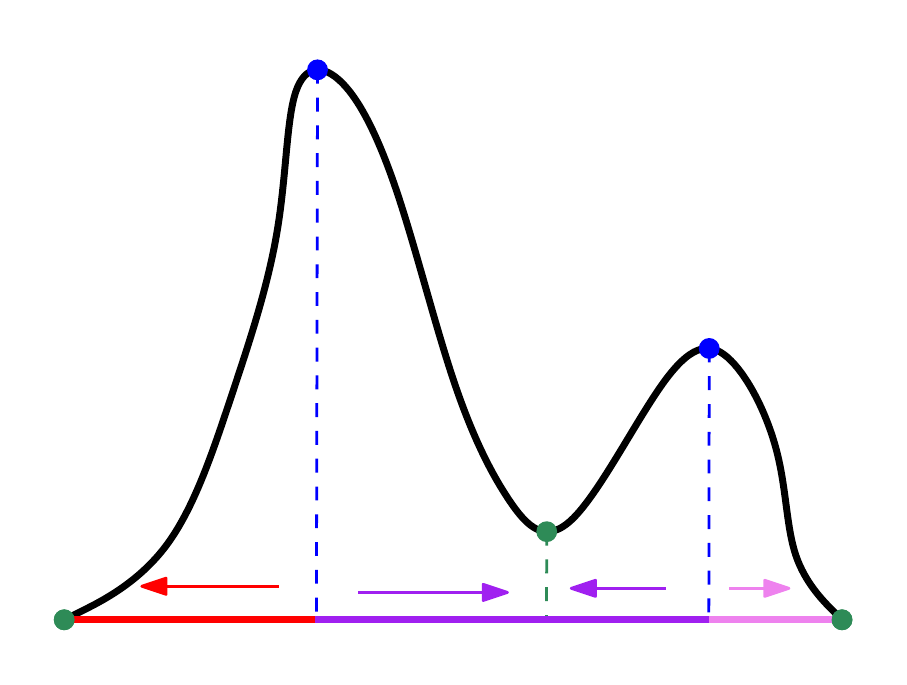}
	\includegraphics[width=1.5 in]{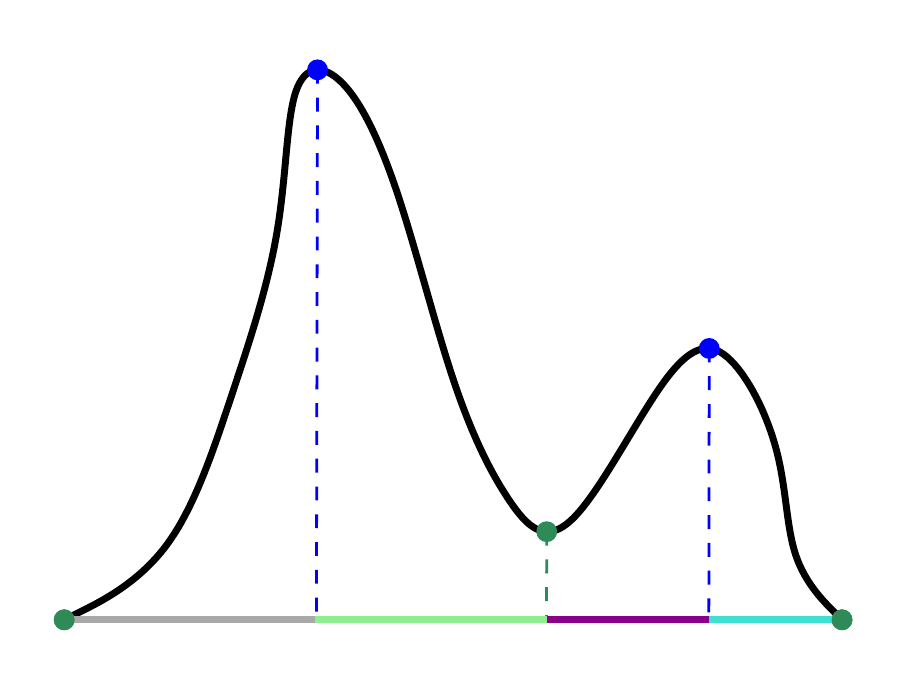}
\caption{A one dimensional example.
The blue dots are local modes and the green dots are local minima.
Left panel: the basins of attraction for two local modes
are colored by brown and orange.
Middle panel:
the basin of attraction (negative gradient)
for the local minima are colored by red, purple and violet.
Right panel:
The intersection of the basins, which are called $d$-cells.
}
\label{Fig::ex_1d}
\end{figure}

To motivate formal definitions,
we start with the simple, one-dimensional example depicted in Figure \ref{Fig::ex_1d}.
The left panel shows the sets associated with each local maximum
(i.e. the basins of attraction of the maxima).
The middle panel shows the sets associated with each local minimum.
The right panel show the intersections of these basins,
which gives the Morse-Smale complex defined by the function.
Each interval in the complex, called a cell,
is a region where the function is increasing or decreasing.

Now we give a formal definition.
Let $f:\mathbb{K}\subset\mathbb{R}^d\mapsto \mathbb{R}$
be a function with bounded third derivatives that is defined on a compact set $\K$.
Let $g(x)= \nabla f(x)$ and $H(x) = \nabla \nabla f(x)$ be the gradient and Hessian matrix
of $f$, respectively, and let $\lambda_j(x)$ be the $j$th largest eigenvalue of $H(x)$.
Define $\cC = \{x\in\mathbb{K}: g(x)=0\}$ to be the set of all $f$'s critical points,
which we call the \emph{critical set}.
Using the signs of the eigenvalues of the Hessian, the critical set $\cC$ can be partitioned
into $d+1$ distinct subsets $C_0,\cdots,C_d$,
where 
\begin{equation}
C_k = \{x\in\mathbb{K}: g(x)=0, \lambda_k(x)>0, \lambda_{k+1}(x)<0\}, \quad k=1,\cdots, d-1.
\end{equation}
We define $C_0, C_d$ to be the sets of all local maxima and minima
(corresponding to all eigenvalues being negative and positive respectively).
The set $C_k$ is called $k-$th order critical set.

A smooth function $f$ is called a \emph{Morse function}
\citep{morse1925relations, milnor1963morse}
if its Hessian matrix is non-degenerate at each critical point.
That is,
$|\lambda_j(x)|>0, \forall x\in \cC$ for all $j$.
In what follows we assume $f$ is a Morse function 
(actually, later we will assume further that $f$ is 
a Morse-Smale function).
%{\bf need to the transversality property to get a Morse-Smale function.}

\begin{figure}
\center
\subfigure[]{
	\includegraphics[trim=1in 0in 1in 2.5in, clip,width=2.2in]{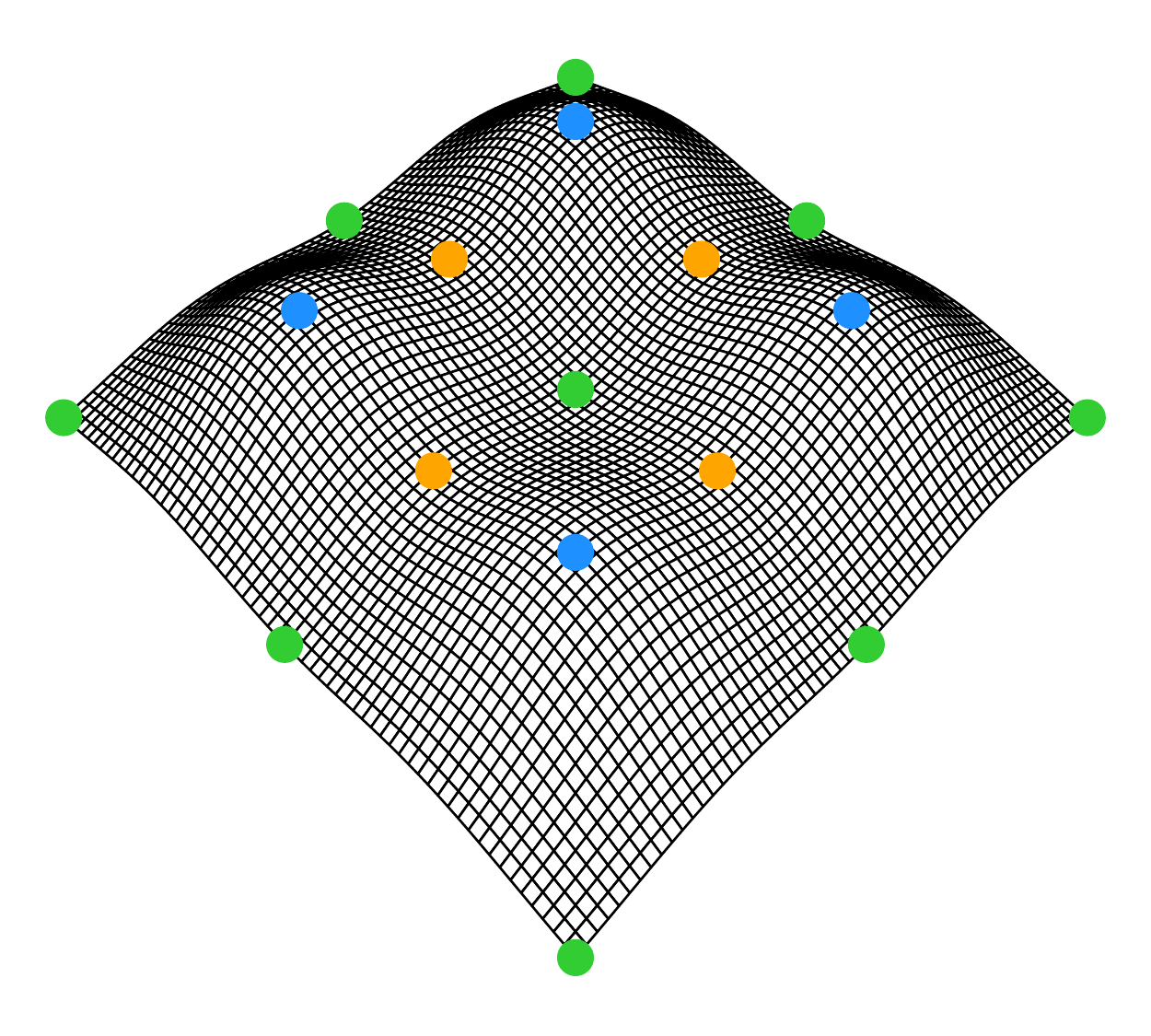}
}
\subfigure[]{
	\includegraphics[trim=1in 0in 1in 2.5in, clip,width=2.2in]{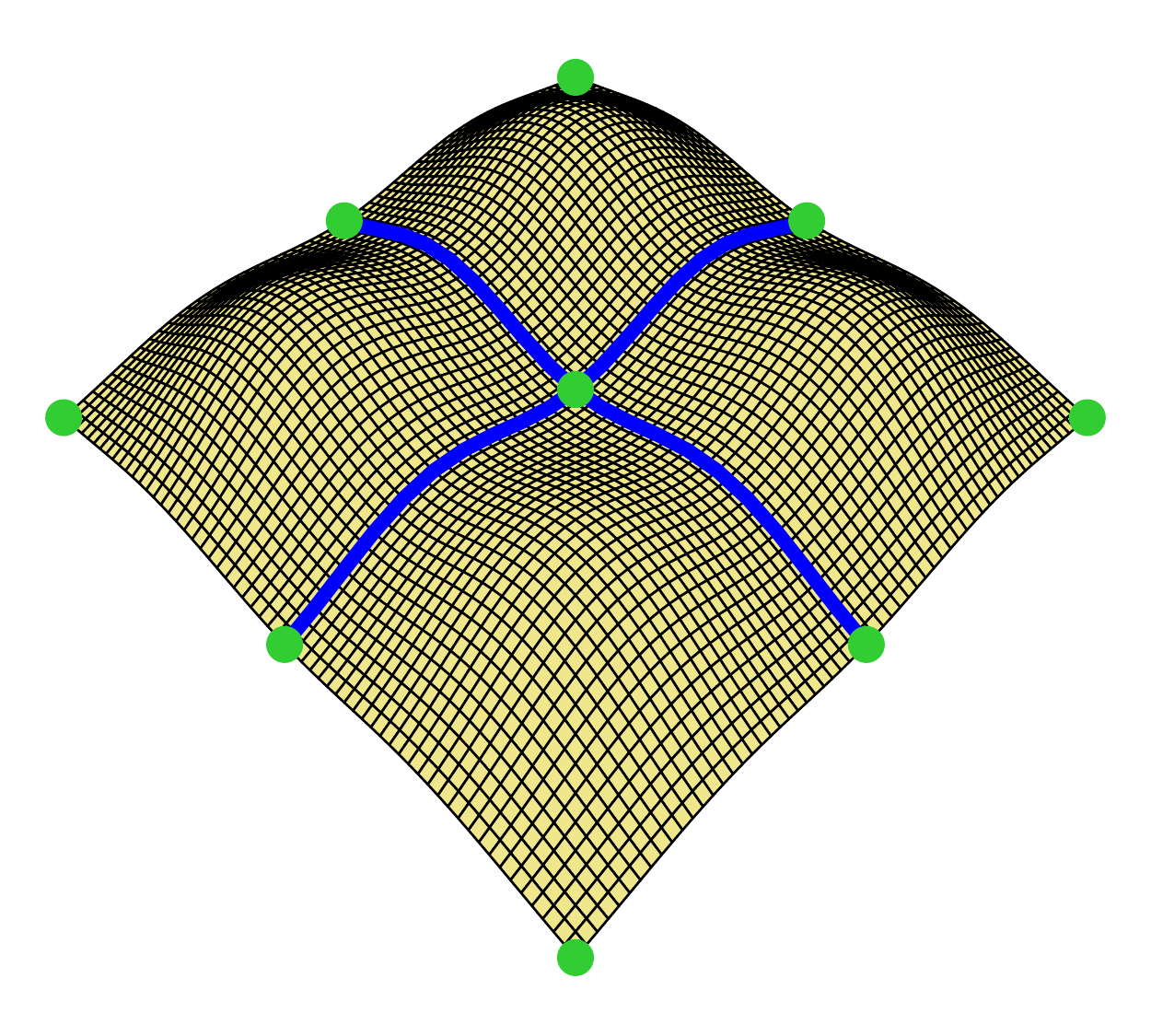}
}
\subfigure[]{
	\includegraphics[trim=1in 0in 1in 2.5in, clip,width=2.2in]{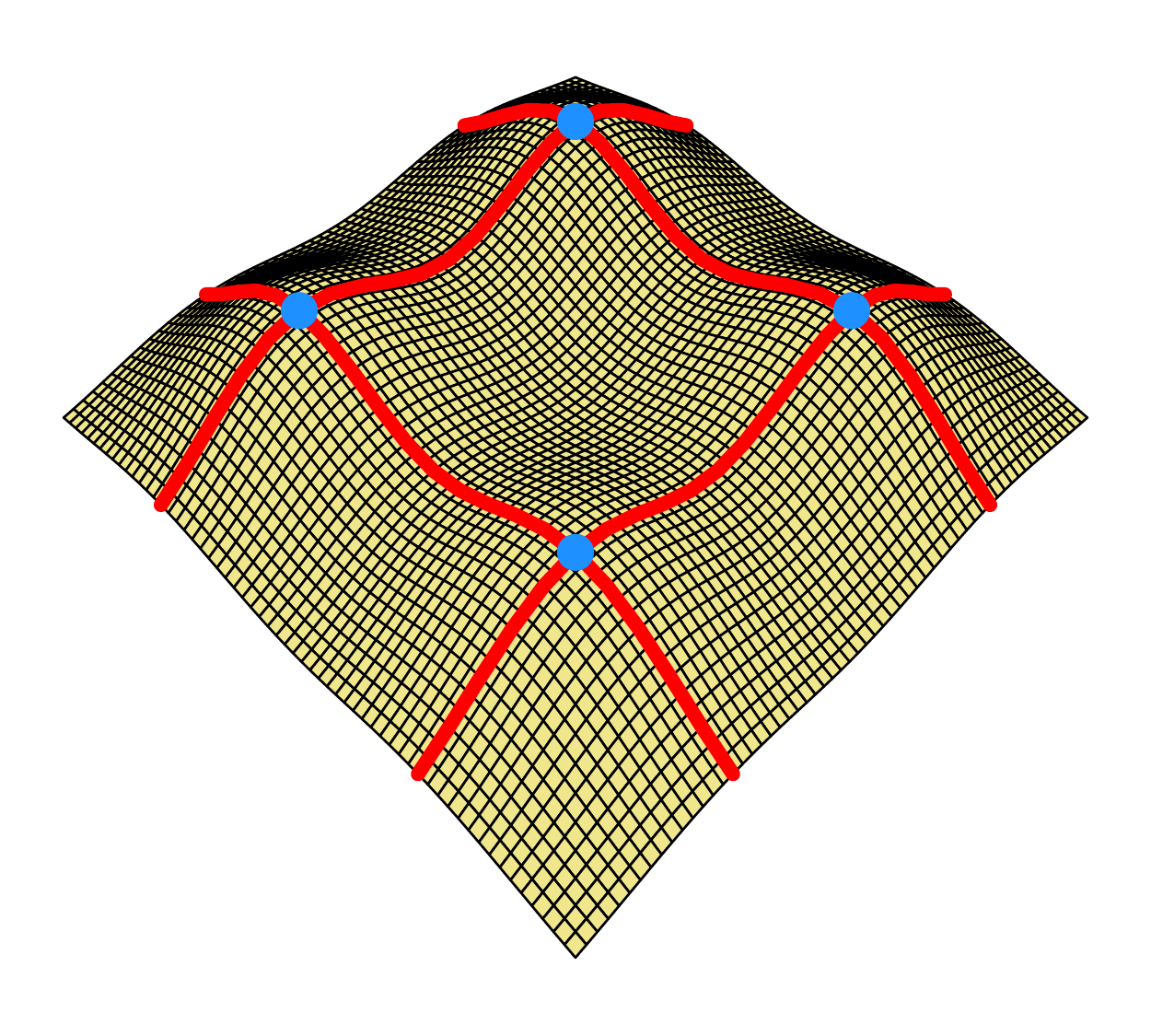}
}
\subfigure[]{
	\includegraphics[trim=1in 0in 1in 2.5in, clip,width=2.2in]{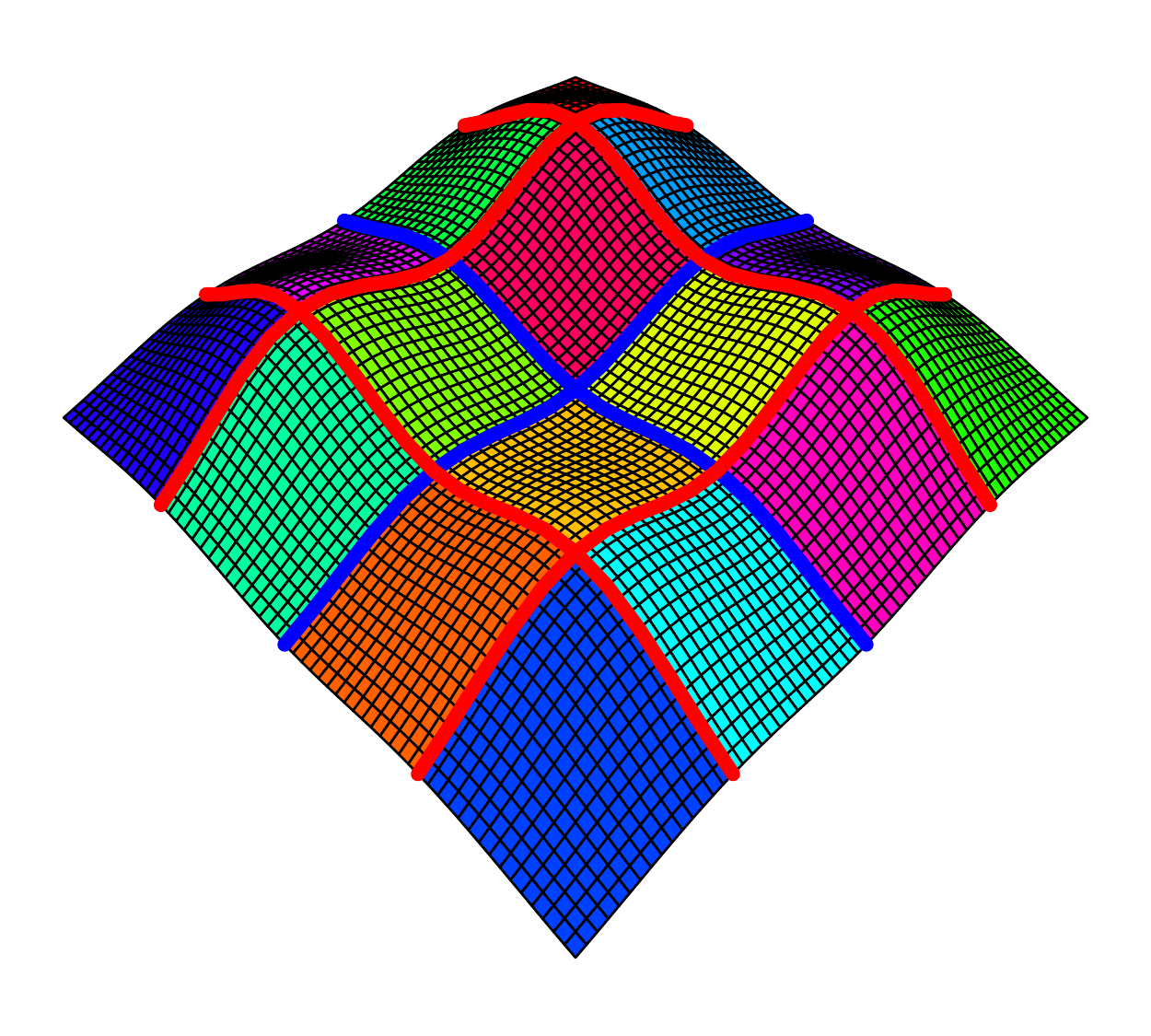}
}

\caption{Two-dimensional examples of critical points, descending manifolds, ascending manifolds, and $2$-cells.
This is the same function as Figure~\ref{Fig::ex_MS}.
(a): The set $C_k$ for $k=0,1,2$.
The four blue dots are $C_0$, the collection of local modes (each of them is $c_{0,j}$ some $j=1,\cdots, 4$).
The four orange dots are $C_1$, the collection of saddle points (each of them is $c_{1,j}$ for some $j=1,\cdots, 4$).
The green dots are $C_2$, the collection of local minima (each green dot is $c_{2,j}$ for some $j=1,\cdots,9$). 
(b): The set $D_k$ for $k=0,1,2$.
The yellow area is $D_2$ (each subregion separated by blue curves are $D_{2,j}, j=1,\cdots, 4$).
The two blue curves are $D_1$ (each of the 4 blue segments are $D_{1,j}, j=1,\cdots, 4$).
The green dots are $D_0$ (also $C_2$), the collection of local minima (each green dot is $D_{0,j}$ for some $j=1,\cdots,9$). 
(b): The set $A_k$ for $k=0,1,2$.
The yellow area is $A_0$ (each subregion separated by red curves are $A_{0,j}, j=1,\cdots, 9$).
The two red curves are $A_1$ (each of the 4 red segments are $A_{1,j}, j=1,\cdots, 4$).
The blue dots are $A_2$ (also $C_0$), the collection of local modes (each green dot is $A_{0,j}$ for some $j=1,\cdots,4$). 
(d): Example for $2$-cells. The thick blue curves are $D_1$ and 
thick red curves are $A_1$. 
}
\label{Fig::ex_D}
\end{figure}

Given any point $x\in\K$, we define the gradient ascent flow starting at $x$, $\pi_x: \mathbb{R}^+ \mapsto \K$,
by
\begin{equation}
\begin{aligned}
\pi_x(0)& = x\\
\pi'_x(t) & = g(\pi(t)).
\end{aligned}
\label{eq::GS}
\end{equation}
A particle on this flow moves along the gradient from $x$
towards a ``destination'' given by
$$
\dest(x) \equiv \lim_{t\rightarrow \infty} \pi_x(t).
$$
It can be shown that $\dest(x) \in \cC$ for $x\in\K$.

We can thus partition $\K$ based on the value of $\dest(x)$.
These partitions are called \emph{descending manifolds} in Morse theory 
\citep{morse1925relations, milnor1963morse}.
Recall $C_k$ is the $k$-th order critical points, 
we assume $C_k=\{c_{k,1},\cdots,c_{k,m_k}\}$ contains $m_k$ distinct elements.
For each $k$, define
\begin{equation}
\begin{aligned}
D_k &= \left\{x: \dest(x)\in C_{d-k}\right\}\\
D_{k,j} & = \left\{x: \dest(x)= c_{d-k,j}\right\}, \quad j=1,\cdots m_{d-k}.
\end{aligned}
\end{equation}
That is, $D_k$ is the collection of all points whose gradient ascent 
flow converges to a $(d-k)$-th order critical point
and $D_{k,j}$ is the collection of points whose gradient ascent flow
converges to the $j$-th element of $C_{d-k}$.
Thus, $D_k = \bigcup_{j=1}^{m_{d-k}}D_{k,j}$.
From
Theorem 4.2 in \cite{banyaga2004lectures},
each $D_k$ is a disjoint union of $k$-dimensional manifolds
($D_{k,j}$ is a $k$-dimensional manifold).
We call $D_{k,j}$ a
\emph{descending k-manifold} of $f$.
Each descending k-manifold is a $k$-dimensional manifold
such that the gradient flow from every point converges to the same $(d-k)$-th order critical point.
Note that $\{D_0,\cdots,D_k\}$ forms a partition of $\K$.
The top panels of Figure~\ref{Fig::ex_D} give an example of the descending manifolds for a two dimensional case.

The \emph{ascending manifolds} are similar to descending manifolds
but are defined through the gradient descent flow.
More precisely, given any $x\in \K$, a gradient descent flow $\gamma: \mathbb{R}^+ \mapsto \K$
starting from $x$ is given by
\begin{equation}
\begin{aligned}
\gamma_x(0)& = x\\
\gamma'_x(t) & = -g(\pi(t)).
\end{aligned}
\label{eq::GD}
\end{equation}
Unlike the ascending flow defined in \eqref{eq::GS}, $\gamma_x$ 
is a flow that moves along the gradient descent direction.
The descent flow $\gamma_x$ shares similar properties to the ascent flow $\pi_x$;
the limiting point $\lim_{t\rightarrow\infty}\gamma_x(t) \in \cC$ is also in critical set
when $f$ is a Morse function.
Thus, similarly to $D_k$ and $D_{k,j}$, we define
\begin{equation}
\begin{aligned}
A_k &= \left\{x: \lim_{t\rightarrow \infty}\gamma_x(t)\in C_{d-k}\right\}\\
A_{k,j} &= \left\{x: \lim_{t\rightarrow \infty}\gamma_x(t)= c_{d-k,j}\right\}, \quad j=1,\cdots, m_{j-k}.\\
\end{aligned}
\end{equation}
$A_k$ and $A_{k,j}$ %share similar properties as $D_k$ and $D_{k,j}$:
have dimension $d-k$ and each $A_{k,j}$ is a partition for $A_k$
and $\{A_0,\cdots,A_d\}$ consist of a partition for $\K$.
We call each $A_{k,j}$ an \emph{ascending k-manifold} to $f$.

A smooth function $f$ is called a \emph{Morse-Smale function} if it is a Morse function
and any pair of the ascending and descending manifolds of $f$ intersect
each other transversely (which means that pairs of manifolds are not parallel
at their intersections); 
see e.g. \cite{banyaga2004lectures} for more details.
In this paper, we also assume that $f$ is a Morse-Smale function.
Note that by the Kupka-Smale Theorem (see e.g. Theorem 6.6 in \cite{banyaga2004lectures}),
Morse-Smale functions are generic (dense) in the collection of smooth functions.
For more details, we refer to Section 6.1 in \cite{banyaga2004lectures}.

A \emph{k-cell} (also called Morse-Smale cell or crystal)
is the non-empty intersection between any descending $k_1$-manifold 
and an ascending $(d-k_2)$-manifold
such that $k = \min\{k_1,k_2\}$ (the ascending $(d-k_2)$-manifold has dimension $k_2$).
When we simply say a cell, we are referring to the $d$-cell
since $d$-cells consists of the majority of $\K$ 
(the totality of $k$-cells with $k<d$ has Lebesgue measure 0).
The \emph{Morse-Smale complex} for $f$ is the collection of
all $k$-cells for $k=0,\cdots, d$.
The bottom panels of Figure~\ref{Fig::ex_D} give examples for the ascending manifolds
and the $d$-cells for $d=2$.
Another example is given in Figure~\ref{Fig::ex_MS}.

The cells of a smooth function can be used to construct an additive decomposition
that is useful in data analysis.
For a Morse-Smale function $f$, let $E_1,\cdots, E_L$ be its associated cells. 
Then we can decompose $f$ into
\begin{equation}
f(x) = \sum_{\ell=1}^L f_\ell (x) 1(x\in E_\ell),
\label{eq::additive}
\end{equation}
where each $f_\ell(x)$ behaves like a multivariate isotonic 
function \citep{barlow1972statistical,bacchetti1989additive}.
Namely, $f(x) = f_\ell(x)$ when $x\in E_\ell$.
This decomposition is because within each $E_\ell$, $f$ has exact a local mode and a local minimum
on the boundary of $E_\ell$. 
The fact that $f$ admits such a decomposition will be used frequently 
in Section~\ref{sec::MSR} and \ref{sec::MSS}.

Among all descending/ascending manifolds,
the descending $d$-manifolds and the ascending $0$-manifolds are
often of great interest.
%the highest order (d-manifolds) manifolds are often of great interest.
For instance, mode clustering \citep{Li2007, azzalini2007clustering}
uses the descending $d$-manifolds to partition the 
domain $\K$ into clusters.
Morse-Smale regression \citep{gerber2011data, gerber2013morse}
fits a linear regression individually over 
each $d$-cell (non-empty intersection of pairs of descending $d$-manifolds and ascending $0$-manifolds).
Regions outside descending $d$-manifolds or ascending $0$-manifolds
have Lebesgue measure $0$.
Thus, later in our theoretical analysis, we will focus on the stability of the set $D_d$ and $A_0$
(see Section~\ref{sec::thm::stability}). 
We define boundaries of $D_d$ as
\begin{equation}
B\equiv\partial D_d = D_{d-1}\cup\cdots \cup D_{0}.
\end{equation}
The set $B$ will be used frequently in Section \ref{sec::thm}.
%and equivalently, we define
%\begin{equation}
%A\equiv\partial A_d = A_{d-1}\cup\cdots \cup A_{0}
%\end{equation}
%to be the boundaries for $A_d$.

\section{Applications in Statistics}

\subsection{Mode Clustering}	\label{sec::mode}

\begin{figure}
\center
\subfigure[Basins of attraction]{
	\includegraphics[width=1.45in]{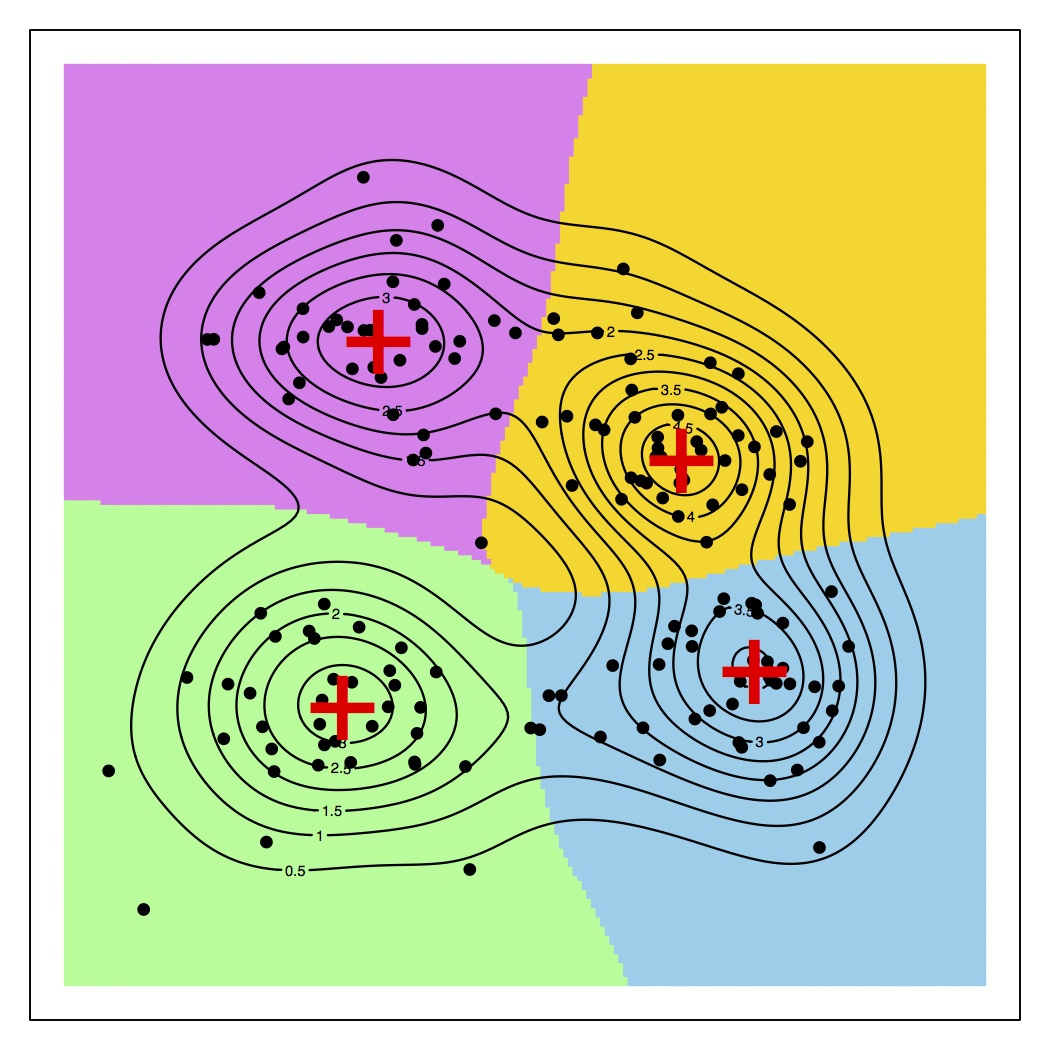}
}
\subfigure[Gradient ascent]{
	\includegraphics[width=1.45in]{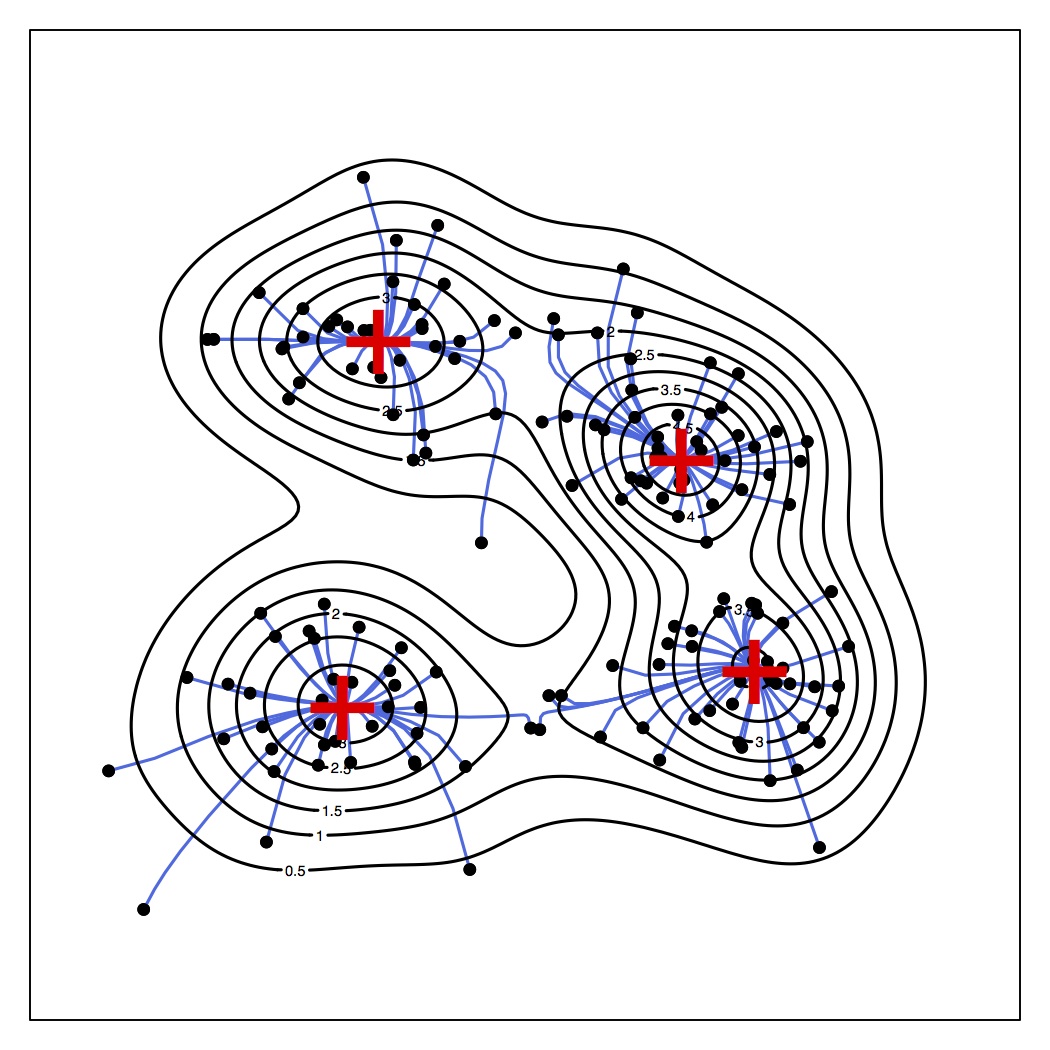}
}
\subfigure[Mode clustering]{
	\includegraphics[width=1.45in]{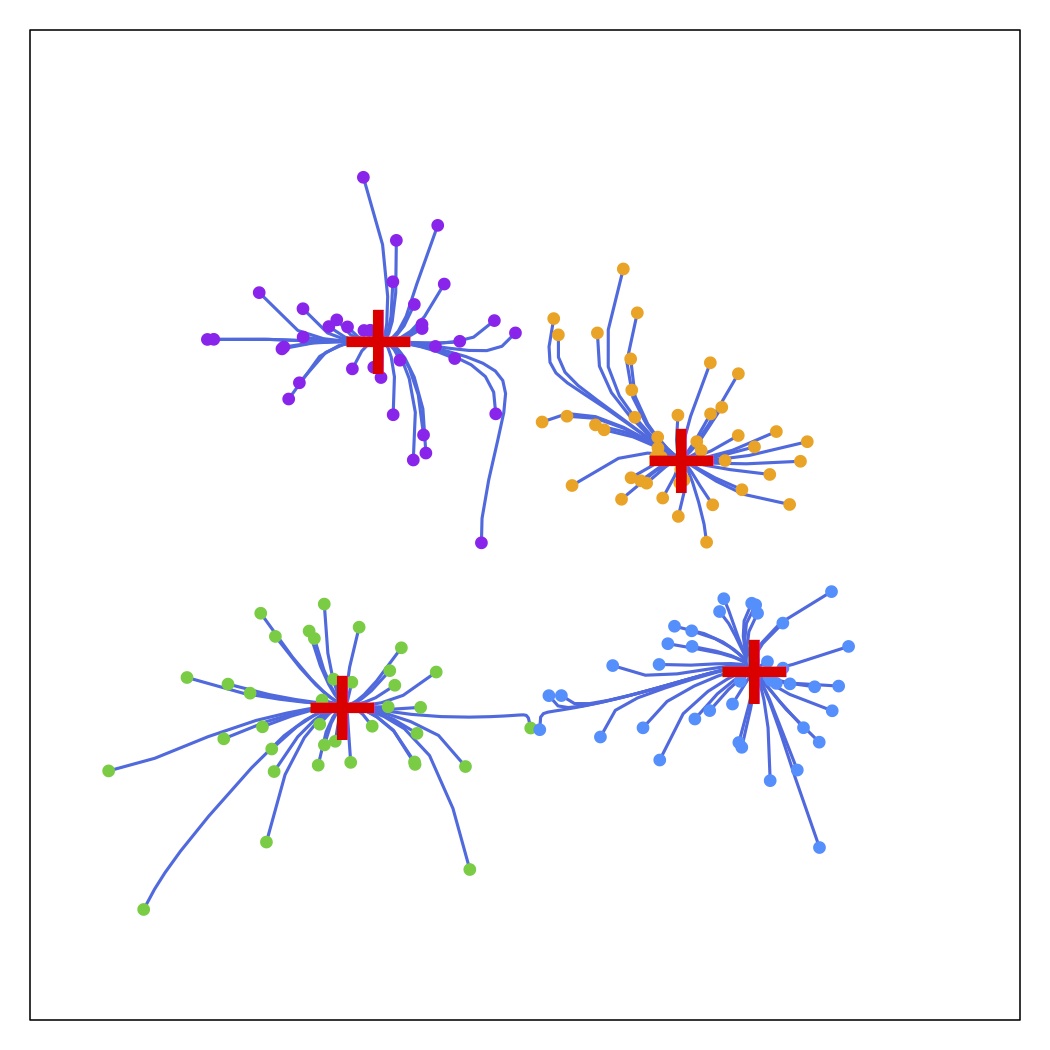}
}
\caption{An example of mode clustering.
(a): Basin of attraction for each local mode (red $+$).
Black dots are data points.
(b): Gradient flow (blue lines) for each data point.
The gradient flow starts at one data point and ends at one local modes.
(c): Mode clustering; we use the destination for
gradient flow to cluster data points.
}
\label{Fig::ex_Modeclustering}
\end{figure}

Mode clustering
\citep{Li2007,azzalini2007clustering,
Chacon2012,arias2016estimation,chacon2015population, chen2016comprehensive}
is a clustering technique based on the Morse-Smale complex
and is
also known as mean-shift clustering
\citep{fukunaga1975estimation,cheng1995mean,comaniciu2002mean}.
Mode clustering
uses the descending $d$-manifolds of
the density function $p$
to partition the whole space $\K$.
(Although the $d$-manifolds do not contain all
points in $\K$, the regions outside $d$-manifolds have Lebesgue measure $0$).
See Figure~\ref{Fig::ex_Modeclustering} for an example.

Now, we briefly describe the procedure of mode clustering.
Let $\cX = \{X_1,\cdots,X_n\}$ be a random sample from density $p$
defined on a compact set $\K$ and
assumed to be a Morse function.
Recall that $\dest(x)$ is the destination of the gradient ascent flow starting from $x$.
Mode clustering partitions the sample based on $\dest(x)$ for each point;
specifically, it partitions $\cX = \cX_1\bigcup\cdots\bigcup\cX_K$ such that
$$
\cX_\ell = \{X_i\in\cX: \dest(X_i)= m_\ell\},
$$
where each $m_\ell$ is a local mode of $p$.
We can also view mode clustering as a clustering technique based on the $d$-descending manifolds.
Let $D_d = D_{d,1}\bigcup\cdots\bigcup D_{d,L}$ be the $d$-descending manifolds of $p$,
assuming that $L$ is the number of local modes.
Then each cluster $\cX_\ell = \cX\bigcap D_{d,\ell}$.

In practice, however, we do not know $p$ so we have to use a density estimator $\hat{p}_n$.
A common density estimator is 
the kernel density estimator (KDE):
\begin{equation}
\hat{p}_n(x) = \frac{1}{nh^d}\sum_{i=1}^n K\left(\frac{x-X_i}{h}\right),
\end{equation}
where $K$ is a smooth kernel function and $h>0$ is the smoothing parameter.
Note that mode clustering is not limited to the KDE; other density estimators also give 
us a sample-based mode clustering.
Based on the KDE, we are able to estimate gradient $\hat{g}_n(x)$, the gradient flows $\hat{\pi}_x(t)$,
and the destination $\hat{\dest}_n(x)$ (note that the mean shift algorithm is an algorithm to perform these tasks). 
Thus, we can estimate the $d$-descending manifolds by the plug-in from $\hat{p}_n$.
Let $\hat{D}_d = \hat{D}_{d,1}\bigcup\cdots\bigcup \hat{D}_{d,\hat{L}}$ be the $d$-descending manifolds of $\hat{p}_n$,
where $\hat{L}$ is the number of local modes of $\hat{p}_n$.
The estimated clusters will be $\hat{\cX}_1,\cdots,\hat{\cX}_{\hat{L}}$, where each
$\hat{\cX}_\ell = \cX \bigcap \hat{D}_{d,\ell}$.
Figure~\ref{Fig::ex_Modeclustering} displays an example of mode clustering using the KDE.

%An unsolved problem for the mode clustering is its statistical consistency.
%We will derive rate of convergence for mode clustering 

A nice property of mode clustering is that 
there is a clear population quantity that our 
estimator (clusters based on the given sample)
is estimating: the population partition of the data points.
Thus we can consider properties of the procedure such as consistency,
which we discuss in detail in Section \ref{sec::thm::MC}.

\subsection{Morse-Smale Regression}	\label{sec::MSR}

Let $(X,Y)$ be a random pair where
$Y\in \mathbb{R}$ and
$X_i \in\mathbb{K}\subset \mathbb{R}^d$.
Estimating the regression function
$m(x) = \mathbb{E}[Y|X=x]$ is challenging for $d$ of even moderate size.
A common way to address this problem is to use
a simple regression function that can be estimated with low variance.
For example, one might use an additive
regression
of the form
$m(x) = \sum_j m_j(x_j)$ which is a sum of one-dimensional smooth functions.
Although the true regression function is unlikely to be of this form,
it is often the case that the resulting estimator is useful.

A different approach, \emph{Morse-Smale regression} (MSR),
is suggested in \cite{gerber2013morse}.
This takes advantage of the (relatively) simple structure
of the Morse-Smale complex and the isotone behavior of the function on
each cell.
Specifically, MSR constructs a piecewise linear approximation to $m(x)$ over the cells of the 
Morse-Smale complex.

We first define the population version of the MSR.
Let $m(x) = \E(Y|X=x)$ be the regression function and is assumed 
to be a Morse-Smale function.
Let $E_1,\cdots E_L$ be the $d$-cells for $m$.
The Morse-Smale Regression for $m$ is a piecewise linear function within
each cell $E_\ell$ such that
\begin{equation}
m_{\MSR}(x) = \mu_\ell+\beta_\ell^Tx, \mbox{ for }x \in E_\ell,
\end{equation}
where $(\mu_\ell,\beta_\ell)$ are obtained by minimizing mean square error:
\begin{equation}
\begin{aligned}
(\mu_\ell, \beta_\ell)&= \underset{\mu,\beta}{\sf \argmin}\,\, \mbox{ }\mathbb{E}\left((Y-m_{\MSR}(X))^2|X\in E_\ell\right)\\
&= \underset{\mu,\beta}{\sf \argmin}\,\, \mbox{ }\mathbb{E}\left((Y- \mu-\beta^TX)^2|X\in E_\ell\right)
\end{aligned}
\label{eq::LSE1}
\end{equation}
That is,
$m_{\MSR}$ is the best linear piecewise predictor using the $d$-cells.
One can also view MSR as using a linear function to approximate 
$f_\ell$ in the additive model \eqref{eq::additive}.
Note that $m_{\MSR}$ is well defined except 
on the boundaries of $E_\ell$ that have Lebesgue measure $0$.

%Essentially, Morse-Smale Regression \citep{gerber2013morse}
%is very similar to the Morse-Smale approximation function.
%The only difference is that instead of minimizing the $\cL_2$ loss,
%we minimize the $\cL_2(\P_X)$ loss where $\P_X$
%is the distribution to the covariates.
%Namely, we are looking for the 
%best piecewise linear \emph{predictor}.

Now we define the sample version of the MSR.
Let $(X_1,Y_1),\cdots,(X_n,Y_n)$ be the random sample from the probability measure 
$\P_X\times \P_Y$
such that $X_i\in\K\subset\R^d$ and $Y_i\in \R$.
Throughout section \ref{sec::MSR},
we assume
the density of covariates $X$ is bounded, positive and 
has a compact support $\K$ and
the response $Y$ has finite second moment.

Let $\hat{m}_n$ be a smooth nonparametric regression estimator for $m$.
We call $\hat{m}_n$ the pilot estimator.
%\ATTNC{Does it deserve some comment that this will need to be computed
%in moderate to large dimensions?}
For instance, one may use
the kernel regression \cite{nadaraya1964estimating}
$\hat{m}_n(x) = \frac{\sum_{i=1}^n Y_i K\left(\frac{x-X_i}{h}\right)}{\sum_{i=1}^n K\left(\frac{x-X_i}{h}\right)}$ 
as the pilot estimator.
%Assume that we are using the kernel regression $\hat{m}_n$
%\eqref{eq::Kreg}
%for estimating $m$ with a smooth kernel function (e.g. Gaussian kernel). 
We define $d$-cells for $\hat{m}_n$
as $\hat{E}_1,\cdots,\hat{E}_{\hat{L}}$.
Using the data $(X_i,Y_i)$ within each estimated $d$-cell, $\hat{E}_\ell$,
the MSR for $\hat{m}_n$ is given by
\begin{equation}
\hat{m}_{n, \MSR}(x) = \hat{\mu}_\ell+\hat{\beta}_\ell^Tx, \mbox{ for }x \in \hat{E}_\ell,
\end{equation}
where $(\hat{\mu}_\ell,\hat{\beta}_\ell)$ are obtained by minimizing the empirical
squared error:
\begin{equation}
\begin{aligned}
(\hat{\mu}_\ell, \hat{\beta}_\ell)= \underset{\mu,\beta}{\sf \argmin} \mbox{ }\sum_{i: X_i\in \hat{E}_\ell}(Y_i- \mu-\beta^TX_i)^2
\end{aligned}
\label{eq::LSE2}
\end{equation}
This MSR is slightly different from the original version in
\cite{gerber2013morse}.  We will discuss the difference in Remark
\ref{rm::MSR}.  
Computing the parameters of MSR is not very
difficult--we only need to compute the cell labels of each observation
(this can be done by the mean shift algorithm or some fast variants
such as the quick-shift algorithm \citealt{vedaldi2008quick}) and then
fit a linear regression within each cell.

MSR may give low prediction error in some cases; see
\cite{gerber2013morse} for some concrete examples.
In Theorem~\ref{thm::MSR2}, we prove that 
we may estimate $m_{\MSR}$ at a fast rate. 
Moreover, the regression function may be visualized by the methods discussed later.

\begin{remark}
\label{rm::MSR}
The original version of Morse-Smale regression proposed in \cite{gerber2013morse}
does not use $d$-cells of a pilot nonparametric estimate $\hat{m}_n$.
Instead, they directly find local modes and minima using the original data points $(X_i,Y_i)$.
This saves computational effort but comes with a price: 
there is no clear population quantity being
estimated by their approach.
That is, when the sample size increases to infinity, there is no guarantee that
their method will converge.
In our case, we apply a consistent pilot estimate for $m$
and construct $d$-cells on this pilot estimate.
As is shown in Theorem~\ref{thm::MSR}, 
our method is consistent for this population quantity.
\end{remark}

\subsection{Morse-Smale Signatures and Visualization}	
\label{sec::MSS}

In this section we define a new method
for visualizing multivariate functions
based on the Morse-Smale complex,
called 
\emph{Morse-Smale signatures}.
The idea is very similar to the Morse-Smale regression but the signatures
can be applied to any Morse-Smale function.

Let $E_1,\cdots, E_K$ be the $d$-cells (nonempty intersection of 
a descending $d$-manifold and
an ascending $0$-manifold) for a
Morse-Smale function $f$ that has a compact support $\K$.  The
function $f$ depends on the context of the problem.  For density
estimation, $f$ is the density $p$ or its estimator $\hat{p}_n$.  For
regression problem, $f$ is the regression function $m$ or a
nonparametric estimator $\hat{m}_n$ .  For two sample test, $f$ is the
density difference $p_1-p_2$ or the estimated density difference
$\hat{p}_1-\hat{p}_2$.  Note that $E_1,\cdots, E_K$ form a partition
for $\K$ except a Lebesgue measure $0$ set.  
Each cell corresponds to a unique pair of a local
mode and a local minimum.  Thus, the local modes and minima along with
$d$-cells form a \emph{bipartite} graph which we call it \emph{signature
graph}.  The signature graph contains geometric information about
$f$.  See Figure~\ref{fig::ex::MSS} and \ref{Fig::ex::MSvis} for
examples.

The signature is defined as follows.
We project the maxima and minima of the function 
into $\mathbb{R}^2$ using multidimensional scaling.
We connect a maximum and minimum by an edge
if there exists a cell that connects them.
The width of the edge is proportional to
the norm of the linear coefficients of the linear approximation to the function
within the cell.
The linear approximation is
\begin{equation}
f_{\MS}(x) = \eta^\dagger_\ell+\gamma^{\dagger T}_\ell x, \quad \mbox{for }x \in E_\ell,
\end{equation}
where $\eta_\ell^\dagger\in\R$ and $\gamma_\ell^\dagger \in \R^d$
are parameters from
\begin{equation}
(\eta_\ell^\dagger, \gamma_\ell^\dagger) = \underset{\eta, \gamma}{\sf argmin} \int_{E_\ell} 
\left(f(x)-\eta-\gamma^Tx\right)^2 dx.
\label{eq::MSS1}
\end{equation}
This is again a linear approximation for $f_\ell$ in the additive model \eqref{eq::additive}.
Note that $f_{\MS}$ may not be continuos when we move from one cell to another. 
The summary statistics for the edge associated with cell $E_\ell$ 
are the parameters $(\eta_\ell^\dagger, \gamma_\ell^\dagger)$.
We call the function $f_{\MS}$ the \emph{(Morse-Smale) approximation function};
it is the best piecewise-linear representation for $f$ (piecewise linear within each cell)
under $\cL_2$ error given the $d$-cells.
This function is well-defined except on a set of Lebesgue measure $0$ (the boundaries of each cell).
See Figure~\ref{fig::ex::MSS} for a example on the approximation function.
The details are in Algorithm \ref{Alg::MSS}.

\begin{figure}
\center
\subfigure[Original function]{
	\includegraphics[width=1.7 in]{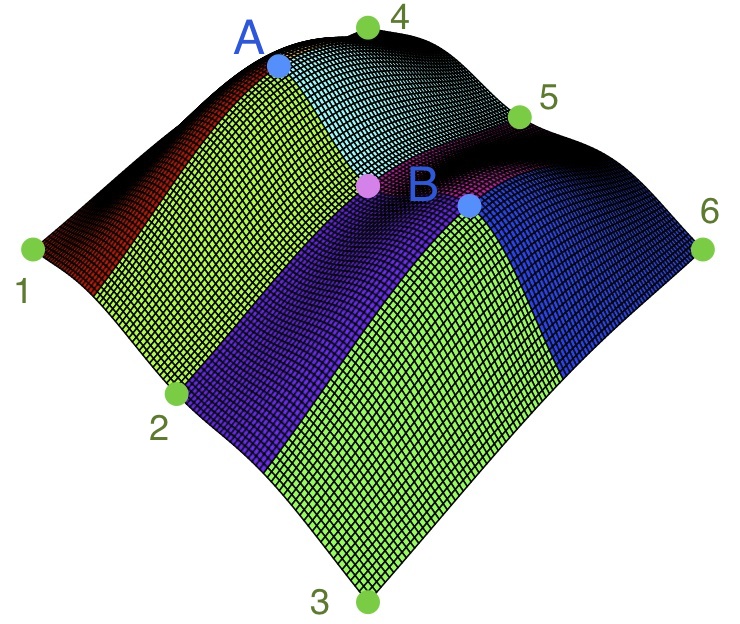}
}
\subfigure[Approximation function]{
	\includegraphics[width=1.7 in]{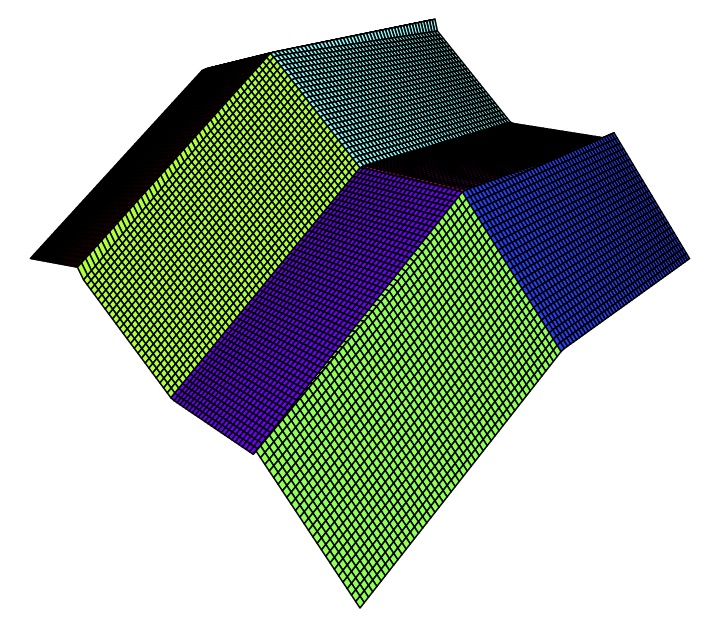}
}
\subfigure[Signature graph]{
	\includegraphics[width=1.7 in]{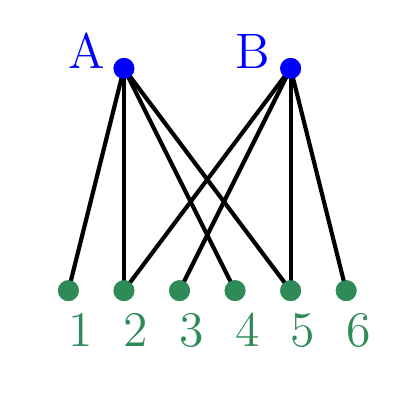}
}
\caption{Morse-Smale signatures for a smooth function.
(a): The original function. 
The blue dots are local modes, the green dots are local minima
and the pink dot is a saddle point. (b): The Morse-Smale approximation to (a).
This is the best piecewise linear approximation to the original function given $d$-cells.
(c): This bipartite graph 
has
nodes that are local modes and minima and edges
that represent the $d$-cells.
Note that we can summarize the smooth function (a) by the signature graph (c)
and the parameters for constructing approximation function (b).
The signature graph and parameters for approximation function define 
the Morse-Smale signatures.
%\textcolor{red}{can we fix this bi-partite graph?}
}
\label{fig::ex::MSS}
\end{figure}

\begin{figure}
\center
\includegraphics[width=2.5in]{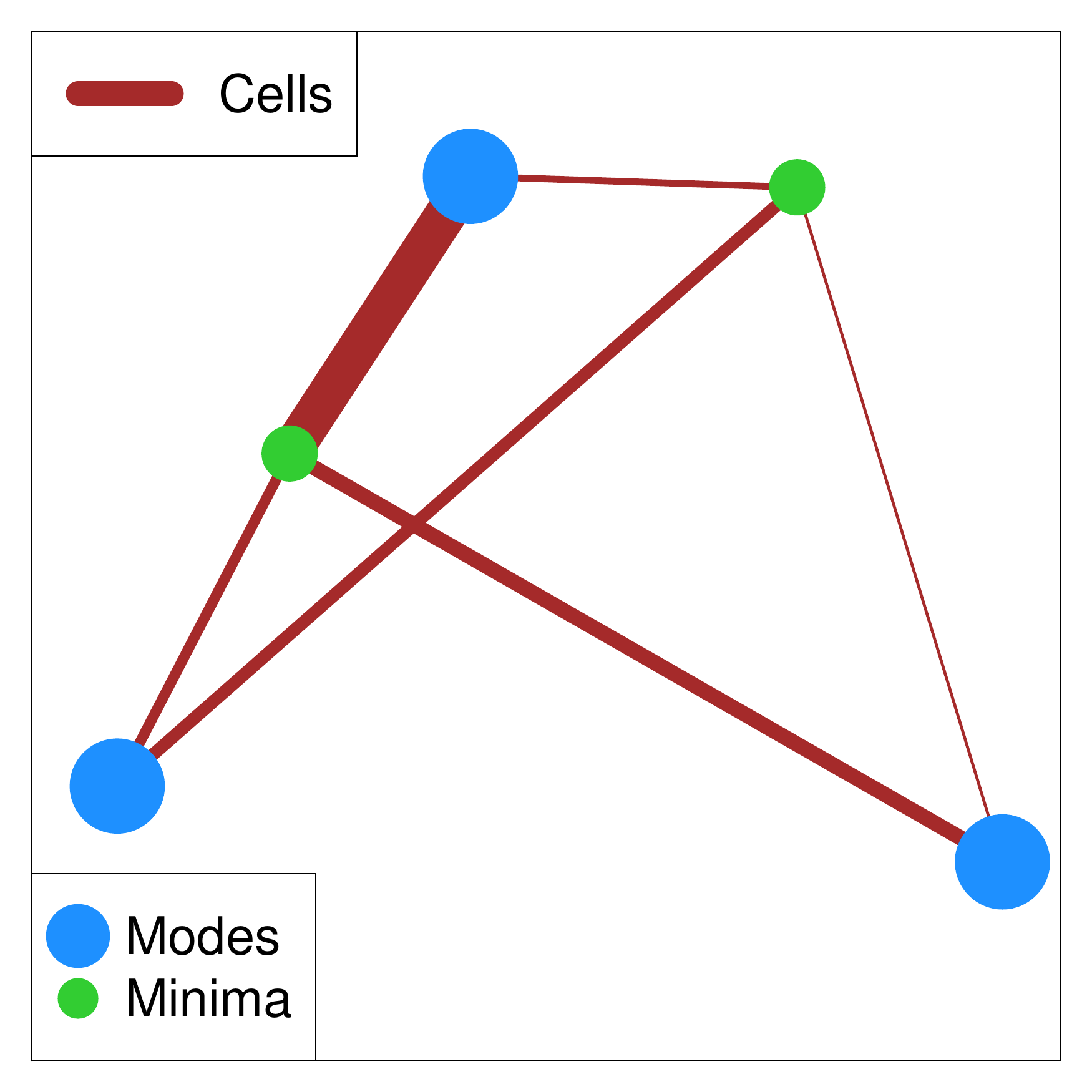}
\caption{Morse-Smale Signature visualization (Algorithm \ref{Alg::MSS})
of the density difference for GvHD dataset (see Figure~\ref{Fig::ex0}).
The blue dots are local modes; the green dots are local minima;
the brown lines are $d$-cells.
These dots and lines form the signature graph.
The width indicates the $\cL_2$ norm
for the slope of regression coefficients. i.e. $\norm{\gamma^\dagger_\ell}$.
The location for modes and minima are obtained by
multidimensional scaling so that 
the relative distance is preserved.
%\ATTNC{It would be nice if this could be clear without the legend, as it
%becomes a simple icon.  Is there an obvious way to denote maxima and minima
%rather than different sized circles? And we probably don't need to indicate
%cells as that is basic to the definition.  Just a thought, not essential.}
}
\label{Fig::ex::MSvis}
\end{figure}

{\bf Example.}
Figure~\ref{Fig::ex::MSvis} is an example
using the GvHD dataset.
We first conduct multidimensional scaling
\citep{kruskal1964multidimensional} on the local modes and minima
for $f$ and plot them on the 2-D plane.
In Figure~\ref{Fig::ex::MSvis}, the blue dots are local modes
and the green dots are local minima.
These dots act as the nodes for the signature graph.
Then we add edges, representing the cells for $f$
that connect pairs of local modes and minima, to form the signature graph.
Lastly, we adjust the width for the edges according to the 
strength ($\cL_2$ norm) of regression function within each cell (i.e. $\norm{\gamma^\dagger_\ell}$).
Algorithm \ref{Alg::MSS} provides a summary for visualizing 
a general multivariate function using what we described in this paragraph.

\begin{algorithm}
\caption{Visualization using Morse-Smale Signatures}
\label{Alg::MSS}
\begin{algorithmic}
\State \textbf{Input:} 
Grid points $x_1,\cdots,x_N$ and the functional evaluations $f(x_1),\cdots,f(x_N)$.
\State 1. Find local modes and minima of $f$ on the discretized points $x_1,\cdots, x_N$.
Let $M_1,\cdots M_K$ and $m_1,\cdots, m_S$ denote the grid points for modes and minima.
\State 2. Partition $\{x_1,\cdots, x_N\}$ into $\mathcal{X}_1,\cdots\mathcal{X}_L$ 
according to the $d$-cells of $f$ 
(1. and 2. can be done by using a k-nearest neighbor gradient
ascent/descent method; see Algorithm 1 in \cite{gerber2013morse}).
\State 3. For each cell $\mathcal{X}_\ell$, fit a linear regression with
$(X_i, Y_i) = (x_i, f(x_i))$, where $x_i \in \mathcal{X}_\ell$.
Let the regression coefficients (without intercept) be $\beta_\ell$.
\State 4. Apply multidimensional scaling to modes and minima jointly.
Denote their 2 dimensional representation points as
$$
\{M^*_1,\cdots M^*_K, m^*_1,\cdots, m^*_S\}.
$$
\State 5. Plot $\{M^*_1,\cdots M^*_K, m^*_1,\cdots, m^*_S\}$.
\State 6. Add edge to a pair of mode and minimum if there exist a cell
that connects them. The width of the edge is in proportional to
$\norm{\beta_\ell}$ (for cell $\mathcal{X}_\ell$).
\end{algorithmic}
\end{algorithm}

\subsection{Two Sample Comparison}	\label{sec::two}

The Morse-Smale complex can be used to compare two samples.
There are two ways to do this.
The first one is to test the difference in two density functions locally
and then use the Morse-Smale signatures to visualize regions where 
the two samples are different.
The second approach is to conduct
a nonparametric two sample test within each
Morse-Smale cell.
The advantage of the first approach is that we obtain a visual display on
where the two densities are different.
The merit of the second method is that we gain additional power in testing
the density difference by using the shape information.

%\textcolor{red}{Under $H_0$ there is no bias in the the difference of the density estimates.
%This suggests we can make $h$ large. Can we make this precise?}
%YC: Actually, large $h$ may result in a low power. Despite the fact that we have fast rate for variance,
%the bias will make two densities look similar (this makes it hard to reject H0).  

\subsubsection{Visualizing the Density Difference}

Let $X_1,\ldots X_n$ and $Y_1,\ldots, Y_m$ be two random sample
with densities $p_X$ and $p_Y$.
In a two sample comparison,
we not only want to know
if $p_X=p_Y$
but we also want to find the regions
that they significantly disagree.
That is, we are doing the local tests
\begin{equation}
H_0(x): p_X(x)= p_Y(x)
\end{equation}
simultaneously for all $x\in\K$
and we are interested in the regions where we reject $H_0(x)$.
A common approach is to estimate the density for both sample
by the KDE and set a threshold to pickup those regions
that the density difference is large. Namely,
we first construct density estimates
\begin{equation}
\hat{p}_X(x) = \frac{1}{nh^d}\sum_{i=1}^nK\left(\frac{x-X_i}{h}\right), \quad
\hat{p}_Y(x) = \frac{1}{mh^d}\sum_{i=1}^mK\left(\frac{x-Y_i}{h}\right)
\end{equation}
and then compute $\hat{f}(x) = \hat{p}_X(x)-\hat{p}_Y(x)$.
The regions 
\begin{equation}
\Gamma(\lambda) = \left\{x\in\K: |\hat{f}(x)|> \lambda\right\}
\end{equation}
are where we have strong evidence 
to reject $H_0(x)$.
%that the two densities disagree with each other.
The threshold $\lambda$ can be picked by
quantile values of the bootstrapped $\cL_{\infty}$ density deviation
to control type 1 error
or can be chosen by controlling the false discovery rate \citep{duong2013local}.

Unfortunately, $\Gamma(\lambda)$ is hard to visualize when $d > 3$.
So we use the Morse-Smale complex for $\hat{f}$
and visualize $\Gamma(\lambda)$ by its behavior on the $d$-cells of the complex.
Algorithm \ref{Alg::vis} gives a method for visualizing density differences
like $\Gamma(\lambda)$ in the context of comparing two independent samples.

\begin{algorithm}
\caption{Visualization For Two Sample Test}
\label{Alg::vis}
\begin{algorithmic}
\State \textbf{Input:} Sample 1: $\{ X_1,...X_n\}$,
Sample 2: $\{Y_1,\cdots, Y_m\}$, threshold $\lambda$ and radius constant $r_0$
\State 1. Compute the density estimates $\hat{p}_X$ and $\hat{p}_Y$.
\State 2. Compute the difference function $\hat{f} = \hat{p}_X-\hat{p}_Y$
and the significant regions
\begin{equation}
\Gamma^+(\lambda) =\left \{x\in\K: \hat{f}(x)>\lambda\right\},\quad
\Gamma^-(\lambda) = \left\{x\in\K: \hat{f}(x)< -\lambda\right\}
\end{equation}
\State 3. Find the $d$-cells for $\hat{f}$, denoted as $E_1,\cdots, E_L$.
%% find center
\State 4. For cell $E_\ell$, do (4-1) and (4-2):
\State 4-1. compute the cell center $e_\ell$, cell size $V_\ell = \Vol(E_\ell)$,
\State 4-2. compute the positive significant ratio and negative significant ratio
\begin{equation}
r^+_\ell = \frac{\Vol(E_\ell \cap \Gamma^+(\lambda))}{\Vol (E_\ell)}, \quad
r^-_\ell = \frac{\Vol(E_\ell \cap \Gamma^-(\lambda))}{\Vol (E_\ell)}.
\end{equation}
\State 5. For every pair of cell $E_j$ and $E_\ell$ $(j\neq \ell)$,
compute the shared boundary size:
\begin{equation}
B_{j\ell} = \Vol_{d-1} (\bar{E}_j\cap \bar{E}_\ell),
\end{equation}
where $\Vol_{d-1}$ is the $d-1$ dimensional Lebesgue measure.

\State 6. Do multidimensional scaling
\citep{kruskal1964multidimensional} to $e_1,\cdots, e_L$ to obtain
low dimensional representation $\tilde{e}_1,\cdots,\tilde{e}_L$. 
%% compute ratio
\State 7. Place a ball center at each $\tilde{e}_\ell$ with radius 
$r_0\times\sqrt{V_\ell}$.
\State 8. If $r^+_\ell + r^-_\ell>0$, add a pie chart center
at $\tilde{e}_\ell$
with radius $r_0\times\sqrt{V_\ell}\times(r^+_\ell + r^-_\ell)$.
The pie chart contains two groups, each with ratio $\left(\frac{r^+_\ell }{r^+_\ell + r^-_\ell}, \frac{ r^-_\ell}{r^+_\ell + r^-_\ell}\right)$.
\State 9. Add a line to connect two nodes $\tilde{e}_j$ and $\tilde{e}_\ell$
if $B_{j\ell}>0$. We may adjust the thickness of the line according to $B_{j\ell}$.

%\State \textbf{Output:} 
\end{algorithmic}
\end{algorithm}

An example for Algorithm \ref{Alg::vis} is in Figure~\ref{Fig::ex0},
in which we apply the visualization algorithm for the
the GvHD dataset by using kernel density estimator. 
We choose the threshold $\lambda$ by bootstrapping 
the $\cL_{\infty}$ difference for $\hat{f}$ i.e.
$\sup_x |\hat{f}^*(x)-\hat{f}(x)|$, where $\hat{f}^*$
is the density difference for the bootstrap sample. 
We pick $\alpha=95\%$ upper quantile value for 
the bootstrap deviation as the threshold.

The radius constant $r_0$ is defined by the user.
It is a constant for visualization and does not affect the analysis.
Algorithm \ref{Alg::vis} preserves the relative position
for each cell and visualizes the cell according to its size.
The pie-chart provides the ratio of regions where the two densities are significantly
different.
The lines connecting two cells provide the geometric information about
how cells are connected to each other.

By applying Algorithm \ref{Alg::vis} to the GvHD dataset (Figure~\ref{Fig::ex0}), we 
find that there are 6 cells and one cell much larger than the others.
Moreover, 
in most regions, the blue regions are larger than the red areas.
This indicates that compared to the density of the control group, the density of the GvHD group seem to 
concentrates more so that the regions above the threshold are larger.

\subsubsection{Morse-Smale Two-Sample Test}

Here we introduce a technique combining the energy test
\citep{baringhaus2004new,szekely2004testing,szekely2013energy}
and the Morse-Smale complex to conduct a two sample test.
We call our method the \emph{Morse-Smale Energy test (MSE test)}.
The advantage of the MSE test is that it is a nonparametric test
and its power can be higher than the energy test; see Figure~\ref{Fig::ex_GMM}.
Moreover, we can combine our test with the visualization tool proposed in the previous section
(Algorithm \ref{Alg::vis});
see Figure~\ref{Fig::p_values} for an example for displaying p-values from MSE test
when visualizing the density difference.

%{\bf SAY WHY THIS IS BETTER THAN JUST USING THE ENERGY TEST.
%WHAT ARE THE ADNANTAGES?}

Before we introduce our method, we first review the ordinary energy test.
Given two random variables $X\in\R^d$ and $Y\in\R^d$, the energy distance
is defined as
\begin{equation}
\cE(X,Y) = 2\E\norm{X-Y} -\E\norm{X-X'}-\E\norm{Y-Y'},
\end{equation}
where $X'$ and $Y'$ are iid copies of $X$ and $Y$.
The energy distance has several useful applications such as
the goodness-of-fit testing \citep{szekely2005new}, 
two sample testing \citep{baringhaus2004new,szekely2004testing,szekely2013energy}, 
clustering \citep{szekely2005hierarchical}, and distance components
\citep{rizzo2010disco}
to name but few. 
We recommend an excellent review paper in \citep{szekely2013energy}.

For the two sample test, let $X_1,\cdots,X_n$ and $Y_1,\cdots,Y_m$
be the two samples we want to test.
The sample version of energy distance is
\begin{equation}
\hat{\cE}(X,Y) = \frac{2}{nm}\sum_{i=1}^n \sum_{j=1}^m \norm{X_i-Y_j}
- \frac{1}{n^2}\sum_{i=1}^n \sum_{j=1}^n \norm{X_i-X_j}
- \frac{1}{m^2}\sum_{i=1}^m \sum_{j=1}^m \norm{Y_i-Y_j}.
\end{equation}
If $X$ and $Y$ are from the sample population (the same density),
$\hat{\cE}(X,Y)\overset{P}{\rightarrow} 0$.
Numerically, we use the permutation test for computing the p-value
for $\hat{\cE}(X,Y)$. 
This can be done quickly in the R-package `energy' \citep{rizzo2008energy}.

Now we formally introduce our testing procedure: the MSE test
(see Algorithm \ref{Alg::two} for a summary).
%We call our test \emph{Morse-Smale Energy Test (MSE test)}.
Our test consists of three steps.
First, we split the data into two halves.
Second,
we use one half of the data (contains both samples)
to do a nonparametric density estimation (e.g. the KDE)
and then compute the Morse-Smale complex ($d$-cells).
Last,
we use the other half of the data
to conduct the energy distance two sample test
`within each $d$-cell'.
That is,
we partition the second half of the data by the $d$-cells.
Within each cell, we do the energy distance test.
If we have $L$ cells, we will have $L$ p-values
from the energy distance test.
We reject $H_0$ if any one of the $L$ p-values is smaller than
$\alpha/L$ (this is from Bonferroni correction).
Figure~\ref{Fig::p_values} provides an example for using 
the above procedure (Algorithm \ref{Alg::two})
along with the visualization method proposed in Algorithm \ref{Alg::vis}.
Data splitting is used to avoid using the same data twice,
which ensures we have a valid test.

\begin{algorithm}
\caption{Morse-Smale Energy Test (MSE test)}
\label{Alg::two}
\begin{algorithmic}
\State \textbf{Input:} Sample 1: $\{ X_1,...X_n\}$,
Sample 2: $\{Y_1,\cdots, Y_m\}$, smoothing parameter $h$, significance level $\alpha$
\State 1. Randomly split the data into halves $\cD_1$ and $\cD_2$; both contain
equal number of $X$ and $Y$ (assuming $n$ and $m$ are even).
\State 2. Compute the KDE $\hat{p}_X$ and $\hat{p}_Y$ by the first sample $\cD_1$.
\State 3. Find the $d$-cells for $\hat{f}=\hat{p}_X-\hat{p}_Y$, denoted as $E_1,\cdots, E_L$.
%% find center
\State 4. For cell $E_\ell$, do 4-1 and 4-2:
\State 4-1. Find $X$ and $Y$ in the second sample $\cD_2$,
\State 4-2. Do the energy test for two sample comparison. Let the p-value be $p(\ell)$
\State 5. Reject $H_0$ if $p(\ell)<\alpha/L$ for some $\ell$.
%\State \textbf{Output:} 
\end{algorithmic}
\end{algorithm}

%\ATTNC{It's fine to have the individual examples distributed throughout, but I wonder if
%a separate examples section might have more impact. Just a thought.}

{\bf Example.}
Figure~\ref{Fig::ex_GMM} shows a simple comparison
for the proposed MSE test to the usual Energy test.
We consider a $K=4$ Gaussian mixture model in $d=2$
with standard deviation of each component being the same $\sigma=0.2$
and the proportion for each component is $(0.2, 0.5, 0.2, 0.1)$.
The left panel displays a sample with $N=500$ from this mixture distribution.
We draw the first sample from this Gaussian mixture model.
For the second sample, we draw
a similar Gaussian mixture model except
that we change the deviation of one component.
In the middle panel, we change the deviation 
to the third component (C3 in left panel, which contains $20\%$ data points).
In the right panel, we change the deviation to the fourth component 
(C4 in left panel, which contains $10\%$ data points).
We use significance level $\alpha=0.05$
and for MSE test, we consider the Bonferroni correction and the smoothing bandwidth is chosen using 
Silverman's rule of thumb \citep{Silverman1986}.

%{\bf YEN-CHI: the reference to Silverman did not compile properly.}

Note that in both the middle and the right panels,
the left most case (added deviation equals $0$)
is where $H_0$ should not be rejected.
As can be seen from Figure~\ref{Fig::ex_GMM},
the MSE test has much stronger power compared to the usual Energy test.
%despite the fact that we slightly lost control of type-1 error 
%(we only control type-1 errors asymptotically).

%{\bf OUCH: We need to control the type I error}

The original energy test has low power
while the MSE test has higher power.
This is because the two distributions only differ at
a small portion of the regions so that a global test like energy test
requires large sample sizes to detect the difference.
On the other hand, the MSE test partitions the space 
according to the density difference
so that it is capable of detecting the local difference.

\begin{figure}
\center
	\includegraphics[width=1.5 in]{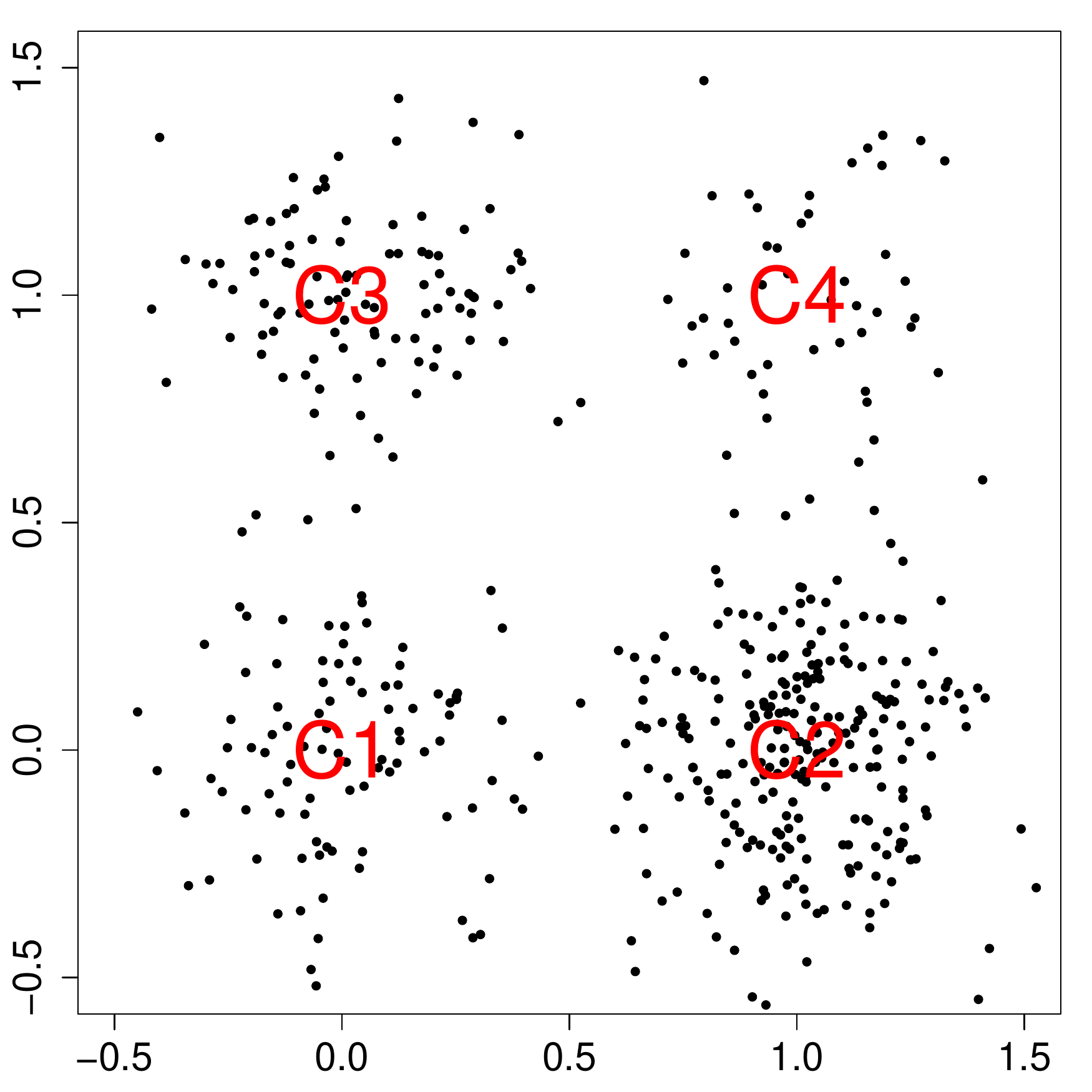}
	\includegraphics[width=1.5 in]{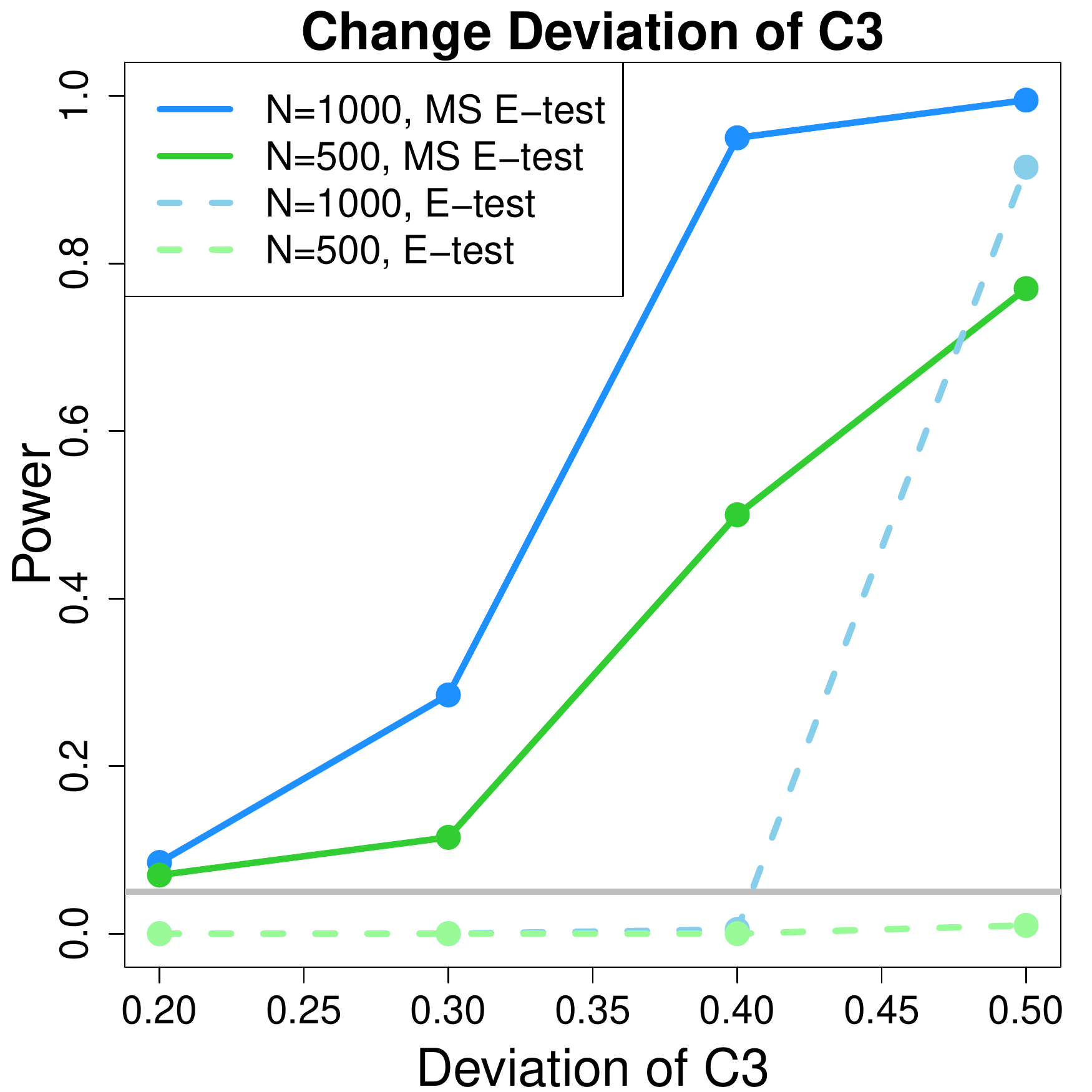}
	\includegraphics[width=1.5 in]{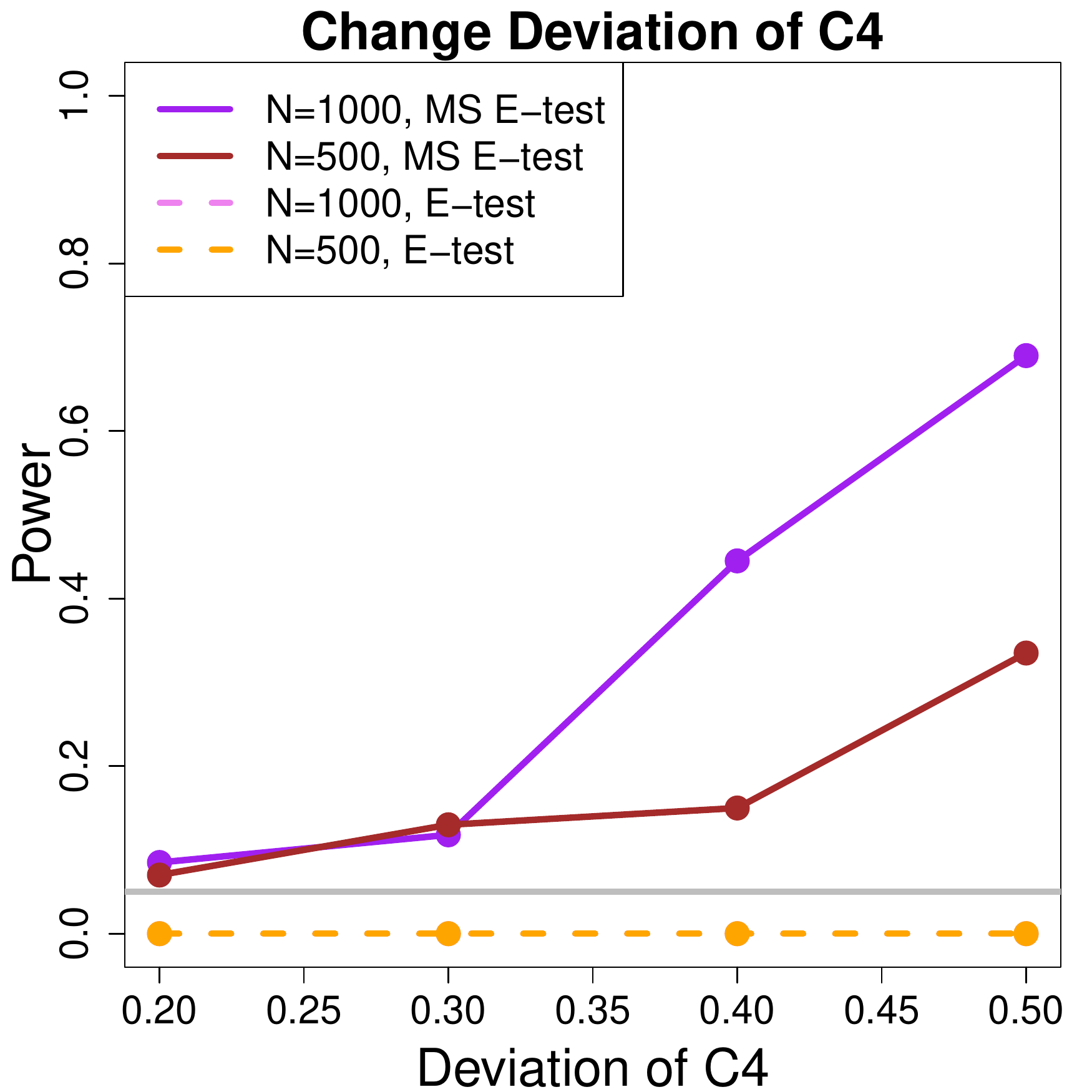}
\caption{An example comparing the Morse-Smale Energy test
to the original Energy test. We consider a $d=2$, $K=4$ Gaussian mixture model.
Left panel: an instance for the Gaussian mixture. We have four mixture components,
denoting as C1, C2, C3 and C4.
They have equal standard deviation ($\sigma=0.2$) and 
the proportions for each components are $(0.2, 0.5, 0.2, 0.1)$.
Middle panel: We changed the standard deviations of component C3 to $0.3, 0.4$ and $0.5$
and compute the power for the MSE test and the usual Energy test
at sample size $N=500$ and $1000$.
(Standard deviation equals $0.2$ is where $H_0$ should not be rejected.)
Right panel: We add the variance of component C4 (the smallest component) 
and do the same comparison as in the middle panel.
We pick the significance level $\alpha=0.05$ (gray horizontal line) and in the MSE test,
we reject $H_0$ if the minimal p-value is less than $\alpha/L$, where $L$
is the number of cells (i.e. we are using the Bonferroni correction).
}
\label{Fig::ex_GMM}
\end{figure}

{\bf Example.}
In addition to the higher power, we may combine the MSE test with the visualization
tool in Algorithm~\ref{Alg::vis}.
Figure~\ref{Fig::p_values} displays an example where we visualize the
density difference and simultaneously indicate the p-values from the Energy test within each cell
using the GvHD dataset.
This provides us more information about how two distributions differ from each other.

\begin{figure}
\center
\includegraphics[scale=0.3]{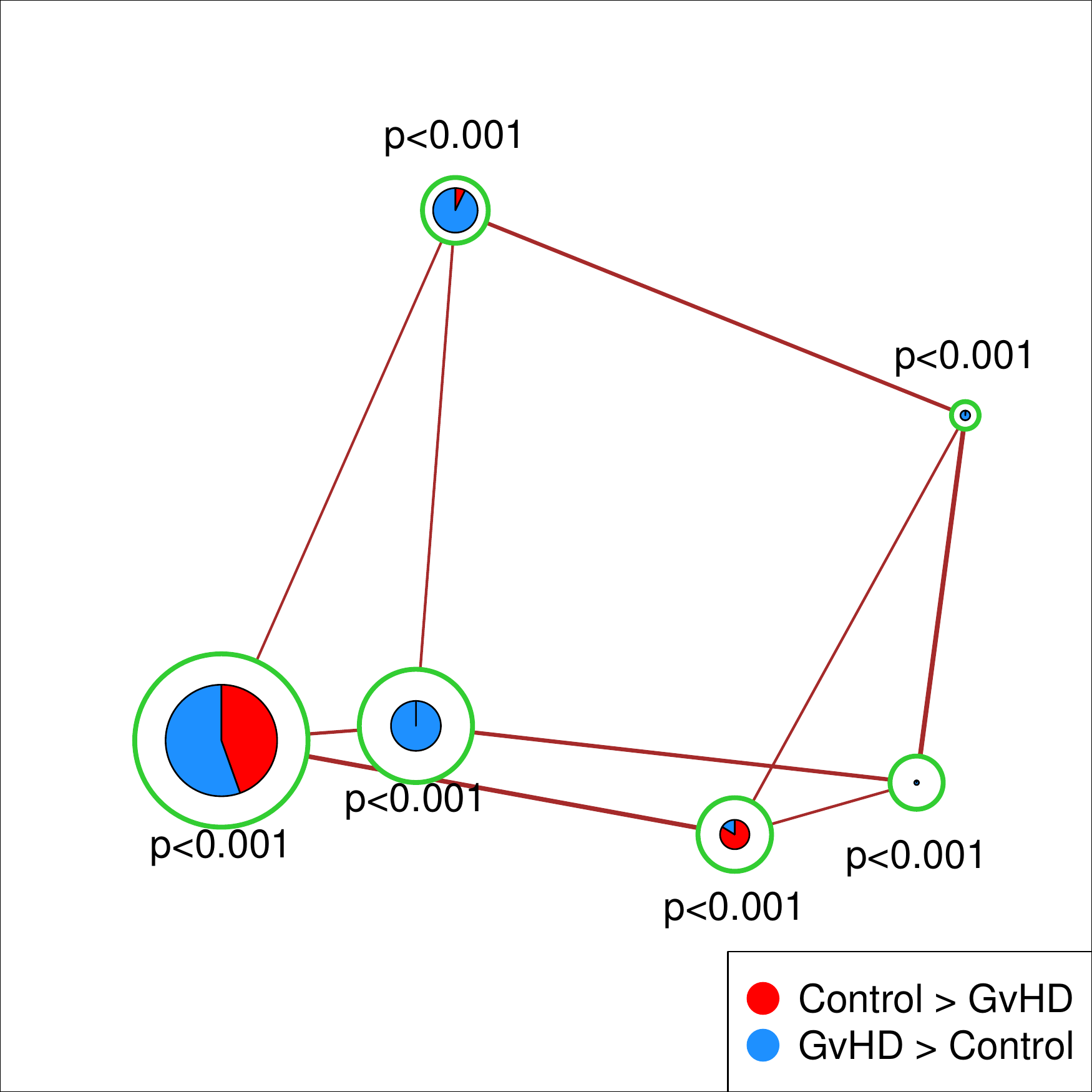}
\caption{
An example using both Algorithm \ref{Alg::vis} and \ref{Alg::two}
to the GvHD dataset introduced in Figure~\ref{Fig::ex0}.
We use data splitting as described in Algorithm \ref{Alg::two}.
For the first part of the data, we compute the cells and
visualize the cells using
Algorithm \ref{Alg::vis}.
Then we apply the energy distance two sample test for each cell as described 
in Algorithm \ref{Alg::two} and we annotate each cell with a p-value.
Note that the visualization is slightly different to Figure~\ref{Fig::ex0}
since we use only half of the original dataset in this case.
}
\label{Fig::p_values}
\end{figure}

\section{Theoretical Analysis}	\label{sec::thm}

We first define some notation for the theoretical analysis.
Let $f$ be a smooth function. We define $\|f\|_{\infty} = \sup_x |f(x)|$ to be the $\cL_\infty$-norm of $f$.
In addition, let $\norm{f}_{j,\max}$ denote
the elementwise $\cL_{\infty}$-norm for $j$-th derivatives of $f$.
For instance, 
$$
\norm{f}_{1,\max}  = \max_i \|g_i(x)\|_\infty, \quad\norm{f}_{2,\max} =  \max_{i,j}\|H_{ij}(x)\|_\infty.
$$
We also define $\|f\|_{0,\max} = \|f\|_{\infty}$. 
We further define
\begin{equation}
\norm{f}^*_{\ell, \max} = \max\left\{\norm{f}_{j,\max}: j=0,\cdots, \ell\right\}.
\label{eq::norm}
\end{equation}
The quantity $\norm{f-h}^*_{\ell,\max}$ measures the difference between
two functions $f$ and $h$ up to $\ell$-th order derivative.

For two sets $A,B$,
the Hausdorff distance is 
\begin{equation}
\Haus(A,B) = \inf\{r: A\subset B\oplus r, B\subset A\oplus r\},
\label{eq::Haus}
\end{equation}
where $A\oplus r = \{y: \min_{x\in A}\norm{x-y}\leq r\}$.
The Hausdorff distance is like the $\cL_\infty$ distance for sets.

Let $\tilde{f}:\mathbb{K}\subset\mathbb{R}^d\mapsto \mathbb{R}$ be a smooth function with 
bounded third derivatives.
Note that as long as $\norm{\tilde{f}-f}^*_{3,\max}$ is small, $\tilde{f}$
is also a Morse function by Lemma~\ref{lem::critical}.
Let $\tilde{D}$ denote the boundaries of the descending $d$-manifolds of $\tilde{f}$.
We will show if
$\norm{f-\tilde{f}}^*_{3,\max}$is sufficiently small, 
then $\Haus(\tilde{D},D) = O(\norm{\tilde{f}-f}_{1,\max})$.

\subsection{Stability of the Morse-Smale Complex}	\label{sec::thm::stability}

Before we state our theorem, we first derive some properties of descending manifolds.
Recall that we are interested in $B=\partial D_d$, the boundary of the descending $d$-manifolds 
(and $B$ is also the union 
of all $j$-descending manifolds for $j<d$).
Since each $D_j$ is a collection of smooth $j$-dimensional 
manifolds embedded in $\mathbb{R}^d$,
for every $x\in D_j$, there exists a basis $v_{1}(x),\cdots,v_{d-j}(x)$ 
such that each $v_k(x)$ is
perpendicular to $D_j$ at $x$ for $k=1,\cdots d-j$ 
\citep{bredon1993topology,helgason1979differential}.
%\textcolor{red}{Reference}
That is, $v_{1}(x),\cdots,v_{d-j}(x)$ span the normal space to $D_j$ at $x$.
For simplicity, we write 
\begin{equation}
V(x) = (v_{1}(x),\cdots,v_{d-j}(x))\in\mathbb{R}^{d\times (d-j)}
\end{equation} 
for $x\in D_j$.

Note the number of columns $d-j\equiv d-j(x)$ in $V(x)$ depends 
on which $D_j$ the point $x$ belongs to.
We use $j$ rather than $j(x)$ to simplify the notation.
For instance, if $x\in D_1$, $V(x) \in\mathbb{R}^{d\times (d-1)} $ 
and if $x\in D_{d-1}$, $V(x)\in \mathbb{R}^{d\times 1}$.
We also let
\begin{equation}
\mathbb{V}(x) = {\rm span}\{v_{1}(x),\cdots,v_{d-j}(x)\}
\end{equation} 
denote the normal space to $B$ at $x$.
One can view $\mathbb{V}(x)$ as the normal map of the manifold $D_j$ at $x\in D_j$.

For each $x\in B$, define the projected Hessian 
\begin{equation}
H_V(x) = V(x)^TH(x)V(x),
\end{equation}
which is the Hessian matrix of $p$ by taking gradients along column space of $V(x)$.
If $x\in D_j$, $H_V(x)$ is a $(d-j)\times (d-j)$ matrix. 
The eigenvalues of $H_V(x)$ determine how the gradient flows are moving away from $B$.
We let $\lambda_{\min}(M)$ be the smallest eigenvalue for a symmetric matrix $M$. 
If $M$ is a scalar (just one point), then $\lambda_{\min}(M)=M$.

\vspace{1cm}

\noindent
{\bf Assumption (D):} We assume that
$H_{\min} = \min_{x\in B} \lambda_{\min} (H_V(x))>0.$

\vspace{1cm}

This assumption is very mild; it requires that the gradient flow moves away from 
the boundary of ascending manifolds. In terms of mode clustering,
this requires the gradient flow to move away from the boundaries of clusters.
For a point $x\in D_{d-1}$, let $v_1(x)$ be the corresponding
normal direction.
Then the gradient $g(x)$ is normal to $v_1(x)$ by definition.
That is, $v_1(x)^T g(x) =v_1(x)^T\nabla p(x)=0$,
which means that the gradient along $v_1(x)$ is $0$.
Assumption (D) means that the 
the second derivative along $v_1(x)$ is positive, which implies
that the density along direction $v_1(x)$ behaves like a local minimum
at point $x$.
Intuitively, this is 
how we expect the density to behave around the boundaries:
gradient flows are moving away from the boundaries 
(except for those flows that are already on the boundaries).

\begin{thm}[Stability of descending $d$-manifolds]
Let $f,\tilde{f}:\mathbb{K}\subset\mathbb{R}^d\mapsto \mathbb{R}$ be two smooth functions with 
bounded third derivatives defined as above and let $B,\tilde{B}$ be the boundaries of the associated 
ascending manifolds. Assume $f$ is a Morse function satisfying condition {\bf (D)}. 
When $\norm{f-\tilde{f}}^*_{3,\max}$ is sufficiently small,
\begin{equation}
\Haus(\tilde{B},B) = O(\norm{\tilde{f}-f}_{1,\max}).
\end{equation}
\label{thm::Haus}
\end{thm}
This theorem shows that the boundaries of descending $d$-manifolds for two Morse functions are close
to each other and the difference between the boundaries is controlled
by the rate of the first derivative difference.

Similarly to descending manifolds, we can define all the analogous quantities
for ascending manifolds.
We introduce the following assumption:

\vspace{1cm}

{\bf Assumption (A):} We assume 
$H_{\max} = \max_{x\in \partial A_0} \lambda_{\max} (H_V(x)) < 0.$

%\ATTNC{Changed to max from min. Right?}

\vspace{1cm}
Note that $\lambda_{\max}(M)$ denotes the largest eigenvalue of a matrix $M$.
If $M$ is a scalar, $\lambda_{\max}(M)=M$.
Under assumption (A), 
we have a similar stability result (Theorem~\ref{thm::Haus}) for ascending manifolds.
Assumptions (A) and (D) together imply the stability of $d$-cells.

Theorem~\ref{thm::Haus} can be applied to nonparametric density estimation.
Our goal is to estimate the boundary of the descending $d$-manifolds, $B$,
of the unknown population density function $p$. % (or its smooth surrogate $p_h$).
Our estimator is $\hat{B}_n$,
the boundary of the descending $d$-manifolds to a nonparametric density estimator 
e.g. the kernel density estimate $\hat{p}_ n$.
Then under certain regularity condition, their difference is given by
$$
\Haus \left(\hat{B}_n, B\right) = O\left(\norm{\hat{p}_n-p}_{1,\max}\right).
$$
We will see this result in the next section when we discuss mode clustering.

Similar reasoning works for the nonparametric regression case.
Assume that we are interested in $B$, the boundary of descending $d$-manifolds, for the regression function
$m(x) = \mathbb{E}(Y|X=x)$. 
And our estimator $\hat{B}$ is again a plug-in estimate based on 
$\hat{m}_n(x)$, a nonparametric regression estimator (e.g., kernel estimator).
Then under mild regularity conditions,
$$
\Haus \left(\hat{B}_n, B\right) = O\left(\norm{\hat{m}_n-m}_{1,\max}\right).
$$

\subsection{Consistency of Mode Clustering}\label{sec::thm::MC}

A direct application of Theorem~\ref{thm::Haus} is the consistency
of mode clustering.
%Note that several literatures have attempted to solve this problem
%but none of them formally derive the clustering consistency in term
%A direct application of Theorem~\ref{thm::Haus} is the consistency
%of mode clustering \citep{Li2007,azzalini2007clustering,
%Chacon2012,arias2016estimation,chacon2015population}.
Let $K^{(\alpha)}$ be the $\alpha$-th derivative of $K$ and let $\mathbf{BC}^r$
denote the collection of functions with bounded continuously derivatives up
to the $r$-th order.
We consider the following two assumptions on the kernel function:
\begin{itemize}
\item[\bf(K1)] The kernel function $K\in\mathbf{BC}^3$ and is symmetric, non-negative and 
$$\int x^2K^{(\alpha)}(x)dx<\infty,\qquad \int \left(K^{(\alpha)}(x)\right)^2dx<\infty
$$ 
for all $\alpha=0,1,2,3$.
\item[\bf(K2)] The kernel function satisfies condition $K_1$ of
  \cite{Gine2002}. That is, there exists some $A,v>0$ such that for
  all $0<\epsilon<1$,
$\sup_Q N(\mathcal{K}, L_2(Q), C_K\epsilon)\leq \left(\frac{A}{\epsilon}\right)^v,$
where $N(T,d,\epsilon)$ is the $\epsilon-$covering number for a
semi-metric space $(T,d)$ and
$$
\mathcal{K} = 
\Biggl\{u\mapsto K^{(\alpha)}\left(\frac{x-u}{h}\right)
: x\in\R^d, h>0,|\alpha|=0,1,2\Biggr\}.
$$
\end{itemize}

(K1) is a common assumption; see \cite{wasserman2006all}.  (K2) is a
weak assumption guarantee the consistency for KDE under $L_{\infty}$
norm; this assumption first appeared in \cite{Gine2002} and has been
widely assumed \citep{Einmahl2005,rinaldo2010generalized,   genovese2012geometry,
rinaldo2012stability,genovese2014nonparametric, chen2015asymptotic}.

\begin{thm}[Consistency for mode clustering]
Let $p,\hat{p}_n$ be the density function and the KDE.
Let $B$ and $\hat{B}_n$ be the boundaries of clusters by 
mode clustering over $p$ and $\hat{p}_n$ respectively.
Assume (D) for $p$ and (K1--2),
%then as $\norm{\hat{p}_n-p}^*_{3,\max}$ is sufficiently small,
then when $\frac{\log n }{nh^{d+6}}\rightarrow 0, h\rightarrow 0$,
$$
\Haus\left(\hat{B}_n,B\right) = O(\norm{\hat{p}_n-p}_{1,\max}) = O(h^2) + O_{\P}\left(\sqrt{\frac{\log (n)}{nh^{d+2}}}\right).
$$
\label{thm::mode}
\end{thm}
The proof is simply to combine Theorem~\ref{thm::Haus} and the rate of convergence 
for estimating the gradient of density using KDE (Theorem~\ref{thm::KDE}). 
Thus,
we omit the proof.
Theorem~\ref{thm::mode} gives a bound for the rate of convergence
for the boundaries for mode clustering.
The rate can be decomposed into two parts, the bias $O(h^2)$
and the (square root of) variance $O_{\P}\left(\sqrt{\frac{\log (n)}{nh^{d+2}}}\right)$.
This rate is the same for the $\cL_{\infty}$-loss of estimating the gradient of a density function,
which makes sense since the mode clustering is completely determined
by the gradient of density.

Another way to describe the consistency for
mode clustering is to show that the proportion
of data points that are \emph{incorrectly clustered (mis-clustered)}
converges to $0$.
This can be quantified by the use of Rand index 
\citep{rand1971objective, hubert1985comparing, vinh2009information},
which measures the similarity between two
partitions of the data points.
Let $\dest(x)$ and $\hat{\dest}_n(x)$ be the destination
of gradient of the true density function $p(x)$ and the KDE $\hat{p}_n(x)$.
For a pair of points $x,y$, we define 
\begin{equation}
%\begin{aligned}
\Psi(x,y)  = \left\{ 
  \begin{array}{l l}
    1& \quad \text{if $\dest(x)=\dest(y)$}\\
    0 & \quad \text{if $\dest(x)\neq \dest(y)$}
  \end{array} \right.
,\quad
\hat{\Psi}_n(x,y)  = \left\{ 
  \begin{array}{l l}
    1& \quad \text{if $\hat{\dest}_n(x)=\hat{\dest}_n(y)$}\\
    0 & \quad \text{if $\hat{\dest}_n(x)\neq \hat{\dest}_n(y)$}
  \end{array} \right.
  \label{eq::rand1}
%\end{aligned}
\end{equation}
Thus, $\Psi(x,y)=1$ if $x,y$ are in the same cluster
and $0$ if they are not.
The Rand index for mode clustering using $p$ versus using $\hat{p}_n$ is 
\begin{equation}
\rand\left(\hat{p}_n,p\right) = 1 - {n \choose 2}^{-1}\sum_{i\neq j}\left|\Psi(X_i,X_j)-\hat{\Psi}_n(X_i,X_j)\right|,
    \label{eq::rand2}
\end{equation}
which is the proportion of pairs of data points
that the two clustering results disagree on.
If two clusterings output the same partition,
the Rand index will be $1$.

\begin{thm}[Bound on Rand Index]
Assume (D) for $p$ and (K1--2).
Then when $\frac{\log n }{nh^{d+6}}\rightarrow 0, h\rightarrow 0$,
the adjusted Rand index
$$
\rand\left(\hat{p}_n,p\right)= 1-O(h^2) - O_{\P}\left(\sqrt{\frac{\log (n)}{nh^{d+2}}}\right).
$$
\label{thm::number}
\end{thm}
Theorem~\ref{thm::number} shows that
the Rand index converges to $1$ in probability,
which establishes the consistency of mode clustering in
an alternative way.
Theorem~\ref{thm::number} shows that the proportion
of data points that are incorrectly assigned 
(compared with mode clustering using population $p$)
is bounded by the rate $O(h^2) + O_{\P}\left(\sqrt{\frac{\log (n)}{nh^{d+2}}}\right)$
asymptotically.

\cite{azizyan2015risk} also derived the convergence rate of the mode clustering for the rand index. 
Here we briefly compare our results to theirs. 
\cite{azizyan2015risk} consider a low-noise condition that
leads to a fast convergence rate when clusters are well-separated.
Their approach can even be applied to the case of increasing dimensions. 
In our case (Theorem~\ref{thm::number}), 
we consider a fixed dimension scenario but we do not assume the low-noise condition. 
Thus, the main difference between
Theorem~\ref{thm::number} and the result in \cite{azizyan2015risk} is the assumptions being made
so our result complements the findings in \cite{azizyan2015risk}.
%In a sense, our assumption is weaker when the dimension is fixed but the convergence rate is slower and 

%{\bf YEN-CHI: You need to reference our paper with Martin and say how the
%results compare.}

\subsection{Consistency of Morse-Smale Regression}\label{sec::thm::MSR}

In what follows, we will show that $\hat{m}_{n, \MSR}(x)$
is a consistent estimator of $m_{\MSR}(x)$.
%Moreover, if we consider estimating the smoothed version of 
%$m(x)$, denoted as $m_h(x) = \mathbb{E}(\hat{m}_n(x))$,
%with fixed smoothing parameter $h$
%and we use MSR to represent $m_h$, denoted as $m_{h, \MSR}$,
%we will obtain a near parametric rate for estimating $m_{h, \MSR}$
%by $\hat{m}_{n,MSR}$.
Recall that
\begin{equation}
m_{\MSR}(x) = \mu_{ \ell} + \beta_{\ell}^T x, \mbox{ for }x \in E_{\ell},
\end{equation}
where $E_{\ell}$ is the $d$-cell defined on $m$
and the parameters are 
\begin{equation}
\begin{aligned}
(\mu_{\ell}, \beta_{\ell})
&= \underset{\mu,\beta}{\sf \argmin} \mbox{ }\mathbb{E}\left((Y- \mu-\beta^TX)^2|X\in E_{\ell}\right).
\end{aligned}
\label{eq::LSE3}
\end{equation}
And $\hat{m}_{n,\MSR}$ is the two-stage estimator to $m_{\MSR}(x)$ defined by
\begin{equation}
\hat{m}_{n, \MSR}(x) = \hat{\mu}_\ell+\hat{\beta}_\ell^Tx, \mbox{ for }x \in \hat{E}_\ell,
\end{equation}
where $\{\hat{E}_\ell: \ell=1,\cdots, \hat{L}\}$ are the collection of cells of the pilot nonparametric regression estimator $\hat{m}_n$
and $\hat{\mu}_\ell, \hat{\beta}_\ell$ are the regression parameters from equation \eqref{eq::LSE2}:
\begin{equation}
\begin{aligned}
(\hat{\mu}_\ell, \hat{\beta}_\ell)= \underset{\mu,\beta}{\sf\argmin} \mbox{ }\sum_{i: X_i\in \hat{E}_\ell}(Y_i- \mu-\beta^TX_i)^2.
\end{aligned}
\end{equation}

\begin{thm}[Consistency of Morse-Smale Regression]
Assume (A) and (D) for $m$ and assume $m$ is a Morse-Smale function.
%Then as $\norm{\hat{m}_n-m}^*_{3,\max}$ is sufficiently small,
Then when $\frac{\log n }{nh^{d+6}}\rightarrow 0, h\rightarrow 0$,
we have
\begin{equation}
\left|m_{\MSR}(x)-\hat{m}_{n,\MSR}(x)\right| = 
O_{\P}\left(\frac{1}{\sqrt{n}}\right) + O\left(\norm{\hat{m}_n-m}_{1,\max}\right)
\label{eq::thm::MSR1}
\end{equation}
uniformly for all $x$ except for a set $\mathbb{N}_n$ with Lebesgue measure 
$O_\P(\norm{\hat{m}_n-m}_{1,\max})$,

%Moreover, if (A,D) holds for the smoothed regression function $m_h$
%(assumed to be Morse-Smale)
%and (K1--2) holds for the kernel function, then uniformly for all $x$
%except for a set with Lebesgue measure $O_{\P}\left(\sqrt{\frac{\log n}{n}}\right)$,
%\begin{equation}
%\left|m_{h,\MSR}(x)-\hat{m}_{n,\MSR}(x)\right| 
%= O_{\P}\left(\sqrt{\frac{\log(n)}{n}}\right)
%\label{eq::thm::MSR2}
%\end{equation}

\label{thm::MSR}
\end{thm}
Theorem~\ref{thm::MSR} states that when we have a consistent pilot nonparametric regression estimator 
(such as the kernel regression), the proposed MSR estimator converges
to the population MSR.
Similarly as in Theorem~\ref{thm::MSS}, the set $\mathbb{N}_n$ are regions around the boundaries of cells
where we cannot distinguish their host cell.
Note that when we use the kernel regression as the pilot estimator $\hat{m}_n$,
Theorem~\ref{thm::MSR} becomes 
$$
\left|m_{\MSR}(x)-\hat{m}_{n,\MSR}(x)\right| = 
O(h^2)+O_{\P}\left(\sqrt{\frac{\log n}{nh^{d+2}}}\right).
$$
under regular smoothness conditions. 

Now we consider a special case where we may obtain parametric rate of convergence
for estimating $m_{\MSR}$.
Let $\mathcal{E} = \partial \left(E_1\bigcup\cdots\bigcup E_L\right)$ be the boundaries
of all cells.
We consider the following low-noise condition:
\begin{equation}
\P\left(X\in \mathcal{E}\oplus \epsilon\right) \leq A \epsilon^\beta,
\label{eq::low_noise}
\end{equation}
for some $A,\beta>0$.
Equation \eqref{eq::low_noise} is Tsybakov's low noise condition \citep{audibert2007fast}
applied to the boundaries of cells.
Namely, \eqref{eq::low_noise} states that it is unlikely to many observations near
the boundaries of cells of $m$.
Under this low-noise condition, we obtain the following result using kernel regression.

\begin{thm}[Fast Rate of Convergence for Morse-Smale Regression]
Let the pilot estimator $\hat{m}_n$ be the kernel regression estimator.
Assume (A) and (D) for $m$ and assume $m$ is a Morse-Smale function.
Assume also \eqref{eq::low_noise} holds for the covariate $X$
and (K1-2) for the kernel function.
Also assume that $h = O\left(\left(\frac{\log n}{n}\right)^{1/(d+6)}\right)$.
Then uniformly for all $x$ except for a set $\mathbb{N}_n$ with Lebesgue measure 
$O_\P\left(\left(\frac{\log n}{n}\right)^{2/(d+6)}\right)$,
\begin{equation}
\left|m_{\MSR}(x)-\hat{m}_{n,\MSR}(x)\right| = 
O_{\P}\left(\frac{1}{\sqrt{n}}\right) + O_{\P}\left(\left(\frac{\log n}{n}\right)^{2\beta/(d+6)}\right).
\label{eq::thm::MSR2}
\end{equation}
Therefore, when $\beta>\frac{6+d}{4}$, we have 
\begin{equation}
\left|m_{\MSR}(x)-\hat{m}_{n,\MSR}(x)\right| = 
O_{\P}\left(\frac{1}{\sqrt{n}}\right).
\label{eq::thm::MSR3}
\end{equation}
\label{thm::MSR2}
\end{thm}

Theorem~\ref{thm::MSR2} shows that when the low noise condition holds,
we obtain a fast rate of convergence for estimating $m_{\MSR}$.
Note that the pilot estimator $\hat{m}_n$ does not ahve to be a kernel estimator;
other approaches such as the local polynomial regression will also work.

\subsection{Consistency of the Morse-Smale Signature}\label{sec::thm::MSS}

%{\bf I SUGGEST CHANGING THIS SECTION TO BE ABOUT THE CONSISTENCY OF THE BIPARTITR GRAPH SIGNATURE}

Another application of Theorem~\ref{thm::Haus} is to
bound the difference of two Morse-Smale signatures.
Let $f$ be a Morse-Smale function with cells
$E_1,\ldots, E_L$.
Recall that the Morse-Smale signatures are the bipartite graph and summary
statistics (locations, density values) for local modes, local minima, and cells.
It is known in the literature (see, e.g., Lemma \ref{lem::critical}) that 
when two functions $\tilde{f},f$ are sufficiently close, then
\begin{equation}
\max_j\|\tilde{c}_j-c_j\| = O\left(\|\tilde{f}-f\|_{1,\max}\right),\quad 
\max_j\|\tilde{f}(\tilde{c}_j)-f(c_j)\| = O\left(\|\tilde{f}-f\|_{\infty}\right),
\label{eq::critical}
\end{equation}
where $\tilde{c}_j, c_j$ are critical points $\tilde{f}$ and $f$ respectively.
This implies the stability of local modes and minima.

So what we need is the stability of the summary statistics 
 $(\eta_\ell^\dagger, \gamma_\ell^\dagger)$ associated with the edges (cells).
Recall that these summaries are defined through \eqref{eq::MSS1}
\begin{equation*}
(\eta_\ell^\dagger, \gamma_\ell^\dagger) = \underset{\eta, \gamma}{\sf argmin} \int_{E_\ell} 
\left(f(x)-\eta-\gamma^Tx\right)^2 dx.
%\label{eq::MSS1}
\end{equation*}
For another function $\tilde{f}$, let
$(\tilde{\eta}_\ell^\dagger, \tilde{\gamma}_\ell^\dagger)$
be its signatures for cell $\tilde{E}_\ell$.
The following theorem shows that if two functions are close,
their
corresponding Morse-Smale signatures
are also close.

\begin{thm}
Let $f$ be a Morse-Smale function satisfying assumptions A and D, and let $\tilde{f}$ be a smooth function.
%Let $f_{\MS}$ and $\tilde{f}_{\MS}$ be the corresponding Morse-Smale approximation functions
%for $f$ and $\tilde{f}$ respectively.
%Then if $\norm{\tilde{f}-f}^*_{3,\max}$ is sufficiently small,
Then when $\frac{\log n }{nh^{d+6}}\rightarrow 0, h\rightarrow 0$,
after relabeling the indices of cells of $\tilde{f}$,
$$
\max_\ell\left\{\|\tilde{\eta}_\ell^\dagger-\eta_\ell^\dagger\|, \|\tilde{\gamma}_\ell^\dagger-\gamma_\ell^\dagger\|\right\} = O\left(\norm{\tilde{f}-f}^*_{1,\max}\right).
$$
\label{thm::MSS}
\end{thm}

%Theorem~\ref{thm::MSS} shows the stability of $f_{\MS}$
%and thus guarantees the stability of the summary statistics for edges ($d$-cells).
Theorem~\ref{thm::MSS} shows stability of the signatures
$(\eta^\dagger_\ell, \gamma^\dagger_\ell)$.
Note that Theorem~\ref{thm::MSS} also implies that the stability of piecewise approximation 
$$
|f_{\MS}(x)-\tilde{f}_{\MS}(x)| = O\left(\norm{\tilde{f}-f}^*_{1,\max}\right).
$$
%The set $\mathbb{N}_n$ are the regions around boundaries of cells.
%The consistency does not hold for these regions since for a point $x\in\mathbb{N}_n$, 
%we cannot accurately distinguish which
%cell $x$ belong to.
Together with the stability of critical points \eqref{eq::critical},
Theorem~\ref{thm::MSS} proves the stability  
of Morse-Smale signatures.
%and
%for the visualization (see e.g. Figure~\ref{Fig::ex::MSvis} and Algorithm \ref{Alg::MSS}).

\subsubsection{Example: Morse-Smale Density Estimation}

As an example for Theorem~\ref{thm::MSS},
we consider density estimation.
Let $p$ be the density of random sample $X_1,\cdots,X_n$
and recall that $\hat{p}_n$ is the kernel density estimator.
Let $(\eta_\ell^\dagger, \gamma_\ell^\dagger)$ be the signature for
$p$ under cell $E_\ell$
and $(\hat{\eta}_\ell^\dagger, \hat{\gamma}_\ell^\dagger)$
be the signature for $\hat{p}_n$ under cell $\hat{E}_\ell$.
The following corollary guarantees the consistency
of Morse-Smale signatures for the KDE.

\begin{cor}
Assume (A,D) holds for $p$ and the kernel function satisfies (K1--2).
Then when $\frac{\log n }{nh^{d+6}}\rightarrow 0, h\rightarrow 0$,
after relabeling we have
$$
\max_\ell\left\{\|\hat{\eta}_\ell^\dagger-\eta_\ell^\dagger\|, 
\|\hat{\gamma}_\ell^\dagger-\gamma_\ell^\dagger\|\right\}
= O(h^2)+O_{\P}\left(\sqrt{\frac{\log n}{nh^{d+2}}}\right).
$$
\label{thm::density}
\end{cor}
The proof to Corollary~\ref{thm::density} is 
a simple application of Theorem~\ref{thm::MSS} with the rate 
of convergence for the first derivative of the KDE (Theorem~\ref{thm::KDE}).
So we omit the proof.
%Corollary~\ref{thm::density} implies that
%the Morse-Smale signatures of a density estimator $\hat{p}_n$
%converges to the signatures of the population density function.
%The rate is a bit slower than the usual rate for estimating a density function
%since estimating the boundaries depends on the 
%derivatives so the rate is the same as the one for estimating the derivatives.
The optimal rate in Corollary~\ref{thm::density} is $O_\P\left(\left(\frac{\log n}{n}\right)^{\frac{2}{d+6}}\right)$
when we choose $h$ to be of order $O\left(\left(\frac{\log n}{n}\right)^{\frac{1}{d+6}}\right)$.

\begin{remark}
When we compute the Morse-Smale approximation function,
we may have some numerical problem in low-density regions
because the density estimate $\hat{p}_n$ may have unbounded support.
In this case, some cells may be unbounded, and the majority of these cells may
have extremely low density value, which makes the approximation function $0$.
Thus, in practice, we will restrict ourselves only to the regions whose
density is above a pre-defined threshold $\lambda$ so that every cell is bounded.
A simple data-driven threshold is $\lambda = 0.05 \sup_{x} \hat{p}_n(x)$.
Note that Theorem~\ref{thm::density} still works in this case 
but with a slight modification:
the cells are define on the regions $\{x: p_h(x)\geq 0.05\times \sup_x p_h(x)\}$.
\end{remark}

\begin{remark}
Note that for a density function, local minima may not exist or
the gradient flow may not lead us to a local minimum in some regions.
For instance, for a Gaussian distribution, there is no local minimum
and except for the center of the Gaussian, if we follow the gradient descent path,
we will move to infinity.
Thus, in this case we only consider the boundaries
of ascending $0$-manifolds corresponding to well-defined local minima
and assumptions (A) is only for the boundaries corresponding 
to these ascending manifolds.
\end{remark}

\begin{remark}
When we apply the Morse-Smale complex to nonparametric density estimation or regression,
we need to choose the tuning parameter.
For instance, in the MSR, we may use kernel regression or local polynomial regression
so we need to choose the smoothing bandwidth.
For the density estimation problem or mode clustering,
we need to choose the smoothing bandwidth for the kernel smoother.
In the case of regression, because we have the response variable,
we would recommend to choose the tuning parameter by cross-validation.
For the kernel density estimator (and mode clustering), 
because the optimal rate depends on the gradient estimation,
we recommend choosing the smoothing bandwidth using the normal reference
rule for gradient estimation or the cross-validation method for gradient estimation \citep{duong2007ks,Chacon2011}. 
\end{remark}

\section{Discussion}

In this paper, we introduced the Morse-Smale complex
and the summary signatures 
for nonparametric inference.
We demonstrated that the Morse-Smale complex can be applied to
various statistical problems such as
clustering, regression and two sample comparisons.
We showed that a smooth multivariate function
can be summarized by a few parameters associated with 
a bipartite graph, representing the local modes, minima and the complex
for the underlying function.
Moreover, we proved
a fundamental theorem about the stability of the Morse-Smale complex.
Based on the stability theorem, we 
derived consistency for mode clustering
and regression.

The Morse-Smale complex provides a method to synthesize 
both parametric and nonparametric inference.
Compared to parametric inference, we have a
more flexible model to study the structure of the underlying distribution. 
Compared to nonparametric inference, the use of the Morse-Smale complex 
yields a visualizable representation
for the underlying multivariate structures.
This reveals that we may gain additional insights in data analysis
by using geometric features.

Although the Morse-Smale complex has many potential statistical applications,
we need to be careful when applying it to a data set whose dimension is large (say $d>10$). 
When the dimension is large, the curse of dimensionality kicks in and the nonparametric 
estimators (in both density estimation problems or regression analysis) are not accurate
so the errors of the estimated Morse-Smale complex can be huge. 

%And we can use the Morse-Smale complex
%to summarize the multivariate function generated from nonparametric inference.

%We showed that a smooth multivariate function
%can be summarized by a few parameters associated with 
%a bipartite graph, representing the local modes, minima and the complex
%for the underlying function.
%The use of geometric features provides a balance between 
%parametric and nonparametric inference.

%The Morse-Smale complex can be applied to clustering,
%density estimation, regression and two sample comparison.
%We showed that a smooth multivariate function
%can be summarized by a few parameters associated with 
%a bipartite graph, representing the local modes, minima and the complex
%for the underlying function.
%\ATTNC{I think we need something stronger here.}

Here we list some possible extensions for future research:
\begin{itemize}
\item \emph{Asymptotic distribution.}
We have proved the consistency (and the rate of convergence)
for estimating the complex but the limiting distribution is still unknown.
If we can derive the limiting distribution and show that some resampling method 
(e.g. the bootstrap \cite{Efron1979}) converges to the same distribution, we can construct
confidence sets for the complex.

\item \emph{Minimax theory.}
Despite the fact that we have derived the rate of convergence
for a plug-in estimator for the complex,
we did not prove its optimality.
We conjecture the minimax rate for estimating the complex
should be related to the rate for estimating the gradient
and the smoothness around complex \citep{audibert2007fast, singh2009adaptive}.

\end{itemize}

\section*{Acknowledgement}
We thank the referees and the Associate Editor for their very constructive comments and suggestions.

%%%%%%%%%%%%%%%%%%%%%%%%%%%%%%%%%%%%%%%%%%%%%%%%%%%%%%%
\appendix
\section{Appendix: Proofs}

First, we include a Theorem about the rate of convergence for 
the kernel density estimator. This Lemma will be used in deriving the convergence rates.
\begin{thm}
[Lemma 10 in \cite{chen2015asymptotic}; see also \cite{genovese2014nonparametric}]
Assume (K1--2) and that $\log n/n\leq h^d \leq b$ for some $0<b<1$.
Then we have 
\begin{align*}
\norm{\hat{p}_n-p}^*_{\ell, \max} &= O(h^2) + O_{\P}\left(\sqrt{\frac{\log n}{nh^{d+2\ell}}}\right)%\\
%\norm{\hat{p}_n-p_h}^*_{\ell, \max} &=  O_{\P}\left(\sqrt{\frac{\log n}{nh^{d+2\ell}}}\right)
\end{align*}
for $\ell=0,1,2$.
\label{thm::KDE}
\end{thm}

To prove Theorem~\ref{thm::Haus},
we introduce the following useful Lemma for stability
of critical points.
\begin{lem}
[Lemma 16 of \cite{chazal2014robust}]
Let $p$ be a density with compact support $\K$ of $\R^d$. 
Assume $p$ is a Morse function with
finitely many, distinct, critical values with corresponding critical points 
$C = \{c_1,\cdots, c_k\}$. Also assume that $p$
is at least twice differentiable
on the interior of $\K$, continuous and differentiable with non vanishing gradient on the
boundary of $\K$. Then
there exists $\epsilon_0 > 0$ such that for all $0<\epsilon<\epsilon_0$ 
the following is true: for some positive
constant $c$, there exists $\eta\geq c\epsilon_0$ such that, for any density $q$ 
with support $\K$ satisfying
$\norm{p-q}^*_{2,\max}\leq \eta$,
we have 
\begin{itemize}
\item[1.] $q$ is a Morse function with exact $k$ critical points $c'_1,\cdots,c'_k$ and
\item[2.] after suitable relabeling the indices, $\max_{i=1,\cdots, k}\norm{c_i-c'_i}\leq \epsilon$.
\end{itemize}
\label{lem::critical}
\end{lem}
Note that similar result appears in Theorem 1 of \cite{chen2016comprehensive}.
This lemma shows that two close Morse functions $p,q$ will have similar
critical points.

%Basically, this lemma shows that when for a Morse function $p$
%defined on a compact set $\K$,
%when another smooth function $q$ that is sufficiently close to $p$,
%$q$ is also a Morse function and the critical points of $p$ and
%the critical points of $q$ are very close to each other.

\begin{figure}
\center
\includegraphics[scale=0.4]{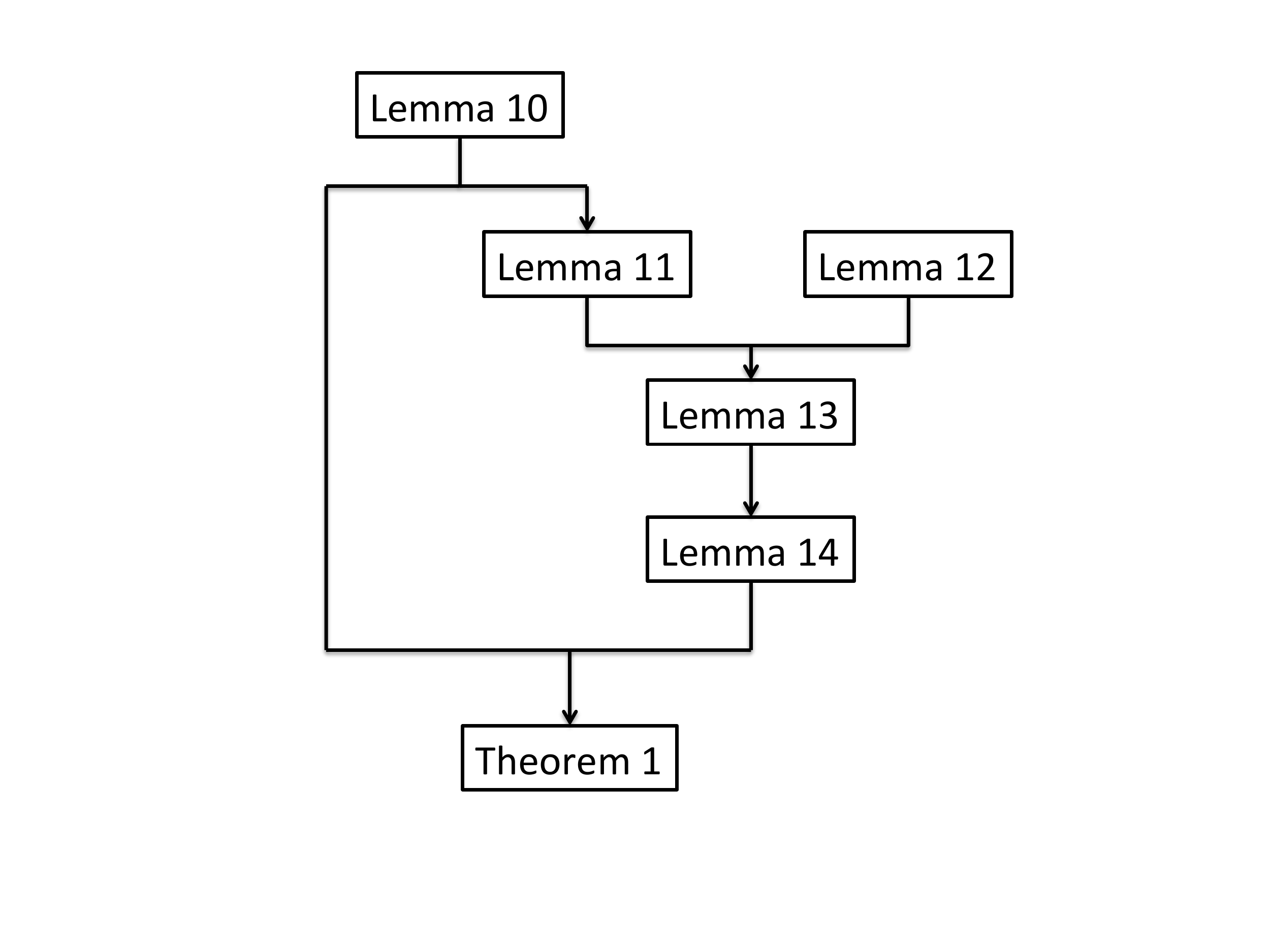}
\caption{Diagram for lemmas and Theorem \ref{thm::Haus}.
}
\label{Fig::diagram}
\end{figure}

The proof of Theorem~\ref{thm::Haus} requires several working lemmas.
We provide a chart for how we are going to prove Theorem~\ref{thm::Haus}.

%To proof this Theorem~\ref{thm::Haus}, we need several working lemmas.
First, we define some notations about gradient flows.
%The above result can be used to bound the minimal gradient for points near to boundaries $D$,
%which will be used in Larry's theorems.
Recall that $\pi_x(t)\in\mathbb{K}$ is the gradient (ascent) flow starting at $x$:
$$
\pi_x(0) = x, \quad \pi_x'(t)= g(\pi_x(t)).
$$
For $x$ that is not on the boundary set $D$, we define the time:
$$
t_\epsilon(x) = \inf \{t: \pi_x(s)\in B(m, \sqrt{\epsilon}),\mbox{ for all} s\geq t\},
$$
where $m$ is the destination of $\pi_x$. %i.e. $m=\lim_{t\rightarrow\infty}\pi_x(t)$, which 
%is a local mode (we assume $x$ is not on $D$, the boundaries).
That is, $t_\epsilon(x)$ is the time to arrive the regions around a local mode.

First, we prove a property for the direction of the gradient field around boundaries.
\begin{lem}[Gradient field and boundaries]
%Assumption condition {\bf (D)}.
Assume the notations in Theorem~\ref{thm::Haus}
and assume $f$ is a Morse function with bounded third derivatives and satisfies 
assumption {\bf (D)}.
Let $s(x) = x-\Pi_x$, where $\Pi_x\in B$
is the projected point from $x$ onto $B$
(when $\Pi_x$ is not unique, just pick any projected point).
For any $q\in B$, let $x$ be a point near $q$ such that 
$x-q\in \mathbb{V}(q)$, the normal space of $B$ at $q$.
Let $\delta(x) = \norm{x-q}$ and
$e(x) = \frac{x-q}{\|x-q\|}$ denote the unit vector.
Then
\begin{enumerate}
\item For every point $x$ such that 
$$
d(x,B)\leq\delta_1 = \frac{2H_{\min}}{d^2\cdot\norm{f}_{3,\max}},
$$
we have
$$
g(x)^T s(x)\geq 0.
$$
That is, the gradient is pushing $x$ away from the boundaries.
\item When $\delta(x)\leq \frac{H_{\min}}{d^2\cdot \norm{f}_{3,\max}}$,
$$
\ell(x) = e(x)^T g(x) \geq \frac{1}{2}H_{\min} \delta(x).$$
\end{enumerate}

\label{lem::Gdist}
\end{lem}

\begin{figure}
\center
\subfigure[Lemma~\ref{lem::Gdist}]{
	\includegraphics[scale=0.7]{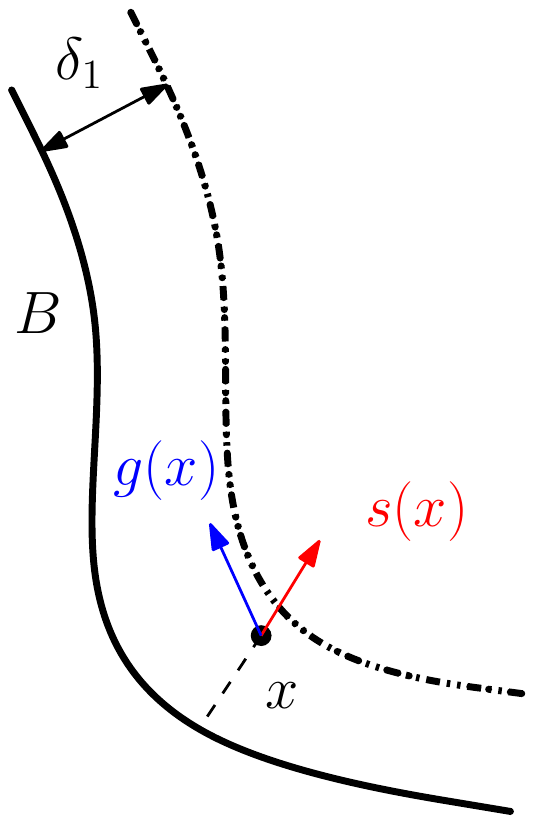}
}
\subfigure[Lemma~\ref{lem::dist}]{
	\includegraphics[scale=0.7]{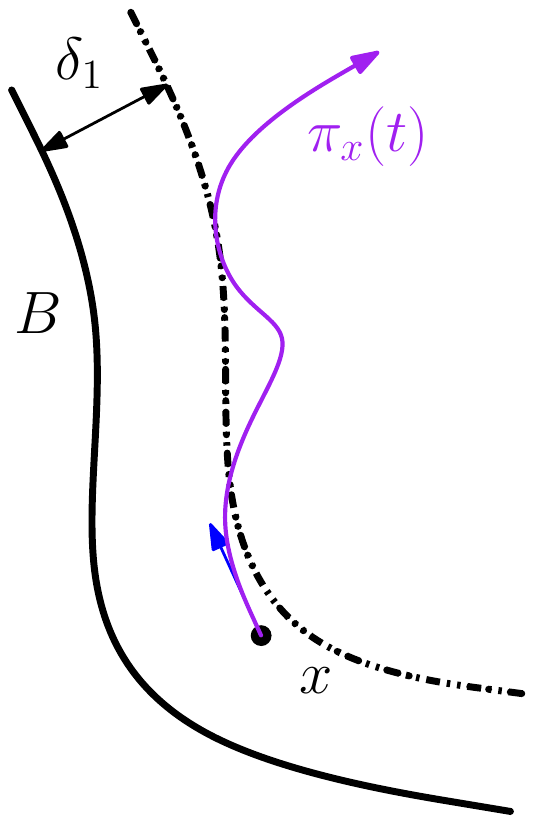}
}
\caption{Illustration for Lemma \ref{lem::Gdist} and \ref{lem::dist}.
(a): We show that the angle between projection vector $s(x)$ 
and the gradient $g(x)$ is always right whenever $x$
is closed to the boundaries $B$. 
(b): According to (a), any gradient flow line start from a point $x$
that is close to the boundaries (distance $<\delta_1$), this flow line is always moving away
from the boundaries when the current location is close to the boundaries.
The flow line can temporally get closer to the boundaries when it is away
from boundaries (distance $>\delta_1$)}
\label{Fig::ex_MG}
\end{figure}

\begin{proof}

%We define the vector $s(x) = x-\Pi_x$, the projection vector pointing `outward'
%from the boundaries.
{\bf Claim 1.}
Because the projection of $x$ onto $B$ is $\Pi_x$, $s(x) \in \mathbb{V}(\Pi_x)$
and $s(x)^Tg(\Pi_x)=0$ (recall that for $p\in B$, $\mathbb{V}(p)$ is the collection of normal
vectors of $B$ at $p$). 

Recall that $d(x,B) = \norm{s(x)}$ is the projected distance.
By the fact that $s(x)^Tg(\Pi_x)=0$,
\begin{equation}
\begin{aligned}
s(x)^T g(x) &= s(x)^T (g(x)-g(\Pi_x))\\
&\geq s(x)^T H(\Pi_x) s(x) - \frac{d^2}{2}\norm{f}_{3,\max}d(x,B)^3 \quad \mbox{(Taylor's theorem)}\\
& = d(x,B)^2 \frac{s(x)^T}{d(x,B)}H(\Pi_x)\frac{s(x)}{d(x,B)} - \frac{d^2}{2}\norm{f}_{3,\max}d(x,B)^3\\
&\geq d(x,B)^2 \left(H_{\min}- \frac{d^2}{2}\norm{f}_{3,\max}d(x,B)\right).
\end{aligned}
\end{equation}
Note that we use the vector-value Taylor's theorem in the first inequality and
the fact that for two close points $x,y$, the difference in the $j$-the element of gradient $g_j(x)-g_j(y)$ 
has the following expansion
\begin{align*}
g_j(x)-g_j(y) &= H_j(y)^T(x-y) + \int_{u=y}^x (u-y) {\sf T}_j(u) du\\
&\geq H_j(y)^T(x-y) - \frac{1}{2}\sup_u\|{\sf T}_j(u)\|_2 \|x-y\|^2\\
&\geq H_j(y)^T(x-y) - \frac{d^2}{2} \norm{f}_{3,\max} \|x-y\|^2,
\end{align*}
where $H_j(y) = \nabla g_j(y)$ 
and ${\sf T}_j(y) = \nabla \nabla g_j(y)$ is the Hessian matrix of $g_j(y)$, whose elements are the third derivatives of $f(y)$.

Thus, when $d(x,B)\leq  \frac{2H_{\min}}{d^2\cdot\norm{f}_{3,\max}}$,
$s(x)^T g(x) \geq0$, which proves the first claim.
%We have completed our proof.
%That is, the gradient at $x$ will not move closer toward the boundaries
%whenever $x$ is close to the boundaries.

{\bf Claim 2.}
By definition, $e(x)^T g(q)=0$ because $g(q)$ is in tangent space of $B$ at $q$ and $e(x)$
is in the normal space of $B$ at $q$.
Thus, 
\begin{equation}
\begin{aligned}
\ell(x) &= e(x)^Tg(x)\\
&= e(x)^T \left(g(x) - g(q)\right)\\
& \geq e(x)^TH(q)(x-q) - \frac{d^2}{2}\norm{f}_{3,\max}\norm{x-q}^2 \\
& = e(x)^TH(\pi(x))e(x)\delta(x) - \frac{d^2}{2}\norm{f}_{3,\max}\delta(x)^2 \\
&\geq \frac{1}{2}H_{\min} \delta(x)
\end{aligned}
\label{eq::LG1}
\end{equation}
whenever $\delta(x) = \norm{x-q}\leq  \frac{H_{\min}}{d^2\cdot\norm{f}_{3,\max}}$.
Note that in the first inequality we use the same lower bound as the one in claim 1.
Also note that $x-q = e(x)\delta(x)$ and $e(x)$ is in the normal space of $B$ at $\pi(x)$
so the third inequality follows from assumption{ \bf(D)}.

\end{proof}

Lemma~\ref{lem::Gdist} can be used to prove the following result.
\begin{lem}[Distance between flows and boundaries]
Assume the notations as the above
and
assumption {\bf (D)}.
Then for all $x$ such that $0<d(x, B)=\delta\leq \delta_1 = \frac{2H_{\min}}{d^2\norm{f}_{3,\max}}$,
$$
d(\pi_x(t), B) \geq \delta,
$$
for all $t\geq0$.
\label{lem::dist}
\end{lem}
The main idea is that the projected gradient (gradient projected to the normal space
of nearby boundaries) is always positive.
This means that the flow cannot move ``closer'' to the boundaries.

\begin{proof}
By Lemma~\ref{lem::Gdist}, for every point $x$ near to the boundaries ($d(x,B)<\delta_1$),
the gradient is moving this point away from the boundaries.
Thus, for any flow $\pi_x(t)$, once it touches the region
$B\oplus \delta_1,$
it will move away from this region.
So when a flow leaves $B\oplus \delta_1$, it can never come back.

Therefore, the only case that a flow can be within the region $B\oplus \delta_1$
is that it starts at some $x\in B\oplus \delta_1$. 
i.e. $d(x,B)<\delta_1$.

Now consider a flow start at $x$ such that $0<d(x,B)\leq \delta_1$.
By Lemma \ref{lem::Gdist}, the gradient $g(x)$ leads $x$ to move away
from the boundaries $B$.
Thus, whenever $\pi_x(t)\in B\oplus \delta_1$,
the gradient is pushing $\pi_x(t)$ away from $B$.
As a result, the time that $\pi_x(t)$ is closest to $B$ is at the beginning of the flow
.i.e. $t=0$.
This implies that $d(\pi_x(t),B) \geq d(\pi_x(0),B)=d(x,B) = \delta$.

\end{proof}

With Lemma~\ref{lem::dist}, we are able to bound the low gradient regions
since the flow cannot move infinitely close to critical points except its destination.
Let $\lambda_{\min}>0$ be the minimal `absolute' value of eigenvalues
of all critical points.

\begin{lem}[Bounds on low gradient regions]
Assume the density function $f$ is a Morse function and has bounded third derivatives.
Let $\mathcal{C}$ denote the collection of all critical points
and let $\lambda_{\min}$ is the minimal `absolute' eigenvalue for
Hessian matrix $H(x)$ evaluated at $x\in\mathcal{C}$.
Then there exists a constant $\delta_2>0$ such that
\begin{equation}
G(\delta) \equiv \left\{x: \norm{g(x)}\leq \frac{\lambda_{\min}}{2}\delta \right\} \subset \mathcal{C}\oplus \delta
\end{equation}
for every $\delta\leq\delta_2$.
\label{lem::eigen}
\end{lem}
\begin{proof}
Because the support $\K$ is compact and $x\in \K \mapsto \|g(x)\|$ is continuous, 
for any $g_0>0$ sufficiently small, there exists a constant $R(g_0)>0$ such that
$$
G_1(g_0) \equiv \left\{x: \norm{g(x)}\leq g_0 \right\} \subset \mathcal{C}\oplus R(g_0)
$$
and when $g_0\rightarrow 0$, $R(g_0)\rightarrow0$.
Thus, there is a constant $g_1>0$ such that $R(g_1) = \frac{\lambda_{\min}}{2d^3\norm{f}_{3,\max}}$.

The set $\mathcal{C}\oplus \frac{\lambda_{\min}}{2\norm{f}_{3,\max}}$ has a useful feature:
for any $x\in \mathcal{C}\oplus \frac{\lambda_{\min}}{2\norm{f}_{3,\max}}$, 
\begin{align*}
\|H(x) - H(c)\|_{F} &= \|(x-c) f^{(3)}(c+t(x-c))dt\|_{F}\\
&\leq d^3\|x-c\| \norm{f}_{3,\max}\\
& \leq d^3\frac{\lambda_{\min}}{2d^3\norm{f}_{3,\max}} \cdot \norm{f}_{3,\max}\\
& = \frac{\lambda_{\min}}{2},
\end{align*}
where $f^{(3)}$ is a $d\times d\times d$ array of the third derivative of $f$ and $\|A\|_F$
is the Frobenius norm of the matrix A. 
By Hoffman--Wielandt theorem (see, e.g., page 165 of \citealt{Bhatia1997}),
the eigenvalues between $H(x)$ and $H(c)$ is bounded by $\|H(x) - H(c)\|_{F}$.
Therefore, 
the smallest eigenvalue of $H(x)$ must be greater than or equal to the smallest eigenvalue of $H(c)$ minus
$\frac{\lambda_{\min}}{2}$.
Because $\lambda_{\min}$ is the smallest absolute eigenvalues of $H(c)$ for all $c\in\cC$,
the smallest eigenvalue of $H(x)$ is greater than or equal to $\frac{\lambda_{\min}}{2}$, for all 
$x\in \mathcal{C}\oplus R(g_1) = \mathcal{C}\oplus \frac{\lambda_{\min}}{2d^3\norm{f}_{3,\max}}$.

%We first focus on the region $\mathcal{C}\oplus R(g_0)$.
%For any $x\in \mathcal{C}\oplus R(g_0)$, there exists $c\in\mathcal{C}$ such that $\|x-c\|$.
Using the above feature and the fact that $G_1(g_1) \subset \mathcal{C}\oplus \frac{\lambda_{\min}}{2d^3\norm{f}_{3,\max}}$,
for any $x\in G_1(g_1)$, 
%For any $x \in G_1(g_1)$, $x\in \mathcal{C}\oplus \frac{\lambda_{\min}}{2d^3\norm{f}_{3,\max}}$
%by definition. 
%Thus, $\|g(x)\|\leq g_1$ and there exists a $c\in\mathcal{C}$ such that $\|x-c\|\leq R(g_1) = \frac{2g_0}{\lambda_{\min}}$. 
we have the following inequalities:
\begin{align*}
g_1 &\geq \|g(x)\|\\
& = \left\|\int_0^1 (x-c)H(c+t(x-c))dt\right\|\\
& \geq \|x-c\| \frac{1}{2}\lambda_{\min}.
\end{align*}
Thus, $\|x-c\|\leq\frac{2 g_1}{\lambda_{\min}}$,
which implies
$$
G_1(g_1) \subset \mathcal{C}\oplus \frac{2 g_1}{\lambda_{\min}}.
$$
Moreover, because $G_1(g_2)\subset G_1(g_3) $ for any $g_2\leq g_3$,
any $g_2\leq g_1$ satisfies
$$
G_1(g_2)\subset \mathcal{C}\oplus \frac{2 g_2}{\lambda_{\min}}. 
$$
Now pick $\delta = \frac{2 g_2}{\lambda_{\min}} $, we conclude
$$
G_1\left(\frac{\lambda_{\min}}{2\delta}\right) =G(\delta) \subset \mathcal{C}\oplus \delta
$$
for all 
\begin{equation}
\delta = \frac{2 g_2}{\lambda_{\min}} \leq \frac{2 g_1}{\lambda_{\min}} = \delta_2,
\label{eq::delta2}
\end{equation}
where 
$g_1$ is the constant such that $R(g_1) = \frac{\lambda_{\min}}{2d^3\norm{f}_{3,\max}}$.

\end{proof}

\begin{lem}[Bounds on gradient flow]
Using the notations above
and assumption {\bf (D)},
let $\delta_1$ be defined in Lemma~\ref{lem::dist} and $\delta_2$ be defined in Lemma \ref{lem::eigen}, equation \eqref{eq::delta2}.
%and we further define
%\begin{equation}
%\delta_3 = \frac{\lambda_{\min}}{2\norm{f}_{3,\max}}.
%\end{equation}
Then for all $x$ such that
$$
d(x, B)=\delta<\delta_0 = \min\left\{\delta_1,\delta_2, \frac{H_{\min}}{d^2\cdot \norm{f}_{3,\max}}\right\},
$$
and picking $\epsilon$ such that $\delta_2>\epsilon^2> \delta$,
we have
$$
\eta_\epsilon(x)\equiv \underset{0\leq t\leq t_\epsilon(x)}{\inf}\norm{g(\pi_x(t))} \geq \delta 
\frac{\lambda_{\min}}{2}.
$$
Moreover,
$$
\gamma_\epsilon(\delta)\equiv  \underset{x\in B_\delta}{\inf}\eta_\epsilon(x) \geq \delta \frac{\lambda_{\min}}{2},
$$
where $B_\delta = \{x: d(x,B) = \delta\}$.
\label{lem::grad}
\end{lem}

\begin{figure}
\center
\includegraphics{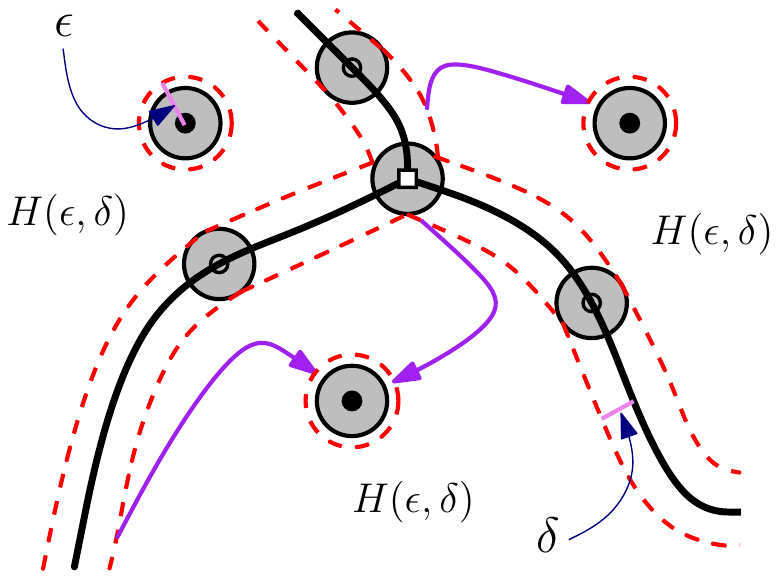}
\caption{Illustration for $\mathcal{H}(\epsilon, \delta)$.
The thick black lines are boundaries $B$; solid dots are local modes;
box is local minimum; empty dots are saddle points.
The three purple lines denote possible gradient flows starting
from some points $x$ with $d(x,B)= \delta$.
The gray disks denote all possible regions such that 
$\norm{g}\leq \frac{\lambda_{\min}}{2}\delta$.
Thus, the amount of gradient within the set $\mathcal{H}(\epsilon, \delta)$
is greater or equal to $\frac{\lambda_{\min}}{2}\delta$.
}
\label{Fig::ex_MG6}
\end{figure}

\begin{proof}
%The proof is a simple application of Lemma~\ref{lem::dist}.

We consider the flow $\pi_x$ starting at
$x$ (not on the boundaries) such that
$$
d(x, B)=\delta<\min\left\{\delta_1,\delta_2\right\}.
$$

For $0\leq t\leq t_\epsilon(x)$, the entire flow is within the set
\begin{equation}
\mathcal{H}(\epsilon, \delta) =\{x: d(x,B)\geq \delta, d(x, M)\geq \sqrt{\epsilon}\}.
\label{eq::H}
\end{equation}
That is,
\begin{equation}
\{\pi_x(t): 0\leq t\leq t_\epsilon(x)\} \subset \mathcal{H}(\epsilon, \delta).
\label{eq::FH}
\end{equation}
This is because by Lemma~\ref{lem::dist}, the flow line cannot get closer
to the boundaries $B$ within distance $\delta$,
and the flow stops when its distance to its destination is at $\epsilon$.
Thus, if we can prove that every point within $\mathcal{H}(\epsilon, \delta)$
has gradient lowered bounded by $\delta \frac{\lambda_{\min}}{2}$,
we have completed the proof.
%Note that we have chosen $\epsilon\geq \delta^2$.
That is, we want to show that 
\begin{equation}
\underset{x\in \mathcal{H}(\epsilon, \delta)}{\inf}\norm{g(x)} \geq \delta \frac{\lambda_{\min}}{2}.
\label{eq::H2}
\end{equation}

To show the lower bound, we focus on those points
whose gradient is small.
Let 
$$
S(\delta) = \left\{x: \norm{g(x)}\leq \delta \frac{\lambda_{\min}}{2}\right\}.
$$
By Lemma~\ref{lem::eigen}, the $S(\delta)$ are regions
around critical points such that
$$ 
S(\delta) \subset \mathcal{C}\oplus \delta.
$$
%since the change in gradient depends on the Hessian matrix
%and the absolute value of eigenvalues of $H(x)$ for $x\in\mathcal{C}\oplus \delta$
%are lower bounded by $\frac{\lambda_{\min}}{2}$
%(by the assumption, $\delta < \frac{\lambda_{\min}}{2\norm{f}_{3,\max}}$ so that
%the changes in second derivative is bounded by $\delta \norm{f}_{3,\max}$,
%which is $\frac{\lambda_{\min}}{2}$).

Since we have chosen $\epsilon$ such that $\epsilon\geq\delta^2$ 
and by the fact that critical points are either in $M$, the collection of all local modes,
or in $B$ the boundaries so that, 
the minimal distance between $\mathcal{H}(\epsilon, \delta)$
and critical points $\mathcal{C}$ is greater that $\delta$ (see equation \eqref{eq::H} for 
the definition of $\mathcal{H}(\epsilon, \delta)$). 
Thus,
$$
(\mathcal{C}\oplus \delta) \cap\mathcal{H}(\epsilon, \delta) = \emptyset,
$$
which implies equation \eqref{eq::H2}:
$$
\underset{x\in \mathcal{H}(\epsilon, \delta)}{\inf}\norm{g(x)} \geq \delta \frac{\lambda_{\min}}{2}.
$$
Now by the fact that all $\pi_x(t)$ with $d(x, B)<\delta$
are within the set $ \mathcal{H}(\epsilon, \delta)$ (equation \eqref{eq::FH}),
we conclude the result.

\end{proof}

\begin{figure}
\center
\includegraphics{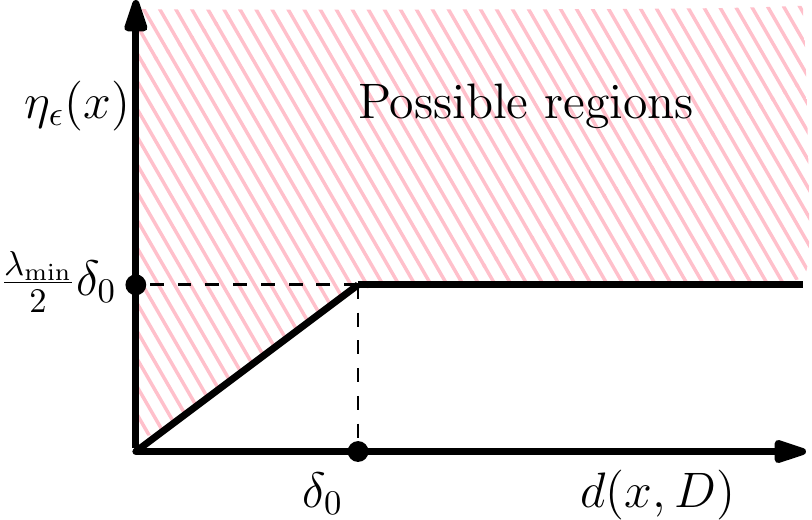}
\caption{Result from Lemma~\ref{lem::grad}: lower bound on minimal gradient. 
This plot shows possible values for minimal gradient $\eta_{\epsilon}(x)$ 
(pink regions) when
$d(x,B)$ is known. Note that we have chosen
$\epsilon^2<\delta_2$.
}
\label{Fig::ex_MGP1}
\end{figure}

%Since the flow cannot move closer to the boundaries
%and assumption (G) ensures that
%the low gradient only occurs around critical points
%(which is either local modes or on the boundaries $B$),
%bounds on distance to boundary automatically bounds
%the minimal gradient.
%Note that the assumption (G) is very weak.
%We can always pick a smaller $\delta_0$ so that 
%only points near critical points have low gradient.

Lemma \ref{lem::grad} links the constant $\gamma_\epsilon(\delta)$ and 
the minimal gradient, which can be used to bound
the time $t_\epsilon(x)$ uniformly and further leads to the following result.
%Thus, we have the following Lemma.
\begin{lem}
Let $\mathbb{K}(\delta) = \{x\in\K: d(x,B)\geq \delta\}=\mathbb{K}\backslash (B\oplus\delta)$
and $\delta_0$ be defined as Lemma~\ref{lem::grad} and $M$ is the collection
of all local modes.
Assume that $f$ has bounded third derivative and is a Morse function and that assumption (D) holds.
Let $\tilde{f}$ be another smooth function.
There exists constants $c_*,c_0,c_1, \epsilon_0$ that all depend
only on $f$ such that
when $(\epsilon, \delta)$ satisfy the following condition
\begin{equation}
\delta<\epsilon<\epsilon_0,\quad \delta<\min\{\delta_0,\Haus(\mathbb{K}(\delta), B(M, \sqrt{\epsilon}))\}
\label{eq::de}
\end{equation}
and if
\begin{equation}
\begin{aligned}
\norm{f-\tilde{f}}^*_{3,\max}&\leq c_0\\
\norm{f- \tilde{f}}_{1,\max}&\leq c_1
\exp\left(-\frac{4\sqrt{d}\norm{f}_{2,\max} \norm{f}_{\infty}}{\delta^2 \lambda^2_{\min}}\right),
%&\norm{f- \tilde{f}}_{\max}\leq c_0, 
%\quad 
%\norm{f- \tilde{f}}_{1,\max}\leq c_1
%\exp\left(-\frac{4\sqrt{d}\kappa_2 \norm{f}_{\max}}{\delta^2 \lambda^2_{\min}}\right),
%\quad \\
%&\norm{f- \tilde{f}}_{2,\max}\leq c_2,
%\quad
%\norm{f- \tilde{f}}_{3,\max}\leq c_3
\label{eq::CCrate1}
\end{aligned}
\end{equation}
then for all $x\in \mathbb{K}(\delta)$
\begin{equation}
\norm{\lim_{t\rightarrow\infty}\pi_x(t) - \lim_{t\rightarrow\infty}\tilde{\pi}_x(t)}\leq 
%c_*\sqrt{\epsilon+\norm{f- \tilde{f}}^*_{1,\max}}.
c_*\sqrt{\norm{f- \tilde{f}}_{\infty}}.
\end{equation}
\label{lem::destination}
\end{lem}
Note that condition \eqref{eq::de} holds when $(\epsilon, \delta)$ are sufficiently small.

\begin{proof}
The proof of this lemma is closely related to the proof of Theorem 2 of \cite{arias2016estimation}.
The results in \cite{arias2016estimation} is a pointwise convergence of gradient flows; now
we will generalize their findings to the uniform convergence. 
%The proof of this lemma basically follows the proof of Theorem 2 of \cite{arias2016estimation}
%with some modification since they only prove the point wise convergence
%and now we extend it to uniform convergence within $\K(\delta)$.

Note that $\K(\delta) = \mathcal{H}(\epsilon, \delta) \cup B(x, \sqrt{\epsilon})$.
For $x\in B(x, \sqrt{\epsilon})$, the result is trivial when $\epsilon$ is sufficiently small.
Thus, we assume $x\in \mathcal{H}(\epsilon, \delta) $.

From equation (40--44) in \cite{arias2016estimation} (proof to their Theorem 2),
\begin{equation}
\begin{aligned}
\norm{\lim_{t\rightarrow\infty}&\pi_x(t) - \lim_{t\rightarrow\infty}\tilde{\pi}_x(t)}\\
&\leq
\sqrt{\frac{2}{\lambda_{\min}} \left(2\lambda_{\min}\epsilon + \frac{\norm{f}_{1,\max}}{\sqrt{d}\norm{f}_{2, \max}} \norm{f-\tilde{f}}_{1,\max}e^{\sqrt{d} \norm{f}_{2,\max} t_\epsilon(x)}+2\norm{f-\tilde{f}}_{\infty}\right) }
\end{aligned}
\label{eq::Crate}
\end{equation}
under condition \eqref{eq::CCrate1} and $\epsilon<\epsilon_0$ for some constant $\epsilon_0$.

Thus, the key is to bound $t_\epsilon(x)$.
Recall that $x\in \mathcal{H}(\epsilon, \delta)$.
Now consider the gradient flow $\pi_x$
and define $z= \pi_x(t_\epsilon(x))$.
\begin{equation}
\begin{aligned}
f(z)- f(x) &= \int_{0}^{t_\epsilon(x)} \frac{\partial f(\pi_x(s))}{\partial s} ds = 
\int_{0}^{t_\epsilon(x)} g(\pi_x(s))^T \pi'_x(s)ds\\ 
&= 
\int_{0}^{t_\epsilon(x)}\norm{g(\pi_x(s))}^2 ds
\geq \gamma_\epsilon(\delta)^2 t_\epsilon(x).
\end{aligned}
\end{equation}
%Note that we use Lemma~\ref{lem::grad} in the last inequality.
Since $f(z)-f(x)\leq 2\norm{f}_{\infty}$, we have
$$
\norm{f}_{\infty} \geq \frac{1}{2}\gamma_\epsilon(\delta)^2 t_\epsilon(x)
$$
and by Lemma~\ref{lem::grad},
\begin{equation}
t_\epsilon(x) \leq \frac{2\norm{f}_{\infty}}{\gamma_\epsilon(\delta)^2}
\leq\frac{8\norm{f}_{\infty}}{\delta^2\lambda^2_{\min}}
\label{eq::Ct}
\end{equation}
for all $x \in \mathcal{H}(\epsilon, \delta)$.

Now plug-in \eqref{eq::Ct} into \eqref{eq::Crate}, we have
\begin{equation}
\norm{\lim_{t\rightarrow\infty}\pi_x(t) - \lim_{t\rightarrow\infty}\tilde{\pi}_x(t)}
\leq
\sqrt{a_0\epsilon + a_1\norm{f-\tilde{f}}_{1,\max}e^{\sqrt{d} \norm{f}_{2,\max} \frac{8\norm{f}_{\infty}}{\delta^2\lambda^2_{\min}}}+a_2\norm{f-\tilde{f}}_{\infty}}
\end{equation}
for some constants $a_0,a_1,a_2$.
Now using condition \eqref{eq::CCrate1} to replace the second
term of right hand side, we conclude 
$$
\norm{\lim_{t\rightarrow\infty}\pi_x(t) - \lim_{t\rightarrow\infty}\tilde{\pi}_x(t)}
\leq
a_3\sqrt{\epsilon+\norm{f- \tilde{f}}^*_{1,\max}}
$$
for some constant $a_3$.

By Lemma 7 in \cite{arias2016estimation}, 
there exists some constant $c_3$ such that
when $a_3\sqrt{\epsilon+\norm{f- \tilde{f}}^*_{1,\max}}< 1/c_3$,
$$
\norm{\lim_{t\rightarrow\infty}\pi_x(t) - \lim_{t\rightarrow\infty}\tilde{\pi}_x(t)}
\leq \sqrt{2}c_3 \norm{f-\tilde{f}}.
$$
Thus, when both $\epsilon$ and $\norm{f-\tilde{f}^*_{3,\max}}$
are sufficiently small, 
there exists some constant $c_*$ such that 
$$
\norm{\lim_{t\rightarrow\infty}\pi_x(t) - \lim_{t\rightarrow\infty}\tilde{\pi}_x(t)}
\leq c_* \norm{f-\tilde{f}}
$$
for all $x\in \mathcal{H}(\epsilon,\delta)$.

\end{proof}
Now we turn to the proof of Theorem~\ref{thm::Haus}.

\begin{proof}[ of Theorem~\ref{thm::Haus}]
The proof contains two parts. 
In the first part,
we show that when 
$
\norm{f-\tilde{f}}^*_{3,\max}
$
is sufficiently small,
we have $\Haus(B,\tilde{B})< \frac{H_{\min}}{d^2\norm{f}_{3,\max}}$,
where $B$ and $\tilde{B}$ are the boundary of descending
$d$-manifolds for $f$ and $\tilde{f}$.
The second part of the proof is to derive the convergence rate.
Because $\Haus(B,\tilde{B})< \frac{H_{\min}}{d^2\norm{f}_{3,\max}}$,
we can apply the second assertion of Lemma~\ref{lem::Gdist} to 
derive the rate of convergence.
Note that $\cC$ and $\tilde{\cC}$ are the critical points for $f$ and $\tilde{f}$
and $M\equiv C_0$, $\tilde{M}\equiv\tilde{C}_0$ are the local modes for $f$ and $\tilde{f}$.

{\bf Part 1: $\Haus(B,\tilde{B})< \frac{H_{\min}}{d^2\cdot\norm{f}_{3,\max}}$,
the upper bound for Hausdorff distance.}
Let $\sigma = \min\{\norm{x-y}: x, y\in M , x\neq y\}$.
That is, $\sigma$ is the smallest distance between any pair of distinct local modes.
By Lemma~\ref{lem::critical}, when
$ \norm{f-\tilde{f}}^*_{3,\max} $
is small, $f$ and $\tilde{f}$ have the same number
of critical points and
$$
\Haus(\cC, \tilde{\cC}) \leq A \norm{f- \tilde{f}}^*_{2,\max}\leq A \norm{f-\tilde{f}}^*_{3,\max},
$$
where $A$ is a constant that depends only on $f$
(actually, we only need $\norm{f-\tilde{f}}^*_{2,\max}$ to be small here).

Thus, whenever $\norm{f-\tilde{f}}^*_{3,\max}$ satisfies
\begin{equation}
\norm{f-\tilde{f}}^*_{3,\max} \leq \frac{\sigma}{3A},
\end{equation}
every $M$ has an unique corresponding point in $\tilde{M}$
and vice versa. In addition, for a pair of local modes
$(m_j, \tilde{m}_j): m_j\in M, \tilde{m}_j\in \tilde{M}$, 
their distance is bounded by
$\norm{m_j-\tilde{m}_j}\leq \frac{\sigma}{3}$.

Now we pick 
%\begin{equation}
$(\epsilon,\delta )$ such that they
%\end{equation}
satisfy equation \eqref{eq::de}.
Then when $\norm{f-\tilde{f}}^*_{3,\max}$ is sufficiently small, by Lemma \ref{lem::destination},
for every $x\in\mathcal{H}(\epsilon, \delta)$ we have
$$
\norm{\lim_{t\rightarrow\infty}\pi_x(t) - \lim_{t\rightarrow\infty}\tilde{\pi}_x(t)}\leq c_*\sqrt{\norm{f- \tilde{f}}_{\infty}}\leq c_*\sqrt{\norm{f-\tilde{f}}^*_{3,\max}} .
$$
Thus, whenever 
\begin{equation}
%\norm{f-\tilde{f}}^*_{3,\max}\leq  \frac{1}{c_*}\sqrt{\frac{\sigma}{3}},
\norm{f-\tilde{f}}^*_{3,\max}\leq  \frac{1}{c^2_*}\left(\frac{\sigma}{3}\right)^2,
\end{equation}
$\pi_x(t)$ and $\tilde{\pi}_x(t)$ leads to the same pair of modes.
Namely,
the boundaries $\tilde{B}$ will not intersect the region
$\mathcal{H}(\epsilon,\delta)$.
And it is obvious that $\tilde{B}$ cannot intersect $B(M,\sqrt{\epsilon})$.
To conclude, 
\begin{equation}
\begin{aligned}
\tilde{B}\cap \mathcal{H}(\epsilon,\delta) &= \emptyset\\
\tilde{B}\cap B(M,\sqrt{\epsilon})&= \emptyset\\
\Rightarrow\tilde{B}\cap \mathbb{K}(\delta)&= \emptyset,
\end{aligned}
\end{equation}
because by definition, $\mathbb{K}(\delta) = \mathcal{H}(\epsilon,\delta)\cap B(M,\sqrt{\epsilon})$.

Thus, $\tilde{B}\subset \mathbb{K}(\delta)^C = B\oplus \delta$,
which implies $\Haus(B,\tilde{B})\leq \delta< \frac{H_{\min}}{d^2\norm{f}_{3,\max}}$
(note that $\delta <\delta_0\leq\frac{H_{\min}}{d^2\norm{f}_{3,\max}}$ appears in equation
\eqref{eq::de} and Lemma~\ref{lem::grad}).

{\bf Part 2: Rate of convergence.}
To derive the convergence rate, we use proof by contradiction. 
Let 
$q \in B,\tilde{q}\in \tilde{B}$ a pair of points
such that their distance attains the Hausdorff distance $\Haus\left(\tilde{B},B\right)$. 
Namely, $q$ and $\tilde{q}$ satisfy
$$
\norm{q-\tilde{q}} = \Haus\left(\tilde{B},B\right)
$$
and
either $q$ is the projected point from $\tilde{q}$ onto $B$
or $\tilde{q}$ is the projected point from $q$ onto $\tilde{B}$.
%We will use proof by contradiction to bound $\Haus(\tilde{B},B)$.
%We begin with a study on the line segment connecting $q,\tilde{q}$
%and show some useful properties for all points on this line segment.

Recall that $\mathbb{V}(x)$ is the normal space to $B$ at $x\in B$ and we define $\tilde{\mathbb{V}}(x)$
similarly for $x\in \tilde{B}$.
An important property of the pair $q,\tilde{q}$ is that 
$q-\tilde{q}\in \mathbb{V}(q), \tilde{\mathbb{V}}(\tilde{q})$.
If this is not true, we can slightly perturb $q$ (or $\tilde{q}$) on $B$ (or $\tilde{B}$)
to get a projection distance larger than the Hausdorff distance, which leads to a contradiction.

Now we choose $x$ to be a point between $q,\tilde{q}$ such that
$ x = \frac{1}{3} q + \frac{2}{3 }\tilde{q}$.
We define $e(x) = \frac{q-x}{\norm{q-x}}$ and $\tilde{e}(x) = \frac{\tilde{q}-x}{\norm{\tilde{q}-x}}$.
Then $e(x)\in\mathbb{V}(q) $ and $\tilde{e}(x)\in \tilde{\mathbb{V}}(\tilde{q})$
and $e(x)= -\tilde{e}(x)$.

%Let $x$ be any point between $q,\tilde{q}$.
%i.e. $x= \alpha q +(1-\alpha) \tilde{q}$ for some $0< \alpha < 1$.
%We define $e(x) = \frac{q-x}{\norm{q-x}}$ and $\tilde{e}(x) = \frac{\tilde{q}-x}{\norm{\tilde{q}-x}}$.
%Then $e(x)\in\mathbb{V}(q) $ and $\tilde{e}(x)\in \tilde{\mathbb{V}}(\tilde{q})$
%and $e(x)= -\tilde{e}(x)$.

By Lemma \ref{lem::Gdist} (second assertion), 
\begin{equation}
\begin{aligned}
\ell(x) &= e(x)^Tg(x) \geq \frac{1}{2}H_{\min} \norm{q-x}>0\\
\tilde{\ell}(x) &= \tilde{e}(x)^T\tilde{g}(x) \geq \frac{1}{2}\tilde{H}_{\min} \norm{\tilde{q}-x}>0.
\end{aligned}
\end{equation}
Thus, for every $x$ between $q,\tilde{q}$, 
\begin{equation}
e(x)^Tg(x)>0, \quad, e(x)^T \tilde{g}(x)= -\tilde{e}(x)^T\tilde{g}(x) <0.
\label{eq::eg1}
\end{equation}
Note that we can apply Lemma \ref{lem::Gdist} to $\tilde{f}$ and its gradient
because when $\norm{f-\tilde{f}}^*_2$ is sufficiently small, 
the assumption {\bf(D)} holds for $\tilde{f}$ as well.

%Now we consider $x \rightarrow \tilde{q}$ and find an upper bound for 
%$\norm{q-\tilde{q}} = \Haus(\tilde{B},B)$.
To get the upper bound of $\norm{q-\tilde{q}} = \Haus(\tilde{B},B)$, 
note that 
$ \|q-x\| = \frac{2}{3}\|q-\tilde{q}\|$, so
\begin{equation}
\begin{aligned}
e(x)^T \tilde{g}(x) & = e(x)^T (\tilde{g}(x)-g(x)) + e(x)^Tg(x)\\
&\geq e(x)^T g(x) - \norm{\tilde{f}-f}_{1,\max}\\
&\geq \frac{1}{2}H_{\min} \norm{q-x} -  \norm{\tilde{f}-f}_{1,\max}\quad 
\mbox{(By Lemma \ref{lem::Gdist})}\\
& = \frac{1}{3}H_{\min} \norm{q-\tilde{q}} -  \norm{\tilde{f}-f}_{1,\max}.
\end{aligned}
\end{equation}
Thus, as long as 
$$
 \Haus(\tilde{B},B) = \norm{q-\tilde{q}}  > 3\frac{\norm{\tilde{f}-f}_{1,\max}}{H_{\min}},
$$
we have $e(x)^T \tilde{g}(x)>0$,
a contradiction to equation \eqref{eq::eg1}.
Hence, we conclude that 
$$ 
\Haus(\tilde{B},B) \leq  3\frac{\norm{\tilde{f}-f}_{1,\max}}{H_{\min}} =O \left(\norm{\tilde{f}-f}_{1,\max}\right).
$$

\end{proof}

\begin{proof}[ of Theorem~\ref{thm::number}]
%It is not so easy to directly prove the rate for the rand index.
%Alternatively, we can make use of the asymptotic 

%To prove Theorem~\ref{thm::number}, a key observation is that 
%the mis-clustered occurs whenever a point is near the cluster boundaries
%and the distance to boundaries must be less than $\Haus(\hat{B}_n,B)$.
To prove the asymptotic rate of the rand index,
we assume that for every local mode of $p$,
there exists one and only one local mode of $\hat{p}_n$
that is close to the specific mode of $p$.
By Lemma~\ref{lem::critical}, this is true when $\norm{\hat{p}_n-p}^*_{3,\max}$ is sufficiently small.
Thus, after relabeling, 
the local mode $\hat{m}_\ell$ of $\hat{p}_n$ 
is an estimator to the local mode $m_\ell$ of $p$.
Let $\hat{W}_\ell$ be the basin of attraction to $\hat{m}_\ell$ using
$\nabla \hat{p}_n$ and $W_\ell$ be the basin of attraction to $m_\ell$ using
$\nabla p$.
Let $A\triangle B =\{x: x\in A, x\notin B\}\cup \{x:x\in B, x\notin A\}$ be 
the symmetric difference between sets $A$ and $B$.
The regions
\begin{equation}
E_n = \bigcup_{\ell} \left(\hat{W}_\ell \triangle W_\ell\right)\subset\K
\end{equation}
are where the two mode clustering disagree with each other.
Note that $E_n$ are regions between the two boundaries $\hat{B}_n$ and $B$

Given a pair of points $X_i$ and $X_j$,
%the function $\Psi(X_i,X_j)$ disagree with $\hat{\Psi}_n(X_i,X_j)$
%if either $X_i$ or $X_j$ (or maybe both) are in $E_n$.
%That is, 
\begin{equation}
\Psi(X_i,X_j)\neq\hat{\Psi}_n(X_i,X_j) \Longrightarrow X_i \mbox{ or }X_j \in E_n.
\label{eq::rand::pf1}
\end{equation}
By the definition of rand index \eqref{eq::rand2},
\begin{equation}
1-\rand\left(\hat{p}_n,p\right) = {n \choose 2}^{-1} \sum_{i,j}1\left(\Psi(X_i,X_j)\neq\hat{\Psi}_n(X_i,X_j)\right)
    \label{eq::rand::pf2}
\end{equation}
Thus, if we can bound the ratio of data points within $E_n$,
we can bound the rate of rand index.

Since $\K$ is compact and $p$ has bounded second derivatives, 
the volume of $E_n$ is bounded by
\begin{equation}
\Vol(E_n) = O\left(\Haus(\hat{B}_n,B)\right).
\end{equation}
Note $\Vol(A)$ denotes the volume (Lebesgue measure) of a set $A$.
We now construct a region surrounding $B$ such that
\begin{equation}
E_n \subset B\oplus \Haus(\hat{B}_n,B) = V_n
\end{equation}
and 
\begin{equation}
\Vol(V_n) = O\left(\Haus(\hat{B}_n,B)\right).
\end{equation}

Now we consider a collection of subsets of $\K$:
\begin{equation}
\mathcal{V} = \{B\oplus r: R> r>0\},
\end{equation}
where $R<\infty$ is the diameter for $\K$.
For any set $A\subset \K$,
let $P(X_i\in A)$ and $\hat{P}_n(A)= \frac{1}{n}\sum_{i=1}^n 1(X_i\in A)$
denote the probability of an observation within $A$
and the empirical estimate for that probability, respectively.
It is easy to see that $V_n\in \mathcal{V}$ for all $n$
and the class $\mathcal{V}$ has a finite VC dimension (actually,
the VC dimension is $1$).
By the empirical process theory (or so-called VC theory, see e.g. \cite{vapnik1971uniform}), 
\begin{equation}
\underset{A\in \mathcal{V}}{\sup} \left|
P(X_i\in A)-\hat{P}_n(A)
\right| = O_\P\left(\sqrt{\frac{\log (n)}{n}}\right).
\end{equation}
%{\color{red}(WARNING: I am not sure if this is true; it is true if $A$ is rectangular set.)}
Thus, 
\begin{equation}
\left|
P(X_i\in V_n)-\hat{P}_n(V_n)
\right|  
%=\left|P(X_i\in E_n)-\frac{\mathcal{N}_n}{n}\right|  
= O_\P\left(\sqrt{\frac{\log (n)}{n}}\right).
\end{equation}

Now by equations \eqref{eq::rand::pf1} and \eqref{eq::rand::pf2},
\begin{equation}
1-\rand\left(\hat{p}_n,p\right) \leq 8\hat{P}_n(E_n) \leq 8\hat{P}_n (V_n) \leq 8P(X_i\in V_n) + O_\P\left(\sqrt{\frac{\log (n)}{n}}\right).
\end{equation}
%The above implies
Therefore,
\begin{equation}
\begin{aligned}
1-\rand\left(\hat{p}_n,p\right) &\leq P(X_i\in V_n) + O_\P\left(\sqrt{\frac{\log (n)}{n}}\right)\\
&\leq \sup_{x\in \K}p(x)\times\Vol(V_n) + O_\P\left(\sqrt{\frac{\log (n)}{n}}\right)\\
&\leq O\left(\Haus(\hat{B}_n,B)\right)+O_\P\left(\sqrt{\frac{\log (n)}{n}}\right)\\
& = O\left(h^2\right)+ O_\P\left(\sqrt{\frac{\log (n)}{nh^{d+2}}}\right),
\end{aligned}
\end{equation}
which completes the proof.
Note that we apply Theorem~\ref{thm::mode} in the last equality.

\end{proof}

\begin{proof}[ of Theorem~\ref{thm::MSR}]
Let $(X_1,Y_1),\cdots, (X_n, Y_n)$ be the observed data.
Let $\hat{E}_\ell$ denote the $d$-cell for the nonparametric pilot regression estimator 
$\hat{m}_n$. With $I_\ell = \{i: X_i\in \hat{E}_\ell\}$, we define $\mathbb{X}_\ell$
as the matrix with rows $X_i$, $i\in I_\ell$ and similarly we define $\mathbb{Y}_\ell$.

We define $\mathbb{X}_{0,\ell}$ to be the matrix similar to $\mathbb{X}_\ell$
except that the row elements are those $X_i$ within $E_\ell$, the $d$-cell
defined on true regression function $m$.
We also define $\mathbb{Y}_{0,\ell}$ to be the corresponding $Y_i$.

By the theory of linear regression,
the estimated parameters
$\hat{\mu}_\ell, \hat{\beta}_\ell$ have a closed form solution:
\begin{equation}
(\hat{\mu}_\ell, \hat{\beta}_\ell)^T = (\mathbb{X}_\ell^T \mathbb{X}_\ell)^{-1}\mathbb{X}_\ell^T \mathbb{Y}_\ell.
\label{eq::Regdef1}
\end{equation}
Similarly, we define
\begin{equation}
(\hat{\mu}_{0,\ell}, \hat{\beta}_{0,\ell})^T = (\mathbb{X}_{0,\ell}^T \mathbb{X}_{0,\ell})^{-1}\mathbb{X}_{0,\ell}^T \mathbb{Y}_{0,\ell}
\label{eq::Regdef2}
\end{equation}
as the estimated coefficients using $\mathbb{X}_{0,\ell}$ and $\mathbb{Y}_{0,\ell}$.

As $\norm{\tilde{m}-m}^*_{3,\max}$ is small,
by Theorem~\ref{thm::number},
the number of rows at which $\mathbb{X}_\ell$ and $\mathbb{X}_{0,\ell}$
differ is bounded by $O(n\times \norm{\hat{m}_n-m}_{1,\max})$.
This is because an observation (a row vector) that appears only in
one of $\mathbb{X}_\ell$ and $\mathbb{X}_{0,\ell}$ is those fallen 
within either $\hat{E}_\ell$ or $E_\ell$ but not both.
Despite the fact that Theorem \ref{thm::number} is for basins of attraction (d-descending manifolds) 
of local modes, it can be easily generalized to $0$-ascending manifolds of local minima under assumption (A).
Thus, the similar bound holds for d-cells as well.
%Theorem~\ref{thm::number} bounds the ratio of these cases 
Thus, we conclude that
\begin{equation}
\begin{aligned}
\left\|\frac{1}{n}\mathbb{X}_{\ell}^T \mathbb{X}_{\ell}-
\frac{1}{n}\mathbb{X}_{0,\ell}^T \mathbb{X}_{0,\ell}\right\|_{\infty} 
&= O(\norm{\hat{m}_n-m}_{1,\max})\\
\left\|\frac{1}{n}\mathbb{X}_{\ell}^T \mathbb{Y}_{\ell}-
\frac{1}{n}\mathbb{X}_{0,\ell}^T \mathbb{Y}_{0,\ell}\right\|_{\infty} 
&= O( \norm{\hat{m}_n-m}_{1,\max})
\end{aligned}
\label{eq::reg}
\end{equation}
since $(\mathbb{X}_{\ell}, \mathbb{Y}_{\ell})$ 
and $(\mathbb{X}_{0,\ell}, \mathbb{Y}_{0,\ell})$ only differ by 
$O(n\times \norm{\hat{m}_n-m}_{1,\max})$ elements.
Thus,
\begin{equation}
\begin{aligned}
\left\|(\hat{\mu}_{0,\ell}-\hat{\mu}_\ell, \hat{\beta}_{0,\ell}-\hat{\beta}_\ell)\right\|_{\infty}
&= \left\|
\left(\frac{1}{n}\mathbb{X}_{0,\ell}^T \mathbb{X}_{0,\ell}\right)^{-1}\frac{1}{n}\mathbb{X}_{0,\ell}^T \mathbb{Y}_{0,\ell} -
\left(\frac{1}{n}\mathbb{X}_\ell^T \mathbb{X}_\ell\right)^{-1}\frac{1}{n}\mathbb{X}_\ell^T \mathbb{Y}_\ell
\right\|_{\infty}\\
& = O( \norm{\hat{m}_n-m}_{1,\max}),
\end{aligned}
\end{equation}
which implies.
\begin{equation}
\max\left\{\norm{\hat{\mu}_{0,\ell}-\hat{\mu}_{\ell}}, \norm{\hat{\beta}_{0,\ell}-\hat{\beta}_{\ell}}\right\}
= O( \norm{\hat{m}_n-m}_{1,\max}).
\label{eq::reg1}
\end{equation}

Now by the theory of linear regression, 
\begin{equation}
\max\left\{\norm{\hat{\mu}_{0,\ell}-\mu_\ell}, \norm{\hat{\beta}_{0,\ell}-\beta_\ell}\right\}
= O_{\P}\left( \frac{1}{\sqrt{n}}\right).
\label{eq::reg2}
\end{equation}
Thus, combining \eqref{eq::reg1} and \eqref{eq::reg2}
and use the fact that all the above bounds are uniform over 
each cell, we have proved that the parameters converge
at rate $O( \norm{\hat{m}_n-m}_{1,\max}) + O_{\P}\left( \frac{1}{\sqrt{n}}\right)$.

%The last part is to show the regions that parameters estimation
%can be transformed into functional estimation; this proof is
%similar to part 2 of the proof to Theorem~\ref{thm::MSS}.
For points within the regions where $E_\ell$ and $\hat{E}_\ell$ agree with each other,
the rate of convergence for parameter estimation translates into
the rate of $\hat{m}_{n, \MSR}- m_{\MSR}$.
The regions that $E_\ell$ and $\hat{E}_\ell$ disagree to each other, denoted as $\mathbb{N}_n$,
have Lebesgue $O(\norm{\hat{m}_n-m}_{1,\max})$ by Theorem~\ref{thm::Haus}.
Thus, we have completed the proof.
%Thus, we have completed the proof of the first assertion (equation \eqref{eq::thm::MSR1}).

%for the first assertion (equation \eqref{eq::thm::MSR1}).

%For the second assertion,
%by theory of nonparametric regression,
%the kernel regression under assumption (K1--2)
%yields the rate
%\begin{equation}
%\norm{\hat{m}_n-m_h}^*_{1,\max} = O_\P\left(\sqrt{\frac{\log (n)}{nh^{d+2}}}\right).
%\end{equation}
%Thus, when $h$ is fixed, the above rate is $O_\P\left(\sqrt{\frac{\log (n)}{n}}\right)$.
%Use this fact and the result from first assertion proves the second assertion
%(equation \eqref{eq::thm::MSR2}).

\end{proof}

\begin{proof}[ of Theorem~\ref{thm::MSR2}]
The proof of Theorem~\ref{thm::MSR2} is nearly identical to the proof of Theorem~\ref{thm::MSR}.
The only difference is that 
the number of rows that $\mathbb{X}_\ell$ and $\mathbb{X}_{0,\ell}$
differ is bounded by $O(n\times \norm{\hat{m}_n-m}_{1,\max}^\beta)$
due to the low noise condition \eqref{eq::low_noise}.
Thus, equation \eqref{eq::reg} becomes
\begin{equation}
\begin{aligned}
\left\|\frac{1}{n}\mathbb{X}_{\ell}^T \mathbb{X}_{\ell}-
\frac{1}{n}\mathbb{X}_{0,\ell}^T \mathbb{X}_{0,\ell}\right\|_{\infty} 
&= O(\norm{\hat{m}_n-m}_{1,\max}^\beta)\\
\left\|\frac{1}{n}\mathbb{X}_{\ell}^T \mathbb{Y}_{\ell}-
\frac{1}{n}\mathbb{X}_{0,\ell}^T \mathbb{Y}_{0,\ell}\right\|_{\infty} 
&= O( \norm{\hat{m}_n-m}_{1,\max}^\beta)
\end{aligned}
\label{eq::reg3}
\end{equation}
so the parameter estimation error \eqref{eq::reg2} is
$O( \norm{\hat{m}_n-m}_{1,\max}^\beta) + O_{\P}\left( \frac{1}{\sqrt{n}}\right)$.

Under assumption (K1--2) and using Theorem~\ref{thm::KDE} (the same result works for kernel regression),
$$O( \norm{\hat{m}_n-m}_{1,\max}) = O(h^2)+O_{\P}\left(\sqrt{\frac{\log n}{nh^{d+2}}}\right).$$
Thus, with the choice that $h = O\left(\left(\frac{\log n}{n}\right)^{1/(d+6)}\right)$,
we have $O( \norm{\hat{m}_n-m}_{1,\max}) = O_\P\left(\left(\frac{\log n}{n}\right)^{2/(d+6)}\right)$,
which proves equation \eqref{eq::thm::MSR2}.

\end{proof}

\begin{proof}[ of Theorem~\ref{thm::MSS}]
%We first prove that the `parameters' for each Morse-Smale cell
%are consistently estimated
%and then extend this to prove the desire result.

%{\bf Part 1: Parameter consistency.}
We first derive the explicit form of the parameters
$(\eta^\dagger_\ell, \gamma^\dagger_\ell)$ within cell $E_\ell$.
Note that the parameters are obtained by \eqref{eq::MSS1}:
$$
(\eta_\ell^\dagger, \gamma_\ell^\dagger) = \underset{\eta, \gamma}{\sf argmin}\,\, \int_{E_\ell} 
\left(f(x)-\eta-\gamma^Tx\right)^2 dx.
$$
Now we define a random variable $U_\ell\in\R^d$ that is uniformly distributed
over $E_\ell$.
Then equation \eqref{eq::MSS1} is equivalent to
\begin{equation}
(\eta_\ell^\dagger, \gamma_\ell^\dagger) = \underset{\eta, \gamma}{\sf argmin}\,\, \mathbb{E}\left(
\left(f(U_\ell)-\eta-\gamma^TU_\ell\right)^2 \right).
\end{equation}
The analytical solution to the above problem is
\begin{equation}
\left( \begin{array}{c}
\eta^\dagger_\ell \\
\gamma^\dagger_\ell \end{array} \right)
= 
\left( \begin{array}{cc}
1& \mathbb{E}(U_\ell)^T \\
\mathbb{E}(U_\ell) & \mathbb{E}(U_\ell U_\ell^T)\end{array} \right)^{-1}
\left( \begin{array}{c}
\mathbb{E}(f(U_\ell)) \\
\mathbb{E}(U_\ell f(U_\ell))\end{array} \right)
\label{eq::pf::MSS1}
\end{equation}

Now we consider another smooth function $\tilde{f}$ that is close to $f$ such that
$\norm{\tilde{f}-f}^*_{3,\max}$ is small so
we can apply Theorem~\ref{thm::Haus}
to obtain consistency for both descending $d$-manifolds and ascending $0$-manifolds.
Note that by Lemma~\ref{lem::critical},
all the critical points are close to each other
and after relabeling, each $d$-cell $E_\ell$ of $f$
is estimated by another $d$-cell $\tilde{E}_\ell$ of $\tilde{f}$.
Theorem \ref{thm::Haus} further implies
that 
\begin{equation}
\begin{aligned}
\left|\Leb (\tilde{E}_\ell) - \Leb(E_\ell)\right| &= O\left(\norm{\tilde{f}-f}_{1,\max}\right)\\
\Leb\left(\tilde{E}_\ell\triangle E_\ell\right)&= O\left(\norm{\tilde{f}-f}_{1,\max}\right),
\end{aligned}
\label{eq::pf::MSS2}
\end{equation}
where $\Leb(A)$ is the Lebesgue measure for set $A$
and $A\triangle B = (A\backslash B)\cup (B\backslash A)$ is the symmetric difference.
By simple algebra, equation \eqref{eq::pf::MSS2} implies that 
\begin{equation}
\begin{aligned}
\|\mathbb{E}(\tilde{U}_\ell) &-\mathbb{E}(U_\ell)\|_{\infty} = O\left(\norm{\tilde{f}-f}_{1,\max}\right)\\
\|\mathbb{E}(\tilde{U}_\ell \tilde{U}_\ell^T) &- \mathbb{E}(U_\ell U_\ell^T)\|_{\infty} = O\left(\norm{\tilde{f}-f}_{1,\max}\right)\\
|\E (\tilde{f}(\tilde{U}_\ell))&- \E (f(U_\ell))|  = O\left(\norm{\tilde{f}-f}^*_{1,\max}\right)\\
\|\E (\tilde{U}_\ell\tilde{f}(\tilde{U}_\ell))&- \E (U_\ell f(U_\ell))\|_{\infty}  = O\left(\norm{\tilde{f}-f}^*_{1,\max}\right).
\end{aligned}
\label{eq::pf::MSS3}
\end{equation}

By \eqref{eq::pf::MSS3} and the analytic solution to 
$(\tilde{\eta}^\dagger_\ell, \tilde{\gamma}^\dagger_\ell)$ from \eqref{eq::pf::MSS1},
we have proved
\begin{equation}
\left\|
\left( \begin{array}{c}
\tilde{\eta}^\dagger_\ell \\
\tilde{\gamma}^\dagger_\ell \end{array} \right)
-
\left( \begin{array}{c}
\eta^\dagger_\ell \\
\gamma^\dagger_\ell \end{array} \right)
 \right\|_{\infty}
 = O\left(\norm{\tilde{f}-f}^*_{1,\max}\right).
 \label{eq::pf::MSS4}
\end{equation}
Since the bound does not depend on the cell indices $\ell$,
\eqref{eq::pf::MSS4} holds uniformly for all $\ell= 1,\cdots, K$.

%%%%
\begin{comment}
%%%%
{\bf Part 2: Extend to the majority region.}
We splits the support $\K$ into two parts,
the first part is where $E_\ell$ and $\tilde{E}_\ell$
agrees with each other
while the second part is the remaining regions.
Let $\G = \bigcup_{\ell} (E_\ell \cap \tilde{E}_\ell)$ be the set
where they agree with each other.
Note that the regions not in $\G$, denoted as $\mathbb{N}_n$, have Lebesgue measure
\begin{equation}
\Leb(\K\backslash\G) = \Leb\left( \bigcup_{\ell} (E_\ell \triangle \tilde{E}_\ell)\right) = 
O\left(\norm{\tilde{f}-f}_{1,\max}\right).
\label{eq::pf::MSS5}
\end{equation}

By the result of part 1, uniformly for all $x\in\G$
we have 
\begin{equation}
|f_{\MS}(x)- \tilde{f}_{\MS}(x)| = O\left(\norm{\tilde{f}-f}^*_{1,\max}\right).
\label{eq::pf::MSS6}
\end{equation}
Thus, putting \eqref{eq::pf::MSS5} and \eqref{eq::pf::MSS6} together,
we have proved the desire result.
%%%%
\end{comment}
%%%%

%Thus, equation \eqref{eq::pf::MSS5} implies
%that the set we cannot apply \eqref{eq::pf::MSS6}
%has Lebesgue measure

\end{proof}

\bibliographystyle{abbrvnat}
\bibliography{Morse.bib}

\end{document}